\documentclass[reqno]{amsart}
\usepackage[a4paper,hmargin=2cm,vmargin=3cm]{geometry}%
\usepackage[T1]{fontenc}%
\usepackage[utf8]{inputenc}%
\usepackage[foot]{amsaddr}

\usepackage{amsmath,mathtools}%
\usepackage{amsfonts}%
\usepackage{amssymb}%
\usepackage{amsthm}%
\usepackage{mathrsfs}%
\usepackage{bm}%
\usepackage{bbm}%
\usepackage[english]{babel}%

\usepackage{yfonts}
\usepackage{bmpsize-base}

\usepackage{setspace}
\usepackage{booktabs}

\usepackage{eso-pic}
\usepackage{shapepar}

\usepackage{enumerate}
\usepackage{enumitem}
\usepackage{braket}
\usepackage{graphicx}
\usepackage{epstopdf}
\usepackage{caption}
\usepackage{float}
\usepackage{subfigure}

\usepackage{placeins}

\makeatletter

\usepackage{listings} 
\usepackage{color} 
\usepackage{hyperref}
\usepackage{soul}
\usepackage{color}
\usepackage{scalerel}
\usepackage{framed,lipsum}
\usepackage{lineno}
\numberwithin{equation}{section}
\newtheorem{theorem}{Theorem}[section]
\newtheorem{corollary}[theorem]{Corollary}
\newtheorem{lemma}[theorem]{Lemma}
\newtheorem{proposition}[theorem]{Proposition}
\theoremstyle{remark}
\newtheorem{remark}[theorem]{Remark}
\theoremstyle{definition}

\definecolor{gray}{rgb}{92.,92.,92.}
\newcommand {\Cdot}{\vcenter{\hbox{\scalebox{1.6}{.}}}}

\newcommand {\N}{\mathbb{N}}
\newcommand {\E}{\mathbb{E}}
\newcommand {\R}{\mathbb{R}}
\newcommand {\s}{\sigma}
\newcommand {\W}{\mathbb{W}}
\newcommand {\BSDE}{BSDE~\eqref{selfEq1Aux}}

\renewcommand {\b}{\beta}
\renewcommand {\(}{\left (}
\renewcommand {\)}{\right )}
\renewcommand {\[}{\left [}
\renewcommand {\]}{\right ]}
\newcommand {\Ee}[1]{\widetilde{\mathbb{E}}_{#1}}
\newcommand {\Econd}[2]{\mathbb{E}\[#1\middle|\mathcal{F}^{\otimes n}_{#2}\]} 
\newcommand {\Eecond}[3]{\widetilde{\mathbb{E}}_{#1}\[#2\middle|\mathcal{F}^{\otimes n}_{#3}\]} 
\newcommand {\Fq}[1]{\mathcal{F}^{\otimes n}_{#1}} 
\newcommand {\Fqfilt}[2]{\(\mathcal{F}^{\otimes n}_{#1}\)_{#1\in #2}} 
\newcommand {\Wn}{\mathbb{W}^{\otimes n}}
\newcommand {\Wnas}{\mathbb{W}^{\otimes n}\textup{--a.s.}}

\def\been#1{\begin{linenomath}\begin{equation}#1\end{equation}\end{linenomath}}
\def\aled#1{\begin{aligned}#1\end{aligned}}
\DeclareMathOperator\erf{erf}

\title[Prop. of the full-RSB free energy func. of the Ising spin glass on RRG]{Properties of the full replica symmetry breaking free energy functional of the Ising spin glass on random regular graph}

\author{Francesco Concetti}
\address{Faculty of Mathematics and Computer Science, UniDistance Suisse, 3900 Brig, Switzerland}
\email{francesco.concetti@unidistance.ch}

\thanks{The author was supported by SNSF grants 176918 and 206148. }

\begin{document}
\begin{abstract}
We analyze the full replica symmetry breaking (full--RSB) free energy functional for the Ising spin glass on a random regular graph proposed by the author in \cite{MyPaper}. We prove that the full--RSB formulation provides an improvement over any replica symmetry breaking approximation with a finite number of steps (finite--RSB), based on the Mézard-Parisi ansatz \cite{ParMezRRG1}. We provide a representation of that functional as the unique solution to a well-posed backward stochastic differential equation. This stochastic formulation enables a refined analysis of the functional and the computation of the derivatives with respect to the order parameters of the model. The techniques developed here hold potential interest for broader areas such as calculus of variations, stochastic optimal control, and functional analysis.
\end{abstract}
\maketitle
\section{Introduction}
\label{sec:intro}
The solution to the Sherrington--Kirkpatrick model (SK)\cite{SK1,SK2}, obtained by Parisi \cite{Par1_0,Par1_1,VPM}, represents one of the most remarkable milestones in the study of disordered and glassy systems.

In his pioneering work \cite{SK2}, Parisi introduced a hierarchical scheme, known as the \emph{Replica Symmetry Breaking (RSB) ansatz}. This ansatz yields a sequence of upper bounds for the free energy of the SK model, where each bound corresponds to a fixed number of RSB \emph{steps}. Increasing the number of steps up to infinity yields a non-increasing sequence of upper bounds. The case with infinitely many steps is known as the \emph{full-RSB formula} \cite[Equation (19)]{SK2}.

Talagrand rigorously proved that the full-RSB formula yields the exact free energy in the SK model \cite{Tala}. This result was later extended by Panchenko to a broader class of fully connected spin glass models \cite{Panchenko_p_spin}. For a rigorous and detailed treatment of this topic, we refer the reader to the books by Talagrand \cite{Talbook2} and Panchenko \cite{Panchbook}.

In recent decades, significant efforts have been devoted to extending the results obtained for fully connected models to spin glasses defined on \emph{sparse} graphs \cite{Viana_bray,Mottishaw}, henceforth referred to as \emph{sparse models}.

In this manuscript, we consider the following model. Given $N\in \N$ and $c\in\N$, we consider a random regular graph (RRG) $\mathcal{G}=\([N],\mathcal{E}\)$ with connectivity $c$ and size $N$, where $[N]=\{1,\cdots ,N\}$ denotes the set of vertices and $\mathcal{E}\subset [N]^2$ denotes the set of edges (see \cite{Bollobas} for a detailed description of RRG). The $2$--spin model on this graph is defined by the random Hamiltonian
\begin{equation}
\label{eq:model}
H(\bm{\sigma}):=-\sum_{(i,j)\in \mathcal{E}}J_{ij}\s_i\s_j,
\end{equation}
where $\bm{\s}:=(\s_i)_{i\in [N]}\in \{-1,1\}^N$ represents the \emph{spins}, and the \emph{couplings} $J_{ij}$, for $(i,j)\in \mathcal{E}$ are independent Rademacher random variables. Given $\b>0$ (the so-called \emph{inverse temperature}), the \emph{quenched free energy} of the model is defined as
\been{
\label{eq:free} 
f:=\lim_{N\to \infty}\frac{1}{N}\E_{\mathcal{G},(J_{ij})}\[\log\(\sum_{\bm{\s}\in \{-1,1\}^N}e^{-\b H(\bm{\sigma})}\)\].
}
Parisi and Mézard adapted the replica symmetry breaking ansatz with a \emph{finite} number of steps (\emph{finite--RSB}), originally developed for the fully connected models, to sparse models, using the so-called \emph{cavity method}. In their seminal paper \cite{ParMezRRG1}, they introduced the ansatz and derived an initial version of the scheme, known as \emph{1--RSB} (see also \cite{Panchenko_2015,Tala_panch} and \cite{MyPaper} for a detailed description of the general M\'ezard-Parisi  ansatz).

Later, Panchenko and Talagrand extended this framework to encompass the more general finite-RSB formulation for spin glasses on Erdős-Rényi graphs \cite{Erdos}, as presented in \cite{Tala_panch}. This approach was further generalized to a broader class of random graphs by Lelarge and Oulamara in \cite{lelarge}.

Most recently, the author derived the full--RSB formula for the 2-spin Ising model on random regular graphs in \cite{MyPaper}. This result adapts the completed Parisi ansatz, originally formulated for the fully connected model, to the context of sparse graphs.

To present the main result of \cite{MyPaper} and introduce the contributions of the present manuscript, we first recall the original $K$--RSB formula (finite--RSB with $K$ steps) as it is presented in \cite{Tala_panch,lelarge}.

 Let $K \geq 1$ be an integer. Define $\mathcal{M}^{(K+1)}$ as the set $[-1,1]$ equipped with the standard topology on $\mathbb{R}$ and the associated Borel $\sigma$--algebra. Recursively, for $0 \leq l \leq K$, define $\mathcal{M}^{(l)}$ as the space of probability measures on $\mathcal{M}^{(l+1)}$, equipped with the weak topology of measures and the associated Borel $\sigma$--algebra. 

Given $\zeta \in \mathcal{M}^{(0)}$, we construct a random sequence
\begin{equation}
W:=(M^{(0)}, \cdots, M^{(K+1)}),
\label{eq:random-sequence}
\end{equation}
where $M^{(0)}=\zeta$, and for each $1 \leq l \leq K+1$, $M^{(l)}$ is a random variable taking values in $\mathcal{M}^{(l)}$, with conditional distribution conditionally on $(M^{(0)}, \cdots, M^{(l-1)})$ equal to $M^{(l-1)}$. Let $\mathbb{M}_{\zeta}$ be the distribution of $W$ given  $\zeta$. 

For each $0 \leq l \leq K+1$, let $\mathcal{A}_l$ be the $\sigma$-algebra generated by the variables $(M^{(0)}, \cdots, M^{(l)})$:
\been{
\label{eq:sigmaA}
\mathcal{A}_l:=\sigma\(M^{(0)}, \cdots, M^{(l)}\)
}
Given an integer $n\in \N$, let $\bm{W} := (W_1, \cdots, W_{n})$ be a sequence of $n$ independent copies of the random sequence $W$, each distributed according to $\mathbb{M}_{\zeta}$. We denote by $\mathcal{A}^{\otimes n}_l$ the $n$-fold \emph{product $\sigma$-algebra} generated by $\mathcal{A}_l$ on each coordinate; that is,
\been{
\label{eq:delAl}
\mathcal{A}^{\otimes n}_l := \bigotimes_{i=1}^{n} \mathcal{A}_l.
}
Given a real-valued, and $\mathcal{A}^{\otimes n}_{K+1}$--measurable random variable $\Psi$, and an increasing sequence
\been{
\label{eq:x_seq0}
x:=(x_0,x_{1},\cdots,x_K,x_{K+1}),\qquad  0=x_0< x_1 \leq \cdots \leq x_K \leq x_{K+1}=1,
}
we define recursively
\begin{equation}
\label{eq:recursive-T0}
\widehat{\Xi}_{l}(\Psi, x, \bm{W}) :=
\begin{dcases}
   e^{\Psi(\bm{W})}, \quad & \text{if } l = K+1; \\
   \E\left[ \left. \widehat{\Xi}_{l+1}(\Psi, x,\bm{W})^{\frac{x_{l}}{x_{l+1}}} \right| \mathcal{A}^{\otimes n}_{l} \right] , & \text{if } 1\leq l \leq K.
\end{dcases}
\end{equation}
Finally, define
\been{
\label{eq:phihat}
\widehat{\Phi}(\Psi, x):=\E\[\frac{1}{x_1}\log(\widehat{\Xi}_{1}(\Psi, x, \bm{W}))\].
}
For the model~\eqref{eq:model} on a RRG with connectivity $c$, we set $n := 2c$.

Let $\psi^{(e)}_{\bm{J}}:[-1,1]^{2c}\to \R$ and $\psi^{(v)}_{\bm{J}}:[-1,1]^{2c}\to \R$ be the function defined as in the following:
\been{
\label{eq:psie}
\psi^{(e)}_{\bm{J}}(m_1,\cdots, m_{2c}):=
\log\(\sum_{\bm{\sigma}\in \{-1,1\}^{2c}}e^{\b\sum^c_{i=1}J_i\s_i\s_{i+c}}\prod^{2c}_{i=1}(1+m_i\s_i)\),
}
\been{
\label{eq:psiv}
\psi^{(v)}_{\bm{J}}(m_1,\cdots, m_{2c}):=\log\(\sum_{\bm{\sigma}\in \{-1,1\}^{2c}}\(\sum_{\bm{\tau}\in \{-1,1\}^2}e^{\b\sum^c_{i=1}J_i\(\tau_1\s_i+\tau_2\s_{i+c}\)}\)\prod^{2c}_{i=1}(1+m_i\s_i)\).
}
Here, $\bm{J}=(J_1,\cdots J_{2c})$ is a sequence of $2c$ independent  Rademacher random variables independent of $\bm{W}$, encoding the dependence on the random couplings of the Hamiltonian. 

Denoting by $(M^{(K+1)}_1,\cdots,M^{(K+1)}_{2c})$ the sequence composed of the last components of each random sequence in $\bm{W}=(W_1,\cdots W_{2c})$, we define
\been{
\label{eq:psievmr}
\psi^{(*)}_{\bm{J}}\circ \zeta^{\times 2c}(\bm{W}):=\psi^{(*)}_{\bm{J}}(M^{(K+1)}_1,\cdots,M^{(K+1)}_{2c}),
}
where the superscript $(*)$ stands for either $(e)$ or $(v)$ and the dependence on $\zeta$ is encoded in the distribution of $M^{(K+1)}_1,\cdots,M^{(K+1)}_{2c}$.

The $K$--RSB functional is defined as:
\begin{equation}
\label{eq:main-identity}
\widehat{\mathcal{P}}_K(\zeta,x):=\E_{\bm{J}}\[\widehat{\Phi} (\psi^{(v)}_{\bm{J}}\circ \zeta^{\times 2c},x)-\widehat{\Phi}(\psi^{(e)}_{\bm{J}}\circ \zeta^{\times 2c},x)\],
\end{equation}
where the functional $\widehat{\Phi} (\psi^{(v)}_{\bm{J}}\circ \zeta^{\times 2c},x)$ and $\widehat{\Phi} (\psi^{(e)}_{\bm{J}}\circ \zeta^{\times 2c},x)$ are evaluated for a fixed realization of $\bm{J}$, and $\E_{\bm{J}}$ denotes the expectation over $\bm{J}$, taken afterward.

This expression provides a variational formula for computing the quenched free energy~\eqref{eq:free} under the assumption that the spin distribution follows the M\'ezard-Parisi $K$--RSB ansatz introduced in \cite{ParMezRRG1}. The $K$--RSB functional has been rigorously proved to give an upper bound of the free energy (see \cite{Franz} and \cite{lelarge}).
\begin{theorem}[Franz Leone upper bound]
For any $K\in \N$
\been{
\label{eq:RSB_UB}
f\leq \inf_{\substack{\zeta \in \mathcal{M}^{(0)}\\
0< x_{1}\leq \cdots\leq x_{K+1}=1}}\widehat{\mathcal{P}}_K(\zeta,x).
}
\end{theorem}
\begin{remark}
\label{rem:topo}
Under the weak topology, the space of Borel probability measures on a Polish space is itself Polish. Since $\mathcal{M}^{(K+1)}$ is a Polish space under the standard topology, it follows by construction that, for each $0\leq l\leq K+1$, the space $\mathcal{M}^{(l)}$, when equipped with its weak topology and corresponding Borel $\sigma$--algebra, is a standard Borel space.

From the definitions~\eqref{eq:psie},~\eqref{eq:psiv}, and~\eqref{eq:psievmr}, we observe that the mappings $\psi^{(e)}_{\bm{J}}\circ \zeta^{\times 2c}[\bm{W}]$ and $\psi^{(v)}_{\bm{J}}\circ \zeta^{\times 2c}[\bm{W}]$ depend only on the tuple $(M^{(K+1)}_1,\cdots,M^{(K+1)}_{2c})$, and they are both bounded and continuous with respect to the standard topology on $[-1,1]$. Consequently, in the recursion~\eqref{eq:recursive-T0}, $\widehat{\Xi}_{K+1}(\psi^{(v)}_{\bm{J}}\circ \zeta^{\times 2c}, x, \bm{W})$ is bounded, Borel measurable, and continuous on $\mathcal{M}^{(K+1)}$. By definition, for any Borel subset of $\mathcal{M}^{(l+1)}$
\been{
M^{(l)}(A)=\E\[\bm{1}_{\{M^{(l+1)}\in A\}}\middle|\mathcal{A}_l\],
}
where $\bm{1}$ is the indicator function. Thus, if $\widehat{\Xi}_{l+1}(\psi^{(*)}_{\bm{J}}\circ \zeta^{\times 2c}, x, \bm{W})$ is bounded, Borel measurable, and continuous in the weak topology of $\mathcal{M}^{(l+1)}$, then  $\widehat{\Xi}_{l}(\psi^{(*)}_{\bm{J}}\circ \zeta^{\times 2c}, x, \bm{W})$ is bounded, Borel measurable, and continuous in the weak topology of $\mathcal{M}^{(l)}$.

Thus, our choice of $\sigma$--algebra suffices for the measurability and continuity requirements, and no further enlargement of the measurable structure is necessary.
\end{remark}
Whether there exist some parameters $K$, $x$, and $\zeta$ for which~\eqref{eq:RSB_UB} holds with equality remains an open question, despite several significant progresses made in this direction (see \cite{AustinPanchenko, PanchenkoExchange}). However, as for the Sherrington-Kirkpatrick (SK) model, it is widely conjectured that if the Parisi ansatz holds for the model~\eqref{eq:model}, then the equality in~\eqref{eq:RSB_UB} is achieved in the limit of infinitely many steps of replica symmetry breaking.

As pointed out in \cite{Panchenko_2015}, the finite–RSB scheme described above is not well suited for deriving a general full–RSB formulation. The order parameter $\zeta$ in this context takes the form of a hierarchical structure—a tower of distributions of distributions—where each additional step of symmetry breaking introduces another layer of complexity. As a result, each finite $K$--RSB scheme resides in its own distinct space of order parameters, preventing the construction of a unified and compact formulation that captures all levels of symmetry breaking simultaneously. Furthermore, as $K$ increases, the structure of the order parameter becomes increasingly intricate, making the limit $K\to \infty$ analytically challenging, due to the emergence of an infinite hierarchy of distributions.

The extension to the general full--RSB formula was derived in \cite{MyPaper}. The key idea in that work is to represent the random sequence~\eqref{eq:random-sequence} as a functional of Brownian motion. This approach is inspired by the representation of hierarchically exchangeable arrays of random variables developed in \cite{AustinPanchenko}, and first applied in the context of spin glasses in \cite{PanchenkoExchange}.

To state the full--RSB formula for the model~\eqref{eq:model} and present the main results of this manuscript, we first introduce the necessary probabilistic framework

Let $(C([0,1], \mathbb{R}), \mathcal{F}, \mathbb{W},(\mathcal{F}_q)_{q \in [0,1]})$ denote the classical Wiener space, where $C([0,1], \mathbb{R})$ is the space of real-valued continuous functions on $[0,1]$, $\mathcal{F}$ is the Borel $\sigma$-algebra, and $\mathbb{W}$ is the Wiener measure under which the canonical process $\{\omega(q)\}_{q \in [0,1]}$ is a Brownian motion with $\omega(0) \sim \mathcal{N}(0,1)$. The filtration $(\mathcal{F}_q)_{q \in [0,1]}$ is the usual augmentation of the natural filtration generated by $\omega$.

Given $n \in \mathbb{N}$, we also consider the probability space $(C([0,1], \mathbb{R}^n), \mathcal{F}^{\otimes n}, \mathbb{W}^{\otimes n}, (\mathcal{F}_q^{\otimes n})_{q \in [0,1]})$, where elements of $C([0,1], \mathbb{R}^n)$ are denoted by $\bm{\omega} = (\omega_1, \dots, \omega_n)$. Under $\mathbb{W}^{\otimes n}$, the coordinate process $\bm{\omega}:=(\omega_1\,,\cdots ,\omega_n)$ is an $\mathbb{R}^n$-valued Brownian motion with independent components, each distributed as $\mathbb{W}$. For each $q \in [0,1]$, the $\sigma$-algebra $\mathcal{F}_q^{\otimes n}$ is the $n$-fold product of $\mathcal{F}_q$.

Let $B_n$ be the space of bounded $\mathcal{F}^{\otimes n}_1$--measurable Wiener functional $\Psi:C([0,1],\R^n)\to \R$, and $D_n$ be the space of processes $\bm{r}:[0,1)\times C([0,1],\R^n)\to \R^{n}$ that are progressively measurable with respect the filtration $\(\mathcal{F}^{\otimes n}_q\)_{q\in[0,1)}$ and
\been{
\W\(\int^1_{0}dq \|\bm{r}(q,\bm{\omega}\|^2_2<\infty\)=1.
}
Given a probability measure $\mu$ on $[0,1]$ and $\bm{r} \in D_n$, we define the Doléans-Dade exponential (DDE) by
\begin{equation}
\label{DDE_xr}
\mathcal{E}(\mu \bm{r}, q, \bm{\omega}) :=
\exp\left(
\int_0^{q} \mu([0,t])\, \bm{r}(t, \bm{\omega}) \cdot d\bm{\omega}(t)
- \frac{1}{2} \int_0^{q} \mu^2([0,t]) \| \bm{r}(t, \bm{\omega}) \|_2^2 \, dt
\right),
\end{equation}
where the integral is understood in the Itô sense. 

For $q \in [0,1]$ and $\bm{r} \in D_n$, define the process  $\bm{r}_q \in D_n$
\been{
\label{eq:rq}
\bm{r}_q(q', \bm{\omega}) := 
\begin{dcases}
\bm{r}(q', \bm{\omega}), & \text{if } q' > q;\\
0, & \text{otherwise}.
\end{dcases}
}

Given $\Psi \in B_n$, define the functional
\begin{equation}
\label{eq:auxiliary_func0}
\aled{
&\Gamma(\Psi, \mu, \bm{r}, q, \bm{\omega}) \\
&:=
\mathbb{E}\left[
\mathcal{E}(\mu \bm{r}_q, 1, \bm{\omega}) \Psi(\bm{\omega}) \mid \mathcal{F}_q^{\otimes n}
\right]
- \frac{1}{2} \mathbb{E}\left[
\int_q^1 \mu([0,t]) \mathcal{E}(\mu \bm{r}_q, t, \bm{\omega}) \|\bm{r}(t, \bm{\omega})\|_2^2 dt
\mid \mathcal{F}_q^{\otimes n}
\right].
}
\end{equation}
Finally, set
\begin{align}
\label{eq:auxiliary_func}
\Phi(\Psi, \mu, 0) &:= \mathbb{E}\left[
\sup_{\bm{r} \in D_n} \Gamma(\Psi, \mu, \bm{r}, 0, \bm{\omega})
\right], \\
\label{eq:auxiliary_funccond}
\Phi(\Psi, \mu, q, \bm{\omega}) &:= \sup_{\bm{r} \in D_n} \Gamma(\Psi, \mu, \bm{r}, q, \bm{\omega}).
\end{align}
The above two quantities are non-linear expectations of the random variable $\Psi$. We refer to the quantity define in~\eqref{eq:auxiliary_func} as $\mu$--\emph{RSB expectation} of $\Psi$~\eqref{eq:auxiliary_func} and to the quantity~\eqref{eq:auxiliary_funccond} as \emph{conditional} $\mu$--RSB expectation. The probability measure $\mu$ is called \emph{Parisi parameter}.

Denote by $\textup{Pr}([0,1])$ the space of probability measure on $[0,1]$, and by $B_1([-1,1])$ the space of bounded, $\mathcal{F}^{\otimes n}_1$--measurable Wiener functional taking values on $[-1,1]$ $\mathbb{W}$--almost surely.

We define the full--RSB functional for the model~\eqref{eq:model} as
\begin{equation} \label{eq:full_rsb_domain}
\mathcal{P} : B_1([-1,1]) \times \textup{Pr}([0,1]) \to \mathbb{R},
\end{equation}
given, for $m \in B_1([-1,1])$ and $\mu \in \textup{Pr}([0,1])$, by
\begin{equation} \label{eq:full_rsb}
\mathcal{P}(m, \mu) := \mathbb{E}_{\bm{J}} \left[ \Phi\left( \psi^{(v)}_{\bm{J}} \circ m^{\times 2c},\, \mu,\, 0 \right) - \Phi\left( \psi^{(e)}_{\bm{J}} \circ m^{\times 2c},\, \mu,\, 0 \right) \right],
\end{equation}
where $\psi^{(e)}$ and $\psi^{(v)}$ are defined in~\eqref{eq:psie} and~\eqref{eq:psiv}, respectively, and
\begin{equation} \label{eq:psic2}
\psi^{(*)}_{\bm{J}} \circ m^{\times 2c}(\bm{\omega}) := \psi^{(*)}_{\bm{J}} \left( m(\omega_1), \dots, m(\omega_{2c}) \right),
\end{equation}
with the superscript $(*)$ denoting either $(e)$ or $(v)$, as appropriate.

The functional
\begin{equation} \label{eq:cavm}
m \in B_1([-1,1])
\end{equation}
is referred to as the \emph{cavity magnetization functional}, and serves as a variational parameter in the full--RSB formulation, alongside $\mu$.

The full--RSB functional provides a refinement of the finite--RSB approximation~\eqref{eq:main-identity}, as formalized in the following result.
\begin{theorem} \label{th:The_Theorem}
For any $K \in [N]$, the free energy satisfies
\begin{equation} \label{eq:rsb_inequality}
f \leq \inf_{\substack{m \in B_1([-1,1]) \\ \mu \in \textup{Pr}([0,1])}} \mathcal{P}(m, \mu)
\leq \inf_{\substack{\zeta \in \mathcal{M}^{(0)} \\ 0 < x_1 \leq \cdots \leq x_K = 1}} \widehat{\mathcal{P}}_K(\zeta, x).
\end{equation}
\end{theorem}
The main conjecture proposed in~\cite{MyPaper}, widely considered a central open problem in the field, states the following full--RSB variational principle:
\begin{equation}
\label{eq:claim}
f = \inf_{\substack{m \in B_1([-1,1]) \\ \mu \in \textup{Pr}([0,1])}} \mathcal{P}(m, \mu).
\end{equation}

Before any progress in proving~\eqref{eq:claim} can be made, a deeper understanding of the analytical properties of the functional~\eqref{eq:full_rsb} is necessary.  A central theme of this manuscript is the analysis of the conditional $\mu$-RSB expectation $\Phi$ defined in~\eqref{eq:auxiliary_funccond}.

The $\mu$--RSB expectation in~\eqref{eq:auxiliary_func} generalizes the variational representation of the solution to the Parisi PDE (see~\cite[Equation (15)]{ChenAuf}). Specifically, setting $n=1$ and
\begin{equation}
\label{eq:Par_initial}
\Psi(\omega) = \log(\cosh(\omega(1))),
\end{equation}
the $\mu$--RSB expectation gives the solution of the Parisi PDE. The Parisi PDE and its variational formulation have been thoroughly studied (see~\cite{Tala2,Talbook2,ChenAuf,jaggarnat}). The analysis in this context relies on the structure of the initial condition~\eqref{eq:Par_initial}, which is smooth, convex, and cylindrical—that is, it depends on $\omega \in C([0,1], \mathbb{R})$ only through its value at time $1$. In contrast, the full--RSB functional~\eqref{eq:full_rsb} and the conjectured identity~\eqref{eq:claim} require considering a broader class of Wiener functionals. Extending the results established for the Parisi PDE to this setting poses significant challenges in both functional and stochastic analysis.
 
Our first result establishes a representation of the $\mu$--RSB expectation as the solution to a proper backward stochastic differential equation (BSDE).
\begin{theorem}
\label{th:func_cont}
Let $\Psi\in B_n$ and $\mu\in \textup{Pr}([0,1])$, then there exists a unique $\Fqfilt{q}{[0,1)}$--predictable process $\bm{r}(\Psi,\mu)\in D_n$ such that
\been{
\Phi(\Psi,\mu,q,\bm{\omega})=\Gamma\big(\,\Psi,\,\mu,\,\bm{r}(\Psi,\mu),\,q,\bm{\omega}).
}
For any $q\in [0,1]$ the pair $(\Phi(\Psi,\mu),\bm{r}(\Psi,\mu))$ verifies the following BSDE:
\been{
\label{selfEq0Aux}
\Phi(\Psi,\mu,q,\bm{\omega})=\Psi(\bm{\omega})-\int^1_q\bm{r}(\Psi,\mu,t,\bm{\omega})\cdot d\bm{\omega}(t)+\frac{1}{2}\int^1_qdt\mu([0,t])\|\bm{r}(\Psi,\mu,t,\bm{\omega})\|^2_2.
}
\end{theorem}
The BSDE~\eqref{selfEq0Aux} can be viewed as a non-Markovian analogue of the Parisi PDE. The above Thorem extends the stochastic representation of the Parisi solution originally developed by Chen and Auffinger~\cite[Theorem 3]{ChenAuf}, which is based on a backward SDE formulation of the classical Parisi PDE~\cite[Eq. III.55]{VPM}.

Unlike their setting, where the variational formula is derived from the PDE, our result proceeds in the opposite direction: starting from a variational representation and constructing a BSDE. Due to the non-Markovian nature of the underlying functional, a PDE formulation is not available, and the analysis must rely on different probabilistic tools.

The next theorem establishes the well-posedness of the BSDE representation.
\begin{theorem}
\label{th:BSDE}
For any $\Psi\in B_n$ and $\mu\in \textup{Pr}([0,1])$ the solution to the BSDE~\eqref{selfEq0Aux} exists and it is unique.
\end{theorem}
The representation of the Parisi functional as a solution to a  PDE has proven to be a valuable tool in the investigation of its properties \cite{jaggarnat}. The BSDE representation enables the application of stochastic dynamic programming methods and the theory of BSDEs (see \cite{PaPeng}) for a quantitative analysis of the functional.

The analysis developed in this manuscript enables us to establish continuity and regularity properties of the $\mu$-RSB expectation~\eqref{eq:auxiliary_func} with respect to both the functional $\Psi$ and the Parisi parameter $\mu$.
\begin{theorem}
\label{th:func_quant}
Let $(\phi(\Psi,\mu),\bm{r}(\Psi,\mu))$ be the unique solution to the BSDE~\eqref{th:BSDE} corresponding to a given $\Psi\in B_n$ and $\mu\in \textup{Pr}([0,1])$. Then:
\begin{enumerate}
\item given $\Psi\in B_n$ and $\mu\in \textup{Pr}([0,1])$, the process $(t,\bm{\omega})\mapsto \mathcal{E}(\mu \bm{r}(\Psi,\mu),q,\bm{\omega})$ is a non-negative $\Wn$--martingale, adapted to $\Fqfilt{q}{[0,1]}$, with mean $1$.
\item Given $\mu \in \textup{Pr}([0,1])$, the map $\Psi \mapsto \Phi(\Psi, \mu)$ is convex and continuous in $\Psi$. Moreover, for any $\Psi' \in B_n$, the following holds $\mathbb{W}$-almost surely:
\been{
\aled{
&\left.\frac{\partial}{\partial t}\Phi(\Psi+t(\Psi'-\Psi),\mu,q,\bm{\omega})\right|_{t=0}\\
&=\Econd{\mathcal{E}(\mu \bm{r}_q(\Psi,\mu),1,\bm{\omega})(\Psi'(\bm{\omega})-\Psi(\bm{\omega}))}{q},\quad \Wnas;
}
}
where $\bm{r}_q(\Psi,\mu)$ is defined as in~\eqref{eq:rq}.
\item Given $\Psi \in B_n$, the map $\mu \mapsto \Phi(\Psi, \mu)$ is continuous on $\textup{Pr}([0,1])$. Moreover, for any $\mu' \in \textup{Pr}([0,1])$, the directional derivative exists $\mathbb{W}$--almost surely and satisfies:
\been{
\aled{
&\left.\frac{\partial}{\partial t}\Phi(\Psi,\mu+t(\mu'-\mu),q,\bm{\omega})\right|_{t=0}\\
&=\Econd{\mathcal{E}(\mu \bm{r}_q(\Psi,\mu),1,\bm{\omega})\int^1_qdt(\mu'-\mu)([0,t])\|\bm{r}(\Psi,\mu,t,\bm{\omega})\|^2_2}{q}.
}
}
\item Let 
\been{
\|\Psi\|_{\infty}=\inf\left\{C>0\vert \,\Wn(|\Psi(\bm{\omega})|\leq C)=1\right\}.
}
Then there exists a constant $K(|\Psi|{\infty}) > 0$ (depending only on $\|\Psi\|_{\infty}$) such that for all $\mu, \mu' \in \textup{Pr}([0,1])$,
\been{
\aled{
&\Econd{\mathcal{E}(\mu \bm{r}_q(\Psi,\mu),1,\bm{\omega})\int^1_qdt(\mu'-\mu)([0,t])\|r(\Psi,\mu,t,\bm{\omega})\|^2_2}{q}\\
&\leq K(\|\Psi\|_{\infty})\sup_{t\in[0,1]}\((\mu'-\mu)([0,t])\),\quad \Wnas.
}
}
\end{enumerate}
\end{theorem}
The manuscript is organized as follows. Section~\ref{sec:3} introduces the notation adopted throughout the manuscript.  Section~\ref{sec:2} is devoted to proving some properties of the process $\mathcal{E}$ that will be useful for the analysis of the BSDE~\ref{selfEq0Aux}. In Section~\ref{sec:41} we prove the existence result for the BSDE~\ref{selfEq0Aux} when the Parisi parameter is a discrete measure with finitely many atoms (finite--RSB) and obtain an explicit solution. In Section~\ref{sec:6} we will show that such a solution is continuous with respect to the Parisi parameters and prove Theorem~\ref{th:func_quant} for discrete $\mu$. Thank to the continuity, in Section ~\ref{sec:42} we extend the existence result to any allowable $\mu$, by taking the limit of a sequence of approximating discrete Parisi parameters. We prove the uniqueness of the solution in Section~\ref{sec:43}, completing the proof of Theorem~\ref{th:BSDE}. In Section~\ref{sec:4} we prove Theorem~\ref{th:func_cont}. We extend the proof of Theorem~\ref{th:func_quant} to all the allowable Parisi parameter in Section~\ref{sec:7}. Finally, in Section~\ref{sec:10}, we prove Theorem~\ref{th:The_Theorem}.
\section{Main notation and definitions}
\label{sec:3}
This section introduces the notations used throughout the manuscript. The symbol $\bm{r}\cdot \bm{v}$ denotes the Euclidean inner product between two vectors $\bm{r}$ and $\bm{v}$ in $\R^n$, and $\|\bm{v}\|_2^2=\bm{v}\cdot \bm{v}$. Given any integer $n\in \N$ ($\N$ is the set of all natural numbers but $0$) we indicate
\been{
[n]:=\{1,\cdots,n\},
}
and
\been{
[n]_0:=\{0,1,\cdots,n\}.
}
We denote the indicator function of a given set $A$ by $\bm{1}_A$.

Given two real numbers $x_1$ and $x_2$, let
\been{
x_1\wedge x_2:=\min\{x_1,x_2\},
}
and
\been{
x_1\vee x_2:=\max\{x_1,x_2\}.
}
For $n\in \N$, we consider the filtered probability space $(C([0,1],\R^n),\mathcal{F}^{\otimes n},\mathbb{W}^{\otimes n},\(\mathcal{F}^{\otimes n}_q\)_{q\in[0,1]})$ defined in the introduction. 

Given two random variables $\Psi_1$ and $\Psi_2$, we say that they are in the same \emph{equivalence class} if
\been{
\W^{\otimes n}\(\Psi_1(\bm{\omega})=\Psi_2(\bm{\omega})\)=1.
}

Denoting by $\mathcal{B}([0,1])$ the Borel $\s$--algebra in $[0,1]$ and given two $\mathcal{B}([0,1])\otimes \mathcal{F}^{\otimes n}$--measurable processes $v_1$ and $v_2$, we say that:
\begin{enumerate}
    \item they are in the same \emph{indistinguishability class} if
    \been{
    \label{eq:indi}
    \W^{\otimes n}(v_1(q,\bm{\omega})=v_2(q,\bm{\omega});\, \forall q\in [0,1))=1;
    }
    \item they are in the same \emph{equivalence class} if
    \been{
    \label{eq:modi}
    \W^{\otimes n}\(\int^1_{0}dt\|v_1(q,\bm{\omega})-v_2(q,\bm{\omega})\|^2_2=0\)=1.
    }
\end{enumerate}
We define the following sets:
\begin{enumerate}
\item  $B_n$ is the set of equivalence classes of all $\mathcal{F}^{\otimes n}_1$--measurable Wiener functional $\Psi:C([0,1],\R^n)\to \R$ that are bounded (i.e., there exists $M>0$ such that $|\Psi(\bm{\omega})|<M$ $\W^{\otimes n}$--a.s.);
\item $D_n$ denotes the set of equivalence classes of all the $\(\mathcal{F}^{\otimes n}_q\)_{q\in[0,1)}$--progressive processes $\bm{v}: [0,1)\times C([0,1],\R^n)\to\R^n$ such that $\mathbb{W}^{\otimes n}\(\displaystyle\int^1_0dq\,\|\bm{v}(q,\bm{\omega})\|^2_2<\infty\)=1$;
\item $S_n$ denotes the space of indistinguishability classes of all the $\(\mathcal{F}^{\otimes n}_q\)_{q\in[0,1)}$--adapted processes $\phi:[0,1)\times \Omega\to\R$ satisfying $\mathbb{W}^{\otimes n}\(\,\underset{q\in[0,1)}{\sup}|\phi(q,\bm{\omega})|< \infty\,\)=1$.
\item $S^b_n$ denotes the space of all the elements of $S_n$ for which there exists $M>0$ such that $\underset{q\in[0,1)}{\sup}|\phi(q,\bm{\omega})|<M$ $\W^{\otimes n}$--a.s.
\end{enumerate}
We say that a process is $\(\mathcal{F}^{\otimes n}_q\)_{q\in[0,1]}$--adapted if it is $\(\mathcal{F}^{\otimes n}_q\)_{q\in[0,1)}$--adapted and the limit $\lim_{q\uparrow 1}\phi(q,\bm{\omega})$ exists and it is a $\mathcal{F}^{\otimes n}_1$--measurable random variable.

Given an equivalence class $\bm{v}\in D_n$ and a process $\bm{r}\in \bm{v}$, let $I(\bm{r})$ denote the Ito integral of $\bm{r}$
\been{
I(\bm{r})(q,\bm{\omega})=I(\bm{r},q,\bm{\omega}):=\int^q_0\bm{r}(q,\bm{\omega})\cdot d\bm{\omega}(q),\quad \forall q\in [0,1).
}
For any $\bm{v}\in D_n$, there exists a unique indistinguishability class $\phi\in S_n$ such that, for any process $\bm{r}\in \bm{v}$, $I(\bm{r})\in \phi$. With abuse of notation, we denote this element as $I(\bm{v})$ and we formally write $I(\bm{v},q,\bm{\omega})=\int^q_0\bm{v}(q,\bm{\omega})\cdot d\bm{\omega}(q)$. The process $I(\bm{v})$ is a continuous $\W^{\otimes n}$--local martingale adapted to the filtration $(\mathcal{F}_q)_{q\in [0,1)}$.

For any pair of progressive processes $\bm{r}_1$ and $\bm{r}_2$ that are in the same equivalence class $\bm{v}\in D_n$, the Monotone Convergence Theorem gives
\been{
\aled{
&\W^{\otimes n}\(\int^q_0dt\|\bm{r}_1(t,\bm{\omega})\|^2_2\neq \int^q_0dt\|\bm{r}_2(t,\bm{\omega})\|^2_2,\quad \forall q\in [0,1]\)\\
&\leq \W^{\otimes n}\(\int^1_0dt\|\bm{r}_1(t,\bm{\omega})-\bm{r}_2(t,\bm{\omega})\|^2_2\neq 0\)=0.
}
}
Thus the process $(q,\bm{\omega})\mapsto\displaystyle \int^q_0dt\|\bm{r}_1(t,\bm{\omega})\|^2_2$ and $(q,\bm{\omega})\mapsto\displaystyle  \int^q_0dt\|\bm{r}_2(t,\bm{\omega})\|^2_2$ are in the same indistinguishability class. We denote such equivalence class as $(q,\bm{\omega})\mapsto\displaystyle  \int^q_0dt\|\bm{v}(t,\bm{\omega})\|^2_2$ as well.

In the following, we will make no distinctions between equivalence classes and their members (for example, we will use the word \emph{process} instead of \emph{equivalence class}). Moreover, if an event occurs almost surely, we will occasionally omit stating  $\W^{\otimes n}$--a.s..

Given a process $X$ and a stopping time $T$, we denote by $X^T$ the \emph{stopped} process
\been{
X^T(q,\bm{\omega}):=X(q\wedge T,\bm{\omega}).
}
Given $q\in [0,1]$, we denote by $X_q$ the process
\been{
\label{eq:uproot}
X_q(q',\bm{\omega}):=
\begin{dcases}
   X(q',\bm{\omega}),\quad &\textup{if }q'>q;\\
   0,\quad &\textup{otherwise}.
\end{dcases}
}
Given a process $\bm{v}\in D_n$, let
\been{
\label{zeta0}
\gamma(\bm{v}, q_2,\bm{\omega}):=
\int^{q_2}_{0} \bm{v}(t\,, \,\bm{\omega}) \cdot d\bm{\omega}(t) -\frac{1}{2}\int^q_{0}dt\,\,\|\bm{v}(t\,, \,\bm{\omega})\|^2_2
}
So, the DDE is given by
\begin{equation}
\label{eq:DDE}
\mathcal{E}(\bm{v},q,\bm{\omega}):=e^{\gamma(\bm{v},q,\bm{\omega})}
\end{equation}
We denote by $\gamma(\bm{v})$ the process $(q,\bm{\omega})\mapsto \gamma(\bm{v},q,\bm{\omega})$  and by $\mathcal{E}(\bm{v})$ the process $(q,\bm{\omega})\mapsto \mathcal{E}(\bm{v},q,\bm{\omega})$.

Using the notation~\eqref{eq:uproot}, given $q\in[0,1]$, the processes $\gamma(\bm{v}_q)$ and $\mathcal{E}(\bm{v}_q)$ are defined as follows
\been{
\label{zeta}
\gamma(\bm{v}_{q}, q',\bm{\omega}):=
\begin{dcases}
\int^{q'}_{q} \bm{v}(t\,, \,\bm{\omega}) \cdot d\bm{\omega}(t) -\frac{1}{2}\int^{q'}_{q}dt\,\,\|\bm{v}(t\,, \,\bm{\omega})\|^2_2,\quad &\text{if}\,\quad q'>q;\\
0\,,\quad &\text{otherwise}.
\end{dcases}
}
and
\been{
\mathcal{E}(\bm{v}_{q},q',\bm{\omega}):=e^{\gamma(\bm{v}_{q},q',\bm{\omega})}=\bm{1}_{q'\leq q}+\bm{1}_{q'> q}\frac{\mathcal{E}( \bm{v},q',\bm{\omega})}{\mathcal{E}( \bm{v},q,\bm{\omega})}.
}

We also define the following sets (sets of equivalence classes):
\been{
D^{2,p}_n:=\left\{\bm{v}\in D_n;\,\mathbb{E}\left[\,\left(\int^1_0dq\,\|\bm{v}(q,\bm{\omega})\|^2_2\right)^{\frac{p}{2}}\]<\infty\right\};
}
\been{
\widehat{D}_n:=\left\{\bm{v}\in D_n;\,\E\[\int^{1}_0\mathcal{E}(\bm{v},q,\bm{\omega})\|\bm{v}(q,\bm{\omega})\|^2_2dq\]<\infty\right\};
}
\been{
S^p_n:=\left\{\bm{v}\in S_n;\,\mathbb{E}\Big[\,\underset{q\in[0,1]}{\sup}|\phi(q,\bm{\omega})|^{p}\,\Big]^{\frac{1}{p}}<\infty\right\};
}
Over the above sets, we define the following norms
\begin{align}
\label{norms1}
\Psi\in B_n \,:&\quad  \|\Psi\|_{\infty}:=\inf \,\left\{M\in \R, |\Psi(\bm{\omega})|\leq M\quad \W^{\otimes n}-a.s.\right\}\,,
\\
\label{norms2}
\bm{v}\in D^{2,p}_n \,:&\quad \|\bm{v}\|_{2,p}:=\mathbb{E}\left[\,\left(\int^1_0dq\,\|\bm{v}(q,\bm{\omega})\|_2^2\right)^{\frac{p}{2}}\,\right]^{\frac{1}{p}}\,,
\\
\label{norms3}
\phi\in S_n^p \,:&\quad \|\phi\|_{\infty,p}:=\mathbb{E}\Big[\,\underset{q\in[0,1]}{\sup}|\phi(q,\bm{\omega})|^{p}\,\Big]^{\frac{1}{p}}\,.
\end{align}
We will say that a sequence of pair $((\phi_k,\bm{v}_k))_{k\in \N}\subseteq S_n^p\times D^{2,p}_n$ converges in $S_n^p\times D^{2,p}_n$ norm if $(\phi_k)_{k\in \N}$ converges with respect the norm $\|\,\cdot\,\|_{\infty,p}$ and $(\bm{r}_k)_{k\in \N}$ converges with respect the norm $\|\,\cdot\,\|_{2,p}$. By Burkholder-Davis-Gundy Inequality  (see \cite[Theorem 3.28]{KS}) if $\bm{v}\in  D^{2,p}$, then $I(\bm{v})\in S_n^p$.

We denote by $\textup{Meas}([0,1])$ the vector space of bounded signed measure on $[0,1]$ endowed with the following norm:
\been{
\|\mu\|_{\infty}:=\sup_{q\in [0,1]}|\mu([0,q])|,\quad \mu\in \textup{Meas}([0,1]).
}
Let $\textup{Pr}([0,1])\subseteq\textup{Meas}([0,1]) $ be the set of probability measure on $[0,1]$.

Given $\bm{r}\in D_n$ and $\mu\in \textup{Meas}([0,1])$, we denote by $\mu \bm{r}\in  D_n$ the progressive process defined as 
\been{
(\mu \bm{r})(q,\bm{\omega}):=\mu([0,q])\bm{r}(q,\bm{\omega}).
}
Obviously, if $\bm{r}\in D^{2,p}_n$, then $\mu\bm{r}\in D^{2,p}_n$. Finally, given $\mu \in \textup{Pr}([0,1])$, let us define
\been{
\widehat{D}^{\mu }_n:=\left\{\bm{v}\in D_n;\,\E\[\int^{1}_0\mathcal{E}(\mu \bm{v},q,\bm{\omega})\mu([0,q])\|\bm{v}(q,\bm{\omega})\|^2_2dq\]<\infty\right\};
}
Since $\mu([0,q])\geq \mu^2([0,q])$, if $\bm{v}\in \widehat{D}^{\mu }_n$, then $\mu \bm{v}\in \widehat{D}_n$.
\section{Properties of the Doléans-Dade exponential (DDE)}
\label{sec:2}
In this section, we recall some classical results about the DDEs and obtain some new results that will be useful for the proof of the main theorems of the manuscript.  We refer the reader to Section 3.5 of the book of Karatzas and Shreve \cite{KS} for a detailed discussion of this topic. To the best of the author's knowledge, the first statement provided in Lemma~\ref{lem:martcond} is a new result.

For any process $\bm{v}\in D_n$, the DDE $\mathcal{E}(\bm{v})$ defined in~\eqref{eq:DDE} is a positive local martingale. Hence it is a supermartingale (see \cite[ Problem 2.28, Chapter 3]{KS}).

Since $I(\bm{v})$ is a local martingale, the definition~\eqref{eq:DDE} is well posed for $q\in [0,1)$. The random variable $\mathcal{E}(\bm{v},1,\bm{\omega})$ can be defined by taking the limit $q\to 1$.
\begin{lemma}
\label{lem:convDDE}
For any $\bm{v}\in D_n$
\been{
\liminf_{q\to 1}\mathcal{E}(\bm{v},q,\bm{\omega})=\limsup_{q\to 1}\mathcal{E}(\bm{v},q,\bm{\omega})\in [0,\infty)\quad \W^{\otimes n}-\textup{a.s.}.
}
Taking $\mathcal{E}(\bm{v},1,\bm{\omega})=\lim_{q\to 1}\mathcal{E}(\bm{v},q,\bm{\omega})$, the process $\mathcal{E}(\bm{v})$ is a $\W^{\otimes n}$--supermartingale adapted to $(\mathcal{F}^{\otimes n}_q)_{q\in [0,1]}$.
\end{lemma}
\begin{proof}
The process $\mathcal{E}(\bm{v})$ is a non-negative supermartingale. Thus, by \cite[Problem 3.16, Chapter 1]{KS}, $\mathcal{E}(\bm{v},q,\bm{\omega})$ converges $\W^{\otimes n}$--almost surely as $q\to 1$ and, taking $\mathcal{E}(\bm{v},1,\bm{\omega})=\lim_{q\to 1}\mathcal{E}(\bm{v},q,\bm{\omega})$, the process $\mathcal{E}(\bm{v})$ is a $W^{\otimes n}$--supermartingale adapted to $(\mathcal{F}^{\otimes n}_q)_{q\in [0,1]}$. Moreover, the super-martingale property implies:
\been{
\E\[\mathcal{E}(\bm{v},1,\bm{\omega})\]\leq \E\[\mathcal{E}(\bm{v},0,\bm{\omega})\]=1
}
\end{proof}
If $\mathcal{E}(\bm{v})$ is a true martingale, then we can define a probability measure $\mathbb{W}_{\bm{v}}$ on the measurable space $(C([0,1],\R^n),\mathcal{F}^{\otimes n}_1)$ that is absolutely continuous with respect to the Wiener measure $\mathbb{W}^{\otimes n}$ and such that the DDE $\mathcal{E}(\bm{v})$ is the Radon-Nikodym derivative of $\mathbb{W}_{\bm{v}}$ with respect to $\mathbb{W}$:
\begin{equation}
\label{eq:defWv}
\frac{d\mathbb{W}_{\bm{v}} }{d\mathbb{W}^{\otimes n}}(\bm{\omega}):=\mathcal{E}\left(\,\bm{v},1,\bm{\omega}\,\right)\,.
\end{equation}
Determining whether a DDE $\mathcal{E}(\bm{v})$ defines a true martingale is a classical and subtle problem in stochastic calculus. Over the years, several sufficient conditions on the drift $\bm{v}$ have been developed, including the well-known Novikov and Kazamaki criteria \cite[Section 3.5.D]{KS}. In the proposition below, we establish a new sufficient condition tailored to the specific structure arising in this work. 

If $\mathcal{E}(\bm{v})$ is a martingale, then we can apply the Cameron--Martin, Girsanov Theorem (CMG).
\begin{theorem}
\label{lem:martcond} 
The following holds:
\begin{enumerate}
    \item If $\bm{v}\in \widehat{D}_n$ the process $\mathcal{E}(\bm{v})$ is a non-negative uniformly integrable martingale with mean $1$;
    \item(CMG Theorem) if $\mathcal{E}(\bm{v})$ is a martingale the measure $\mathbb{W}_{\bm{v}}$ defined in~\eqref{eq:defWv} is a probability measure, and  the following process
    \been{
\label{eq:shiftedBM}
\bm{W}_{\bm{v}}\left(\,q,\bm{\omega}\right):=\bm{\omega}(q)-\bm{\omega}(0)-\int^q_0 dt\, \, \bm{v}(t,\bm{\omega}\,)
}
is a $n$--dimensional Brownian with respect the probability measure $\mathbb{W}_{\bm{v}}$ adapted to $\Fqfilt{q}{[0,1]}$.
\end{enumerate} 
\end{theorem}
\begin{proof}
If the DDE $\mathcal{E}(\bm{v})$ is a martingale, the second statement of the lemma follows from the CMG Theorem (see \cite[Corollary 5.2]{KS}). 

We have to prove the first statement. Let $(\tau_k)_{k\in \N}$ be the sequence of stopping times define as follows
\been{
\tau_k=\inf\left\{q\in [0,1); \,\int^{T_n}_0dq\|\bm{v}(q,\bm{\omega})\|^2_2\geq  k \right\},
}
with the convention $\inf \emptyset=1$. Let 
\been{
\bm{v}_{0,\tau_k}(q,\bm{\omega})=\bm{1}_{\{q\leq \tau_k\}}\bm{v}(q,\bm{\omega}).
}
Thus $\mathcal{E}^{\tau_k}(\bm{v})=\mathcal{E}(\bm{v}_{0,\tau_k})$. By Novikov Criterion \cite[Chapter 3, Proposition 5.2]{KS}, for each $k\in \N$, $\mathcal{E}(\bm{v}_{0,\tau_k})$ is a martingale of average $1$. Then, by the second point of this Theorem,  we can define the probability measure $\mathbb{W}_{\bm{v}_{0,\tau_k}}$ as in~\eqref{eq:defWv} and $\bm{W}_{\bm{v}_{0,\tau_k}}$ is a $\mathbb{W}_{\bm{v}_{0,\tau_k}}$--Brownian motion. Thus, the stochastic integral
\been{
\int^{t\wedge\tau_k }_0\bm{v}(q,\bm{\omega})\cdot d\bm{W}_{\bm{v}}(q,\bm{\omega})=I(\bm{v}_{0,\tau_k},t,\bm{\omega})-\int^{t\wedge\tau_k }_0dq\|\bm{v}(q,\bm{\omega})\|^2_2
}
is a $\mathbb{W}_{\bm{v}_{0,\tau_k}}$--square integrable martingale. Define
\been{
\phi(x)=x \log(x)+e^{-1}.
}By the martingale property, we get
\been{
\aled{
&\E\[\phi\(\mathcal{E}(\bm{v}_{0,\tau_k},t,\bm{\omega})\)\]\\
&=\E\[\mathcal{E}(\bm{v}_{0,\tau_k},t,\bm{\omega})\(\int^{t\wedge\tau_k }_0\bm{v}(q,\bm{\omega})\cdot d\bm{W}_{\bm{v}}(q,\bm{\omega})+\frac{1}{2}\int^{t\wedge\tau_k }_0dq\|\bm{v}(q,\bm{\omega})\|^2_2\)\]+e^{-1}\\
&=\frac{1}{2}\E\[\int^{t\wedge\tau_k }_0dq\mathcal{E}(\bm{v}_{0,\tau_k},q,\bm{\omega})\|\bm{v}(q,\bm{\omega})\|^2_2\]+e^{-1}.
}
}
The above identity and the Monotone Convergence Theorem give
\been{
\aled{
\sup_{k\in \N}\E\[\phi\(\mathcal{E}(\bm{v}_{0,\tau_k},t,\bm{\omega})\)\]\leq \frac{1}{2}\E\[\int^{1 }_0dq\mathcal{E}(\bm{v}_{0,\tau_k},q,\bm{\omega})\|\bm{v}(q,\bm{\omega})\|^2_2\]+e^{-1}<\infty.
}
}
The function $\phi$ is positive and $\lim_{x\to \infty} \phi(x)/x=\infty$. Thus, by de La Vallée Poussin Theorem, the sequence $\(\mathcal{E}(\bm{v}_{0,\tau_k},t,\bm{\omega})\)_{k\in \N}$ is uniformly integrable for any $t\in [0,1]$. Hence, by Vitali Convergence Theorem
\been{
\E\[\mathcal{E}(\bm{v},t,\bm{\omega})\]=\lim_{k\to \infty}\E\[\mathcal{E}(\bm{v}_{0,\tau_k},t,\bm{\omega})\]=1,\quad \forall t\in [0,1].
}
Since $\mathcal{E}(\bm{v})$ is a supermartingale, then, by the above identity, it is a martingale.
\end{proof}
Given $\bm{v}\in \widehat{D}_n$, we denote by $\widetilde{\E}_{\bm{v}}$ the expectation with respect to $(C([0,1],\R^n),\mathcal{F}^{\otimes n}_1,\mathbb{W}_{\bm{v}})$. Given $\bm{u}\in D_n$ and $\bm{v}\in \widehat{D}_n$, let us denote by $I_{\bm{v}}(\bm{u})$ the stochastic integral
\been{
\label{eq:Iv}
I_{\bm{v}}(\bm{u})(q,\bm{\omega}):=I_{\bm{v}}(\bm{u},q,\bm{\omega})=\int^q_0\bm{u}(t,\bm{\omega})\cdot d\bm{W}_{\bm{v}}(t,\bm{\omega}),\quad q\in [0,1).
}
The process above is a local martingale under the probability measure $\W_{\bm{v}}$. For the subsequent results, it is crucial to establish conditions on $\bm{u}$ that guarantee $I_{\bm{v}}(\bm{u})$ is a true martingale. The following lemma provides a straightforward characterization.
\begin{lemma}
\label{lem:mart_cond}
Let $\bm{v}\in \widehat{D}_n$ and $\bm{u}\in D_n$. If
\been{
\label{e:mart_cond}
\widetilde{\E}_{\bm{v}}\[\int^1_0dt\|\bm{u}(t,\bm{\omega})\|^2_2\]< \infty,
}
then $I_{\bm{v}}(\bm{u})$ is a square integrable martingale with mean $0$ with respect the probability measure $\W_{\bm{v}}$.
\end{lemma}
\begin{remark}
By the above lemma, for $\bm{v}\in \widehat{D}_n$, the stochastic integral $I_{\bm{v}}(\bm{v})$ is a square integrable martingale.
 \end{remark}
\begin{proof}
If
\been{
\widetilde{\E}_{\bm{v}}\[\int^1_0dt\|\bm{u}(t,\bm{\omega})\|^2_2\]< \infty
}
then by \cite[Corollary 1.24, Chapter IV]{Yor}, the local martingale $I_{\bm{v}}(\bm{u})$ is a square integrable martingale.
\end{proof}

\section{Some properties of the solutions}
\label{sec:6}
In this section, we assume that the BSDE~\eqref{selfEq0Aux}
\been{
\label{selfEq1Aux}
\phi(q,\bm{\omega}) = \Psi(\bm{\omega}) - \int_q^1 \bm{r}(t, \bm{\omega}) \cdot d\bm{\omega}(t) + \tfrac{1}{2}\int_q^1 \mu([0,t])|\bm{r}(t, \bm{\omega})|_2^2  dt
}
admits at least one solution. Under this assumption, we derive several properties of the solution, which will be used in the proofs of Theorem~\ref{th:func_cont} and Theorem~\ref{th:BSDE}.

The main results of the section are stated in the following proposition.
\begin{proposition}
\label{prop:uniqueness2}
If $\Psi \in B_n$ and the pair $(\phi^*,\bm{r}^*)\in S_n\times D_n$ is a solution to the \BSDE, then $(\phi^*,\bm{r}^*)\in S^b_n\times  \widehat{D}^{\mu}_n$ and
\been{
\label{eq:uniqueness2}
\sup_{q\in [0,1]}|\phi^*(q,\bm{\omega})|\leq \|\Psi\|_{\infty},\quad \Wnas,
}
\been{
\label{eq:uniqueness22}
\sup_{q\in [0,1]}\left| \gamma(\mu\bm{r}^*,q,\bm{\omega})\right|
\leq 2 \|\Psi\|_{\infty},\quad \Wnas,
}
and the process $\mathcal{E}(\mu \bm{r}^*)$ is a $\Wn$--martingale adapted to $\Fqfilt{q}{[0,1]}$ with
\been{
\label{eq:uniqueness222}
\Ee{\mu\bm{r}^*}\[\int^1_0dq\mu([0,q])\,\|\bm{r}(q,\bm{\omega})\|^2_2\]\leq 4\|\Psi\|_{\infty}.
}
\end{proposition}
A direct consequence of the above Proposition is the following
\begin{corollary}
\label{rem:uniqueness222}
If $\Psi \in B_n$ and the pair $(\phi^*,\bm{r}^*)\in S_n\times D_n$ is a solution to the \BSDE, for any $p>0$, there exists  constant $K(p,\|\Psi\|_{\infty})$ independent of $\mu$ such that for any $\bm{v}\in \widehat{D}^{\mu}_n$
\begin{equation}
\label{bound_rr}
\widetilde{\mathbb{E}}_{\mu\bm{r}^*}\left[\left(\int^{1}_{0}dq\,\|\bm{r}^*(q,\bm{\omega})\|^{2}\right)^{p}\right]\leq K(p,\|\Psi\|_{\infty})
\end{equation}
\begin{equation}
\label{bound_rr2}
\E\left[\left(\int^{1}_{0}dq\,\|\bm{r}^*(q,\bm{\omega})\|^{2}\right)^{p}\right]\leq K(p,\|\Psi\|_{\infty})
\end{equation}
and
\been{
\label{bound_rr3}
\Eecond{\mu\bm{r}^*}{\int^{1}_{q}dt\,\|\bm{r}^*(t,\bm{\omega})\|^{2}}{q}\leq  K(1,\|\Psi\|_{\infty}),\quad \Wnas.
}
and
\been{
\label{bound_rr33}
\Econd{\int^{1}_{q}dt\,\|\bm{r}^*(t,\bm{\omega})\|^{2}}{q}\leq  K(1,\|\Psi\|_{\infty}),\quad \Wnas.
}
\end{corollary}
The proofs of the proposition and the corollary are postponed to the end of the section, after the necessary notation and auxiliary lemmas have been introduced.
The proof of the above corollary (as we
By definition, if $(\phi^*,\bm{r}^*)\in S_n\times D_n$ is a solution to the \BSDE, then:
\begin{enumerate}
\item the process $\phi^*$ is $\(\mathcal{F}^{\otimes n}_q\)_{q\in [0,1]}$--adapted and $\lim_{q\uparrow 1}\phi^*(q,\bm{\omega})=\Psi(\bm{\omega})$;
\item the process $\phi^*$ has a continuous sample path (we will simply say that it is a continuous process).
\end{enumerate}
Given $x \in (0,1]$, define  
\begin{equation}
\label{eq:qx}
q_x := \sup\{q \in [0,1] : \mu([0,q)) < x\},
\end{equation}
with the convention $\sup \emptyset := 0$. Set  
\begin{equation}
\label{eq:qx0}
q_0 := \inf_{x \in (0,1]} q_x = \sup\{q \in [0,1] : \mu([0,q)) = 0\}.
\end{equation}
Recalling the nation~\eqref{eq:uproot}, for $q \in [0,1)$, $\bm{r} \in D_n$, we define  
\begin{equation}
\label{eq:rq2}
\bm{r}_q(q',\bm{\omega}) := \bm{1}_{q' \geq q}\bm{r}(q',\bm{\omega}).
\end{equation}
If $q \in [q_x,1]$, then $\mu([0,q]) \geq x$. Thus, for any $\epsilon \in [0,x]$,
\begin{equation}
\epsilon \int_q^t \|\bm{r}^*(s,\bm{\omega})\|^2_2 \, ds \leq \int_q^t \mu([0,s])\|\bm{r}^*(s,\bm{\omega})\|^2_2 \, ds \leq \int_q^t \|\bm{r}^*(s,\bm{\omega})\|^2_2 \, ds.
\end{equation}
Hence, for any $q \in [q_x,1]$, $x \in (0,1]$, and $t \in [0,1]$, the \BSDE yields
\begin{equation}
\label{eq:Wineq2}
\epsilon \big(\phi^*(t,\bm{\omega}) - \phi^*(q,\bm{\omega})\big) \leq \gamma(\epsilon\bm{r}^*_q,t,\bm{\omega}),
\end{equation}
and
\begin{equation}
\label{eq:Wineq3}
\phi^*(t,\bm{\omega}) - \phi^*(q,\bm{\omega}) \geq \gamma(\bm{r}^*_q,t,\bm{\omega}),
\end{equation}
where $\gamma(\cdot)$ is as defined in~\eqref{zeta}.

We first prove the lower bound of $\phi^*$.
\begin{lemma}
\label{lem:martcond4}
If $(\phi^*,\bm{r}^*)\in S_n\times D_n$ is a solution to the BSDE~\eqref{selfEq1Aux} and
\been{
\E\[|\Psi(\bm{\omega})|\]<\infty,
}
If
\been{
\Psi(\bm{\omega})\geq -C,\quad \Wnas
}
then 
\been{
\label{eq:psibound0}
\phi^*(q,\bm{\omega})\geq -C,\quad \forall q\in [q_0,1],\,\Wnas.
}
\end{lemma}
\begin{proof}
If $q_0=1$, the result is trivial. If $q_0<1$, for any $q>q_0$ there exists some $x>0$ such that $q>q_x$. Combining the inequality~\eqref{eq:Wineq2} with $\epsilon=x$, and the inequality $y \leq \frac{1}{x}(e^{xy} - 1)$ for any $y\in \R$, we get that 
\been{
\label{eq:deltaphi0}
\phi^*(t,\bm{\omega}) - \phi^*(q,\bm{\omega})\leq \frac{1}{x}\left(\mathcal{E}(x\bm{r}^*_{q},t,\bm{\omega})-1\right),\quad \forall t\geq q.
}
By Lemma~\ref{lem:convDDE}, $\mathcal{E}(x\bm{r}^*_{q})$ is a $\W^{\otimes n}$--supermartingale. This implies 
\been{
\E\[\mathcal{E}(x\bm{r}^*_{q},t,\bm{\omega})\middle|\mathcal{F}_s^{\otimes n}\]\leq \mathcal{E}(x\bm{r}^*_{q},s,\bm{\omega})=1,\quad \forall s\in [q_x,q].
}
Thus
\been{
\Econd{\Psi^*(1,\bm{\omega}) }{q}- \phi^*(q,\bm{\omega})=\Econd{\phi^*(1,\bm{\omega})- \phi^*(q,\bm{\omega})}{q}\leq 0.
}
We proved the above inequality for any $q\geq [q_x,1]$ with $x\in (0,1]$. We now take the limit $x\to 0$ and extend the inequality to all $q\in [q_0,1]$. The Fatou Lemma yields
\been{
\aled{
&\E\[\Psi^*(1,\bm{\omega}) - \phi^*(q_0,\bm{\omega})\middle|\mathcal{F}_{q_0}^{\otimes n}\]=\E\[\lim_{x\to 0}\(\Psi^*(1,\bm{\omega}) - \phi^*(q_x,\bm{\omega})\)\middle|\mathcal{F}_{q_0}^{\otimes n}\]\\
&=\E\[\lim_{x\to 0}\frac{1}{x}\left(\mathcal{E}(x\bm{r}^*_{q_x},t,\bm{\omega})-1\right)\middle|\mathcal{F}_{q_0}^{\otimes n}\]
\leq\lim_{x\to 0}\frac{1}{x}\left(\E\[\mathcal{E}(x\bm{r}^*_{q_x},t,\bm{\omega})\middle|\mathcal{F}_{q_0}^{\otimes n}\]-1\right)\leq 0.
}
}
Since $\Psi(\bm{\omega}) \geq -C$ $\W^{\otimes n}$--a.s., $\phi^*(q,\bm{\omega})\geq \E\[\Psi^*(1,\bm{\omega}) \middle|\mathcal{F}_{q}^{\otimes n}\]\geq -C$ for any $q\in [q_0,1]$.
\end{proof}
We now derive an upper bound for $\phi^*$. We define
\begin{equation}
\bm{r}^*_{q,T}(t,\bm{\omega}) := \bm{1}_{\{q \leq t \leq T\}} \bm{r}^*(t,\bm{\omega}).
\end{equation}
For any stopping time $T$, we have the following equivalent notations:
\begin{equation}
I^T(\bm{r}^*_{q,T},t,\bm{\omega}) = I_n(\bm{r}^*_{q,T},t,\bm{\omega}),\quad \textup{and}\quad 
\mathcal{E}^T(\epsilon \bm{r}^*_q) = \mathcal{E}(\epsilon \bm{r}^*_{q,T}).
\end{equation}
Given $q\in [0,1]$, we define the sequence of stopping times $(\sigma_{k,q})_{k\in \N}$ such that
\been{
\label{eq:stopping}
\sigma_{k,q}:=\inf\left\{t\in [q,1];\,\int^t_qds\|\bm{r}(s,\bm{\omega})\|^2_2\geq k\right\}.
}
with $\inf \emptyset=1$. Since $\bm{r}^*\in D_n$, then $\sigma_{k,q}\uparrow 1$ $\Wnas$.

Observe that, for any $\alpha \in \R$ and $k > 0$, we have 
$\alpha \bm{r}^*_{q,\sigma_{k,q}} \in \widehat{D}_n$. Hence, by 
Lemma~\ref{lem:martcond}, the process 
$\mathcal{E}(\alpha\bm{r}^*_{q,\sigma_{k,q}})$ is a positive martingale 
with mean one, which allows us to define the probability measure 
$\mathbb{W}_{\alpha\bm{r}^*_{q,\sigma_{k,q}}}$ as in 
\eqref{eq:defWv}. Following the notation in Section~\ref{sec:2}, we let 
$\widetilde{\E}_{\bm{r}^*_{q,\sigma_{k,q}}}$ denote expectation with 
respect to $\mathbb{W}_{\bm{r}^*_{q,\sigma_{k,q}}}$.

The proof of the upper bound proceeds through a sequence of intermediate steps.
\begin{lemma}
\label{lem:k_lemma}
There exists $\epsilon_0>0$ such that, if $\epsilon\in (0,\epsilon_0)$, then there exists $k>0$ such that
\begin{equation}
\E\!\left[\mathcal{E}^{\sigma_{k,q}}(\epsilon\bm{r}^*_q,1,\bm{\omega})\bm{1}_{\{\sigma_{k,q}<1\}}\right]\leq e^{-\epsilon k}.
\end{equation}
\end{lemma}

\begin{proof}
For $k>0$, set
\begin{equation}
f_q(k):=-\log\W^{\otimes n}(\sigma_{k,q}<1)=-\log\left(\W^{\otimes n}\left(\int^1_0dt\|\bm{r}(t,\bm{\omega})\|^2_2>k\right)\),
\end{equation}
with $-\log(0)=\infty$. Then $f_q(k)\ge 0$ and, since $\bm{r}^*\in D_n$, we have $\lim_{k\to\infty} f_q(k)=\infty$.  By the Cauchy--Schwarz inequality
\been{
\aled{
&\E[\mathcal{E}^{\sigma_{k,q}}(\epsilon\bm{r}^*_q,1,\bm{\omega})\bm{1}_{\{\sigma_{k,q}<1\}}]\\
&\leq \sqrt{\E\[\exp\(2\epsilon I(\bm{r}^*_q,\sigma_{k,q},\bm{\omega})-\epsilon^2\int^{\sigma_{k,q}}_qdt\|\bm{r}^*_q(t,\bm{\omega})\|^2_2\)\]}\sqrt{\E[\bm{1}_{\{\sigma_{k,q}<1\}}]}\\
&\leq \sqrt{\E\[\mathcal{E}^{\sigma_{k,q}}(2\epsilon \bm{r}^*_q,1,\bm{\omega})\exp\(\epsilon^2\int^{\sigma_{k,q}}_qdt\|\bm{r}^*_q(t,\bm{\omega})\|^2_2\)\]}\sqrt{\E[\bm{1}_{\{\sigma_{k,q}<1\}}]}\leq e^{\frac{1}{2}\epsilon^2k-\frac{1}{2}f_q(k)}.
}
}
Define
\begin{equation}
g_q(k):=\frac{-k+\sqrt{k^2+k f_q(k)}}{k}.
\end{equation}
Since $f_q(k)\to\infty$, there exists $k_0$ such that $g_q(k)>0$ for all $k>k_0$. Taking $\epsilon_0\leq \sup_{k>0}g_q(k)$, for any $\epsilon\in (0,\epsilon_0)$ one can find $k>k_0$ with $\epsilon\le g_q(k)$, which implies
\begin{equation}
\exp\!\Big(\tfrac{1}{2}\epsilon^2k-\tfrac{1}{2}f_q(k)\Big)\le e^{-\epsilon k}.
\end{equation}
\end{proof}

We now derive a preliminary upper bound. Although this estimate does not hold uniformly for all $q \in [q_0,1]$, it serves as a necessary intermediate step toward establishing the full result.

\begin{lemma}
\label{lem:martcond5}
If $(\phi^*,\bm{r}^*)\in S_n\times D_n$ is a solution to the BSDE~\eqref{selfEq1Aux} and
\been{
\E\[|\Psi(\bm{\omega})|\]<\infty,
}
If
\been{
\Psi(\bm{\omega})\leq C,\quad \Wnas
}
then for sufficiently small $\epsilon>0$ there exists $M_\epsilon>0$ such that
\been{
\label{eq:psibound00}
\phi^*(q,\bm{\omega})\leq C+M_{\epsilon},\quad \forall q\in [q_\epsilon,1],\,\Wnas.
}
\end{lemma}
\begin{proof}
Given $q\in [q_{\epsilon},1]$, we have
\been{
\aled{
&\mathcal{E}^{\epsilon \sigma_{k,q}}(\bm{r}^*_q,1,\bm{\omega})\exp\(-\bm{1}_{\{\sigma_{k,q}=1\}}\frac{\epsilon}{2}\int^{\sigma_{k,q}}_qdt(\mu([0,t])-\epsilon)\|\bm{r}^*(t,\bm{\omega})\|^2_2\)\\
&= e^{\epsilon \Psi^*(\bm{\omega})-\epsilon \phi^*(q,\bm{\omega})}\bm{1}_{\{\sigma_{k,q}=1\}}+\mathcal{E}^{\sigma_{k,q}}(\bm{r}^*_q,1,\bm{\omega})\bm{1}_{\{\sigma_{k,q}<1\}}.
}
}
Note that $\mu([0,t])-\epsilon\geq 0$ fr any $t\in [q_{\epsilon},1]$. Using the above equivalence and Lemma~\ref{lem:k_lemma}, for $\epsilon$ small enough there exist $k>0$ such that
\been{
\aled{
&\Eecond{\epsilon \bm{r}^*_{q,\sigma_{k,q}}}{\exp\(-\bm{1}_{\{\sigma_{k,q}=1\}}\frac{\epsilon}{2}\int^{\sigma_{k,q}}_qdt(\mu([0,t])-\epsilon)\|\bm{r}^*(t,\bm{\omega})\|^2_2\)}{q}\\
&\leq \Econd{e^{\epsilon \Psi^*(\bm{\omega})-\epsilon \phi^*(q,\bm{\omega})}}{q}+e^{-\epsilon k}
}
}
On the other hand, the Jensen inequality gives
\been{
\aled{
&\Eecond{\epsilon \bm{r}^*_{q,\sigma_{k,q}}}{\exp\(-\bm{1}_{\{\sigma_{k,q}=1\}}\frac{\epsilon}{2}\int^{\sigma_{k,q}}_qdt(\mu([0,t])-\epsilon)\|\bm{r}^*(t,\bm{\omega})\|^2_2\)}{q}\\
&\geq \exp\left(-\frac{\epsilon}{2}\Eecond{\epsilon \bm{r}^*_{q,\sigma_{k,q}}}{\bm{1}_{\{\sigma_{k,q}=1\}}\int^{\sigma_{k,q}}_qdt(\mu([0,t])-\epsilon)\|\bm{r}^*(t,\bm{\omega})\|^2_2}{q}\right)\geq e^{-\frac{\epsilon}{2}k}
}
}
Combining the above two inequalities, we get
\been{
\phi^*(q,\bm{\omega})\leq \frac{1}{\epsilon}\log\left(\Econd{e^{\epsilon \Psi(\bm{\omega})}}{q}\right)-\frac{1}{\epsilon}\log(e^{-\frac{\epsilon}{2}k}-e^{-\epsilon k}).
}
the term $\frac{1}{\epsilon}\log(e^{-\frac{\epsilon}{2}k}-e^{-\epsilon k})$ is a positive constant dependent on $\epsilon$. Moreover, since $\Psi(\bm{\omega})<C$ $\Wnas$, we have $\Econd{e^{\epsilon \Psi(\bm{\omega})}}{q}\leq e^{\epsilon C}$
\end{proof}
The next lemma eliminates the dependence on $\epsilon$ in the upper bound and establishes a uniform bound for $\phi^*$ on $[q_0,1]$.
\begin{lemma}
\label{lem:martcond6}
If $(\phi^*,\bm{r}^*)\in S_n\times D_n$ is a solution to the BSDE~\eqref{selfEq1Aux} and
\been{
\E\[|\Psi(\bm{\omega})|\]<\infty.
}
If
\been{
\Psi(\bm{\omega})\leq C,\quad \Wnas
}
then
\been{
\label{eq:psibound01}
\phi^*(q,\bm{\omega})\leq C,\quad \forall q\in [q_0,1],\,\Wnas.
}
\end{lemma}
\begin{proof}
If $q_0=1$, the result is trivial. If $q_0<1$, then for $\epsilon>0$ small enough $q_{\epsilon}<1$. Taking $q\in [q_{\epsilon},1]$ and $t\in [q,1]$, the inequality~\eqref{eq:Wineq3}, Lemma~\eqref{lem:martcond4}, and Lemma~\eqref{lem:martcond5} give
\been{
\gamma(\bm{r}^*_q,t,\bm{\omega})\leq \phi^*(t,\bm{\omega})- \phi^*(q,\bm{\omega})\leq 2C+M_{\epsilon}.
}
Hence the DDE $\mathcal{E}(\bm{r}^*_q)$ is bounded and thus it is a martingale on $[0,1]$. So we have
\been{
\Econd{e^{\Psi(\bm{\omega})- \phi^*(q,\bm{\omega})}}{q}\geq \Econd{\mathcal{E}(\bm{r}^*_q,1,\bm{\omega})}{q}=1.
}
Therefore we get
\been{
\phi^*(q,\bm{\omega})\leq \log\(\Econd{e^{\Psi(\bm{\omega})}}{q}\)\leq C.
}
for any $\epsilon>0$ small enough and $q\in [q_{\epsilon},1]$. Taking the limit $\epsilon\to 0$, we end the proof.
\end{proof}
The above lemmas prove almost all the statement of the Proposition~\ref{prop:uniqueness2}. In order to complete the proof, we have to extend the bound to all $q\in [0,1]$.
We first study the properties of the solution at $q\in [0,q_0]$.
\begin{lemma}
\label{lem:BSDEproof00}
If $(\phi^*,\bm{r}^*)\in S_n\times D_n$ is a solution to the BSDE~\eqref{selfEq1Aux} and
\been{
\label{eq:sol00}
\sup_{q\in [q_0,1]}|\phi^*(q,\bm{\omega})|\leq \|\Psi\|_{\infty},\quad \Wnas,
}
\been{
\sup_{q\in [q_0,1]}\left| \gamma(\mu\bm{r}^*_{q_0},q,\bm{\omega})\right|
\leq 2 \|\Psi\|_{\infty},\quad \Wnas,
}
and
\been{
\Ee{\mu\bm{r}^*_{q_0}}\[\int^1_{q_0}dq\mu([0,q])\,\|\bm{r}(q,\bm{\omega})\|^2_2\]\leq 4\|\Psi\|_{\infty}.
}
then Proposition~\ref{prop:uniqueness2} holds. Moreover
\been{
\label{eq:solq0}
\phi^*(q,\bm{\omega}):=\Econd{\phi^*( q_0,\,\bm{\omega}\,)}{q},\quad \forall q\in [0,q_0],
}
and $\bm{r}^*$ is the unique progressive process such that
\been{
\label{eq:solq01}
\phi^*(q,\bm{\omega})=\E\[\phi^*( q_0,\,\bm{\omega}\,)\]+I(\bm{r}^*,q,\bm{\omega}),\quad \forall q\in [0,q_0].
}
\end{lemma}
\begin{proof}
Since $\mu([0,q])=0$, then $\mu\bm{r}=0$ and the statements~\eqref{eq:uniqueness22} and~\eqref{eq:uniqueness222} of Proposition~\ref{prop:uniqueness2} hold trivially.  For $q\in [0,q_0]$, the \BSDE takes the form
\been{
\label{eq:runcated}
\phi(q,\bm{\omega})=-\int^{q_0}_q\bm{r}(q,\bm{\omega})\cdot d\bm{\omega}(t)+\phi^*( q_0,\,\bm{\omega}\,)=-I(\bm{r}_q,q_0,\bm{\omega})+\phi^*( q_0,\,\bm{\omega}\,).
}
By the Martingale Representation Theorem, there exists a unique process $\bm{r}^*\in D_n$ such that the martingale $\phi^*\in S^b_n$, defined in~\eqref{eq:solq0}, verifies~\eqref{eq:solq01}. Comparing~\eqref{eq:solq0},~\eqref{eq:solq01}, and~\eqref{eq:runcated}, we see that the pair $(\phi^*,\bm{r}^*)\in S_n\times D_n$, defined in this way, verifies the BSDE~\eqref{eq:runcated} for all $q\in [0,q_0]$. We have to show that it is the unique solution.

Let us suppose there exists a second pair $(\overline{\phi},\overline{\bm{r}})\in S_n\times D_n$ solving the BSDE~\eqref{eq:runcated} for $q\in [0,q_0]$. The two solutions verify
\been{
\phi^*(q,\bm{\omega})-\overline{\phi}(q,\bm{\omega})=I(\overline{\bm{r}},q,\bm{\omega})-I(\bm{r}^*,q,\bm{\omega})=I(\overline{\bm{r}}-\bm{r}^*,q,\bm{\omega}).
}
We denote
\been{
\delta\phi:=\phi^*-\overline{\phi},\quad \delta\bm{r}:=\bm{r}^*-\overline{\bm{r}}.
}
By the Ito Formula
\been{\label{eq:Itophi2}
-\delta\phi^2(q,\bm{\omega})=\delta\phi^2(q_0,\bm{\omega})-\delta\phi^2(q,\bm{\omega})=2I(\delta\phi\delta\bm{r}_q,q_0,\bm{\omega})+\int^{q_0}_qds\|\delta\bm{r}(s,\bm{\omega})\|^2_2,
}
with
\been{
I(\delta\phi\delta\bm{r}_q,q_0,\bm{\omega}):=\int^{q_0}_q\delta\phi(q,\bm{\omega})\delta\bm{r}(s,\bm{\omega})\cdot d\bm{\omega}(s).
}
Let $(\tau_k)_{k\in \N}$ the following increasing sequence of stopping time
\been{
\tau_k:=\inf\left\{t\in [0,q_0];\,\max\left\{\int^t_0ds\phi^2(s,\bm{\omega})\|\delta\bm{r}(s,\bm{\omega})\|^2_2,\int^t_0ds\|\delta\bm{r}(s,\bm{\omega})\|^2_2\right\}\geq k\right\}.
}
with $\inf \emptyset=q_0$. Since $\delta\bm{r}\in D_n$ and $\delta\phi\in S_n$, $\tau_{k}$ converges to $q_0$ $\Wnas$ as $k\to \infty$. The processes $I^{\tau_k}(\delta\bm{r}_q)$ and $I^{\tau_k}(\delta\phi\delta\bm{r}_q)$ are square integrable martingale with mean $0$. Thus the equation~\eqref{eq:Itophi2} gives
\been{
\aled{
\E\[(\delta\phi(t\wedge \tau_k,\bm{\omega}))^2\]+\E\[\int^{t\wedge\tau_k}_qds\|\delta\bm{r}(s,\bm{\omega})\|^2_2\]=\E\[I(\delta\phi\delta\bm{r}_q,q_0,\bm{\omega})\]=0.
}
}
Thus, taking the limit $k\to \infty$, we get that
\been{
\delta\phi(q,\bm{\omega})=0,\, \forall q\in [0,q_0],\quad \textup{and}\quad  \int^{q_0}_0ds\|\delta\bm{r}(s,\bm{\omega})\|^2_2=0,\quad  \Wnas.
}So $\overline{\phi}=\phi^*$ and $\bm{r}^*=\overline{\bm{r}}$ (in the sense that they are in the same equivalence class as defined in Section~\ref{sec:3}). Finally, if $\phi^*$ verifies~\eqref{eq:sol00} and~\eqref{eq:solq0}, then it verifies~\eqref{eq:uniqueness2}.
\end{proof}
Combining the above results, we prove Proposition~\ref{prop:uniqueness2}.
\begin{proof}[Proof of Proposition~\ref{prop:uniqueness2}]
By Lemma~\ref{lem:martcond4} and Lemma~\ref{lem:martcond5}, if $|\Psi(\bm{\omega)|}\leq C$ $\W^{\otimes n}-\textup{a.s.}$, then
\been{
\label{eq:start}
\sup_{q\in [q_0,1]}|\phi^*(q,\bm{\omega})|\leq C,\quad \W^{\otimes n}-\textup{a.s.}.
}
 We now prove the statement~\eqref{eq:uniqueness22} for $q\in [q_0,1]$. The differential notation of the \BSDE gives
\been{
d\phi^*(q,\bm{\omega})=\bm{r}^*(q,\bm{\omega})\cdot d\bm{\omega}(q) -\frac{1}{2}\mu([0,q])\|\bm{r}^*(q,\bm{\omega})\|^2_2dt.
}
Thus, the integration by part formula gives
\been{
\label{eq:zeta_part}
\aled{
&\gamma(\mu\bm{r}^*,t,\bm{\omega})=\int^{t}_{q_0}\mu([0,s])\(\bm{r}^*(s,\bm{\omega})\cdot d\bm{\omega}(s) -\frac{1}{2}\mu([0,s])\|\bm{r}^*(s,\bm{\omega})\|^2_2ds\)\\
&=\int^{t}_{q_0}\mu([0,s])d\phi^*(s,\bm{\omega})=\mu([0,t])\phi^*(t,\bm{\omega})-\int^{t}_{q}\mu(ds)\phi^*(s,\bm{\omega})\\
&=\int^{t}_{q_0}\mu(ds)\(\phi^*(t,\bm{\omega})-\phi^*(s,\bm{\omega})\).
}
}
Since $\phi^*(t,\bm{\omega})\leq \|\Psi\|_{\infty}$, the above representation gives $\sup_{t\in [q_0,1]} \gamma(\mu\bm{r}^*,t,\bm{\omega})\leq 2 \|\Psi\|_{\infty}$ proving the statement~\eqref{eq:uniqueness22}. 

We deduce that $\mathcal{E}(\mu\bm{r}^*)$ is bounded, and then it is a martingale. Moreover, since $\phi^*$ is bounded as well. So the martingale property gives
\been{
\aled{
&\E\[\int^1_{q_0}dt\mathcal{E}(\mu\bm{r}^*,t,\bm{\omega})\mu([0,t])\|\bm{r}^*(t,\bm{\omega})\|^2_2\]=\E\[\int^1_{q_0}dt\mathcal{E}(\mu\bm{r}^*,t,\bm{\omega})\mu([0,t])\|\bm{r}^*(t,\bm{\omega})\|^2_2\]\\
&=2\E\[\mathcal{E}(\mu\bm{r}^*,t,\bm{\omega})\(\phi^*(1,\bm{\omega})-\phi^*(q_0,\bm{\omega})\)\]\leq 4\|\Psi\|_{\infty}.
}
}
proving the statement~\eqref{eq:uniqueness222}. The fact that $\mathcal{E}(\mu\bm{r}^*)$ is a $\Wn$--martingale adapted to $\Fqfilt{q}{[0,1]}$ is a consequence of the above inequality and Theorem~\ref{lem:martcond}.

The above results and Lemma~\ref{lem:BSDEproof00} complete the proof.

\end{proof}
\begin{proof}[Proof of Corollary~\ref{rem:uniqueness222}]
For any $q\in [0,1]$ and $k>0$, let $\sigma_{k,q}$ be the stopping time defined in~\eqref{eq:stopping}. The DDE $\mathcal{E}(\frac{1}{2}\mu \bm{r}^*_{q,\sigma_{k,q}})$ is a $\Wn$--martingale. Hence, by the \BSDE and the bound~\eqref{eq:uniqueness2}:
\been{
I_{\frac{1}{2}\mu \bm{r}^*_{q,\sigma_{k,q}}}(\bm{r}^*_{q,\sigma_{k,q}},q',\bm{\omega})=\left(\phi^*(q'\wedge \sigma_{k,0},\bm{\omega})-\phi^*(q,\bm{\omega})\)\bm{1}_{q'\geq q}.
}
is a martingale with respect to the probability measure $\W_{\frac{1}{2}\mu \bm{r}^*_{0,\sigma_{k,0}}}$ bounded by $2\|\Psi\|_{\infty}$. Thus 
\been{
\label{eq:BDG0}
\aled{
&\Eecond{\frac{1}{2}\mu \bm{r}^*_{q,\sigma_{k,q}}}{\int^{\sigma_{k,q}}_qdt\|\bm{r}^*(t,\bm{\omega})\|^2_2}{q}\\
&=\Eecond{\frac{1}{2}\mu \bm{r}^*_{q,\sigma_{k,q}}}{\left(I_{\frac{1}{2}\mu \bm{r}^*_{q,\sigma_{k,q}}}(\bm{r}^*_{q,\sigma_{k,q}},q',\bm{\omega})\)^2}{q}\leq 4\|\Psi\|^2_{\infty}.
}
}
A straightforward computation gives
\begin{equation}
\label{eq:trivialzeta}
\gamma(\mu\bm{r}^*,\tau_m,\bm{\omega})\leq \gamma\left(\tfrac{1}{2}\mu\bm{r}^*,\tau_{m},\bm{\omega}\right)+\tfrac{1}{2}\gamma(\mu\bm{r}^*,\tau_m,\bm{\omega}).
\end{equation}
Thus, using the inequality~\eqref{eq:uniqueness22} and the Proposition~\ref{prop:uniqueness2} we get 
\been{
\label{eq:EE1}
\mathcal{E}(\mu \bm{r}^*_{0,\sigma_{k,0}},q,\bm{\omega})\leq e^{\|\Psi\|_{\infty}}\mathcal{E}\(\frac{1}{2}\mu \bm{r}^*_{0,\sigma_{k,0}},q,\bm{\omega}\)
}
and
\been{
\label{eq:EE2}
1=\frac{\mathcal{E}(\mu \bm{r}^*_{0,\sigma_{k,0}},q,\bm{\omega})}{\mathcal{E}(\mu \bm{r}^*_{0,\sigma_{k,0}},q,\bm{\omega})}\leq e^{2\|\Psi\|_{\infty}}\mathcal{E}(\mu \bm{r}^*_{0,\sigma_{k,0}},q,\bm{\omega})\leq e^{3\|\Psi\|_{\infty}}\mathcal{E}\(\frac{1}{2}\mu \bm{r}^*_{0,\sigma_{k,0}},q,\bm{\omega}\).
}
Taking the limit $k\to \infty$ and using the Fatou Lemma, the inequalities~\eqref{eq:EE1} and~\eqref{eq:BDG0} prove~\eqref{bound_rr3} and the inequalities~\eqref{eq:EE2} and~\eqref{eq:BDG0} prove~\eqref{bound_rr33}.

Moreover, taking $q=0$, by the Burkholder-Davis-Gundy inequality, for any $p\geq 1$ there exists a constant $C_p>0$ such that
\been{
\label{eq:BDG}
\aled{
&\Ee{\frac{1}{2}\mu \bm{r}^*_{0,\sigma_{k,0}}}\[\(\int^{\sigma_{k,0}}_{0}dt\|\bm{r}^*(t,\bm{\omega})\|^2_2\)^p\]\\
&\leq C_p\Ee{\frac{1}{2}\mu \bm{r}^*_{0,\sigma_{k,0}}}\[\sup_{q\in [0,1]}\(I_{\frac{1}{2}\mu \bm{r}^*_{0,\sigma_{k,0}}}(\bm{r}^*_{0,\sigma_{k,0}},q,\bm{\omega})\)^p\]\leq C_p\|\Psi\|^p_{\infty}.
}
}
So, taking the limit $k\to \infty$ and using the inequalities~\eqref{eq:EE1} and~\eqref{eq:EE2} as before, we prove~\eqref{bound_rr} and~\eqref{bound_rr2}.
\end{proof}

\section{Existence of the finite-RSB solution}
\label{sec:41}
In this section, we start the proof of the existence of the solution to the \BSDE.

Here, we explicitly derive the solution in the case where the Parisi parameter $\mu\in \textup{Pr}([0,1])$ is a discrete measure with a finite number of atoms. 

Consider two increasing sequences of $K+2 \in \mathbb{N}$ numbers $q_0,\cdots, q_{K+1}$ and $ x_0,\cdots, x_{K+1}$ with
\begin{equation}
\label{eq:q_seq}
0=q_0\leq q_1\leq\cdots\leq q_K\leq q_{K+1}=1
\end{equation}
and
\begin{equation}
\label{eq:x_seq}
0=x_0< x_1\leq\cdots\leq x_K\leq x_{K+1}=1.
\end{equation}
The number $x_1$ must be non-zero. Let us consider the $\mu \in \textup{Pr}([0,1]) $ defined as
\begin{equation}
\label{discrete}
\mu([0,q]):=\sum_{l\in [K+1]} x_l\bm{1}_{q\in [q_{l-1},q_{l})}+\bm{1}_{q=1}.
\end{equation}
We denote by $\textup{Pr}([0,1])^{\circ}\subset\textup{Pr}([0,1])$ the space of probability measure of this form, for any $K\in \mathbb{N}$. Let us define a sequence of random variables $(\Xi_{l})_{l\in [K+1]_0}$ such that
\been{
\label{eq:rec10}
\Xi_{K+1}(\bm{\omega}):=e^{\Psi(\bm{\omega})},
}
and the remaining $(\Xi_{l})_{l\in [K]}$ are defined via the following backward recursion:
\been{
\label{eq:rec1}
\Xi_{l}(\bm{\omega}):=\Econd{\(\Xi_{l+1}(\bm{\omega})\)^{\frac{x_{l}}{x_{l+1}}}}{q_{l}},\quad \forall l\in [K].
}
Note that each $\Xi_l$ is $\Fq{q_{l}}$--measurable. Moreover, if $\Psi\in B_n$ and $x_{l}\leq x_{l+1}$, then $\Xi_{l}\in B_n$, with
\been{
\label{eq:boundXi}
|\log\(\Xi_{l}(\bm{\omega})\)|\leq  \|\Psi\|_{\infty},\quad \Wnas,\,\forall l\in [K+1].
}
For each $l\in [K+1]$, let us consider the adapted process $\xi_l$ defined as
\been{
\label{eq:rec2}
\xi_{l}(q,\bm{\omega}):=\Econd{\(\Xi_{l}(\bm{\omega})\)^{\frac{x_{l-1}}{x_{l}}}}{q}.
}
Note that, for any  $l\in [K]$
\been{
\label{eq:left_cont}
\xi_{l}(q_{l-1},\bm{\omega})=\Xi_{l-1}(\bm{\omega}),\quad \xi_{l}(q_{l},\bm{\omega})=\(\Xi_l(\bm{\omega})\)^{\frac{x_l}{x_{l+1}}}=\(\xi_{l+1}(q_l,\bm{\omega})\)^{\frac{x_l}{x_{l+1}}}.
}
For a fixed $l\in[K]$, the process $\xi_{l}$ is a bounded $\Wn$--martingale with
\been{
\label{eq:bound_xi}
\sup_{q\in [0,1]}|\log\(\xi_{l}(q,\bm{\omega})\)|\leq \|\Psi\|_{\infty},\quad \Wnas.
}
By the Martingale Representation Theorem, for each $l\in [K]$, there exists a progressive process $\bm{m}_l\in D_n$ such that for any $q\in [q_{l-1},1]$ it holds
\been{
\label{eq:ximart}
\aled{
&\xi_{l}(q,\bm{\omega})-\xi_{l}(q_{l-1},\bm{\omega})=\int^q_{q_{l-1}}\bm{m}_l(t,\bm{\omega})\cdot d\bm{\omega}(t).
}
}
From the previous setting, we now construct the pair of processes $(\phi,\bm{r})\in S_n\times D_n$, where
\been{
 \label{discrete_equation0}
 \aled{
\phi(q,\bm{\omega}):=\sum_{l\in [K+1]}\frac{1}{x_{l}}\log\(\xi_{l}(q,\bm{\omega})\)\bm{1}_{q\in [q_{l-1},q_{l})}+\Psi(\bm{\omega})\bm{1}_{q=1},
}
}
and
\been{
\label{auxiliry_Order_piecewise0}
\bm{r}(q,\bm{\omega}):=\sum_{l\in [K+1]}\frac{1}{x_{l}\xi_{l}(q,\bm{\omega})}\bm{m}_{l}(q,\bm{\omega})\bm{1}_{q\in [q_{l-1},q_l)}.
}
Since $(\xi_{l})_{l\in [K+1]}$ is a sequence of bounded process and different from $0$ $\Wnas$ and $\{x_k\}_{l\in [K+1]}$ is a sequence of strictly positive numbers, $I(\bm{r})$ is a $\Wn$--martingale.

Note that the above iteration is equivalent to the finite--RSB iteration given by equations $21$ and $22$ in \cite{MyPaper}.

We will prove that the pair $(\phi,\bm{r})$  is a solution to the \BSDE. We first show that  the process $\phi$ has almost surely continuous sample paths.
\begin{lemma}
\label{inequality_prop_theo}
Let $\mu\in \textup{Pr}([0,1])^{\circ}$ as in~\eqref{discrete} and $\Psi\in B_n$. The process $\phi$, defined in~\eqref{discrete_equation0} is $\Fqfilt{q}{[0,1]}$--adapted and has almost surely continuous sample paths.
\end{lemma}
\begin{proof}
The process $\phi$ is obviously adapted and continuous in $[0,1)\setminus \{q_1,\cdots,q_K\}$. The continuity is provided by the equalities~\eqref{eq:left_cont},
\been{
\aled{
&\lim_{q\uparrow q_{l}}\phi(q,\bm{\omega})=\frac{1}{x_{l}}\log\(\xi_{l}(q_l,\bm{\omega})\)=\frac{1}{x_{l+1}}\log\(\xi_{l+1}(q_{l+1},\bm{\omega})\)=\phi(q_{l},\bm{\omega}),\quad \forall l\in [K],
}
}
and
\been{
\lim_{q\uparrow 1}\phi(q,\bm{\omega})=\lim_{q\uparrow 1}\frac{1}{x_{K+1}}\log\(\xi_{K+1}(q,\bm{\omega})\)=\log\(\Xi_{K+1}(\bm{\omega})\)=\Psi(\bm{\omega}).
}
\end{proof}
We can now prove that $(\phi,\bm{r})$ is a solution to the \BSDE when $\mu$ is the form~\eqref{discrete}.
\begin{proposition}[finite--RSB solution]
\label{prop:existk}
By definition
\been{
\phi(q,\bm{\omega}))-\phi(q_{l-1},\bm{\omega}))=\frac{1}{x_{l}}\log(\xi_{l}(q,\bm{\omega}))-\frac{1}{x_{l}}\log(\xi_{l}(q,\bm{\omega}))
}
Let $\mu\in \textup{Pr}([0,1])^{\circ}$ as in~\eqref{discrete} and $\Psi\in B_n$. The pair of processes $(\phi,\bm{r})\in S_n\times D_n$, given by~\eqref{discrete_equation0} and~\eqref{auxiliry_Order_piecewise0}, is a solution to the \BSDE.
\end{proposition}
\begin{proof}
We proceed by induction over $l\in [K]$. Let us assume that the pair $(\phi,\bm{r})$ verifies the \BSDE for any $q\in [q_{l},1]$. Using the martingale representation~\eqref{eq:ximart} and the Ito Formula, for any $q_{l}\leq q\leq q'<q_l$ we get
\been{
\aled{
&\frac{1}{x_{l}}\(\log(\xi_{l}(q',\bm{\omega}))-\log(\xi_{l}(q,\bm{\omega}))\)=\int^{q'}_{q}\frac{\bm{m}_k(t,\bm{\omega})\cdot d\bm{\omega}(t)}{x_{l}\xi_{l}(t,\bm{\omega})}
-\frac{1}{2}\int^{q'}_{q}dt\frac{\|\bm{m}_k(t,\bm{\omega})\|^2_2}{x_{l}\xi^2_{l}(t,\bm{\omega})}.
}
}
By the definition~\eqref{discrete_equation0} and Lemma~\ref{inequality_prop_theo}
\been{
\lim_{q'\to q_l}\frac{1}{x_{l}}\(\log(\xi_{l}(q',\bm{\omega}))-\log(\xi_{l}(q,\bm{\omega}))\)=\lim_{q'\to q_l}\(\phi(q',\bm{\omega})-\phi(q,\bm{\omega})\)=\phi(q_l,\bm{\omega})-\phi(q,\bm{\omega})
}
and by the definition~\eqref{auxiliry_Order_piecewise0}
\been{
\int^{q'}_{q}\frac{\bm{m}_l(t,\bm{\omega})\cdot d\bm{\omega}(t)}{x_{l}\xi_{l}(t,\bm{\omega})}=\int^{q'}_{q}\bm{r}(t,\bm{\omega})\cdot d\bm{\omega}(t),
}
and
\been{
\int^{q'}_{q}dt\frac{\|\bm{m}_k(t,\bm{\omega})\|^2_2}{x_{l}\xi^2_{l}(t,\bm{\omega})}=\int^{q'}_{q}dt\mu([0,t))\|\bm{r}(t,\bm{\omega})\|^2_2
}
where we used the fact that $x_{l}=\mu([0,t])$ for any $t\in [q_{l-1},q_{l})$. Thus  the pair $(\phi,\bm{r})$ verifies the \BSDE for any $q\in [q_{l-1},1]$, proving the induction hypothesis.
\end{proof}
We finish the section by providing a second iterative representation of the process $\phi$, that will be useful in the following.
\begin{lemma}
\label{lem:rec3}
Let $\mu\in \textup{Pr}([0,1])^{\circ}$ as in~\eqref{discrete} and $\Psi\in B_n$. The process $\phi\in S^b_n$, given by~\eqref{discrete_equation0}, verifies the following recursion:
\been{
\label{eq:rec3}
\phi(q,\bm{\omega})=\frac{1}{\mu([0,q])} \log\,\Econd{\exp \left(\,\mu([0,q])\phi(\mu^{(t)},\,q_{k_q+1}\,,\bm{\omega}\,)\,\,\right)}{q},\quad \forall q\in [0,1)
}
where
\been{
\label{eq:kq}
k_q=\max\{l\in [K]_0;\,q_{l}\leq q\}.
}
and
\been{
\label{eq:rec3_1}
\phi(1,\bm{\omega})=\Psi(\bm{\omega}).
}
Moreover
\been{
\label{eq:rec3exp}
\log\mathcal{E}\(\mu\bm{r}_q,q_{k_q+1},\bm{\omega}\)=\mu([0,q])(\phi(q_{k_q+1},\bm{\omega})-\phi(q,\bm{\omega})).
}
\end{lemma}
\begin{proof}
Since
\been{
q\in [q_{k_q},q_{k_q+1}),
}
it follows that
\been{
\label{eq:rec330}
\phi(q,\bm{\omega})=\frac{1}{x_{k_q+1}}\log\(\xi_{k_q+1}(q,\bm{\omega})\)=\frac{1}{x_{k_q+1}}\log\( \Econd{\(\Xi_{k_q+1}(\bm{\omega})\)^{\frac{x_{k_q}}{x_{k_q+1}}}}{q}\),
}
If $k_q\in [K-1]_0$, by~\eqref{eq:left_cont}
\been{
\label{eq:rec33}
\(\Xi_{k_q+1}(\bm{\omega})\)^{\frac{x_{k_q}}{x_{k_q+1}}}=\xi_{k_q+1}(q_{k_q+1},\bm{\omega})=e^{x_{{k_q}+1}\phi(q_{k_q+1},\bm{\omega})}.
}
Since $x_{k_q+1}=\mu([0,q])$, combining~\eqref{eq:rec330} and~\eqref{eq:rec33}, we prove~\eqref{eq:rec3}.
 Finally, using the \BSDE we get
\been{
\aled{
&\mu([0,q])(\phi(q_{k_q+1},\bm{\omega})-\phi(q,\bm{\omega}))\\
&=x_{k_q+1}I(\bm{r}_q,q_{k_q+1},\bm{\omega})-\frac{x_{k_q+1}}{2}\int^{q_{k_q+1}}_{q}dtx_{k_q+1}\|\bm{r}(t,\bm{\omega})\|^2_2\\
&=\gamma(x_{k_q+1}\bm{r}_q,q_{k_q+1},\bm{\omega})=\gamma(\mu\bm{r}_q,q_{k_q+1},\bm{\omega}),
}
}
proving the last statement.
\end{proof}

\section{Existence of the full-RSB solution}
\label{sec:42}
Proposition~\ref{prop:existk} provides an existence result for a very specific class of Parisi parameters. In this section, we prove the existence of the solution to the BSDE~\eqref{selfEq1Aux} when the Parisi parameter is a generic probability measure $\mu\in \textup{Pr}([0,1])$. 

Through the section, we keep $\Psi\in B_n$ fixed, and denote by $(\phi(\mu),\bm{r}(\mu) )$ the solution to the BSDE~\eqref{selfEq1Aux} corresponding to the Parisi parameter $\mu\in\textup{Pr}([0,1])$. Given $\mu \in \textup{Pr}([0,1])$, we will consider a proper sequence of Parisi parameters $(\mu^{(k)})_{k\in \N}\subset \textup{Pr}([0,1])^{\circ}$ converging in $\|\cdot\|_{\infty}$ norm to $\mu$. Hence, we show that the sequence of corresponding solution $((\phi(\mu^{(k)}),\bm{r}(\mu^{(k)})))_{k\in\N}$ converges to a pair of processes $(\phi,\bm{r})\in S_n\times D_n$ that solves the \BSDE corresponding to the Parisi parameter $\mu$. 
\begin{proposition}
\label{convergence_result}
Given a Parisi parameter $\mu\in\textup{Pr}([0,1])$, consider a sequence of piecewise constant Parisi parameters $(\mu^{(k)})_{k\in \N}\subset \textup{Pr}([0,1])^{\circ}$, where
\begin{equation}
\|\mu^{(k)}-\mu\|_{\infty}\leq 2^{-k}
\end{equation}
The sequence of the solutions $\left(\,\phi(\mu^{(k)}),\bm{r}(\mu^{(k)})\,\right)$ converges almost surely and in $S^p_n \times D^{2,p}_n$ norm to a pair $(\phi,\bm{r})$ that is a solution to the BSDE~\eqref{selfEq1Aux} with Parisi parameter $\mu\in \textup{Pr}([0,1])$.
\end{proposition}
The proof of the above proposition relies on several useful properties of the finite--RSB solution. Note that, since the pair of processes $(\phi(\mu^{(k)}),\bm{r}(\mu^{(k)}))\in  S_n\times D_n$ defined in~\eqref{discrete_equation0} and~\eqref{auxiliry_Order_piecewise0} are solutions to the \BSDE, they verify Propositions~\ref{prop:uniqueness2} and Corollary~\ref{prop:uniqueness2}.

We start by compering the two processes $\phi(\mu^{(1)})$ and $\phi(\mu^{(2)})$, corresponding to the piecewise constant Parisi parameters $\mu^{(1)}$ and $\mu^{(2)}$. 

Let
\begin{equation}
\label{eq:deltamut}
\delta \mu=\mu^{(2)}-\mu^{(1)}\in\textup{Meas}([0,1]),
\end{equation}
and 
\begin{equation}
\label{eq:mut}
\mu^{(t)}=(1-t)\mu^{(1)}+t \mu^{(2)}\in \textup{Pr}([0,1])^{\circ}.
\end{equation}
The functions $q\mapsto \delta \mu([0,q]) $ and $q\mapsto \mu^{(t)}([0,q])$ are piecewise constant.

Let $0= q_0<q_1< \cdots q_K<q_{K+1}=1$ be the union of the discontinuity points of the two functions. We can write
\been{
\delta \mu([0,q])=\sum_{k\in [K+1]}\delta x_{k}\bm{1}_{q\in [q_{k-1},q_{k})},
}
and
\been{
\label{equation_Discretemut}
\mu^{(t)}([0,q])=\sum_{k\in [K+1]} x^{(t)}_k\bm{1}_{q\in [q_{k-1},q_{k})}+\bm{1}_{q=1}.
}
Since $\mu^{(t)}\in \textup{Pr}([0,1])^{\circ}$ we can define the pair $(\phi(\mu^{(t)}),\bm{r}(\mu^{(t)}))\in S_n\times D_n$ solution of the \BSDE, corresponding to the Parisi parameter $\mu^{(t)}$. By Proposition~\ref{prop:uniqueness} and Corollary~\ref{prop:uniqueness2}, the DDE $\mathcal{E}(\mu^{(t)}\bm{r}(\mu^{(t)}))$ is a $\Wn$--martingale and $I_{\mu^{(t)}\bm{r}(\mu^{(t)})}(\bm{r}(\mu^{(t)}))$ is a $\W_{\mu^{(t)}\bm{r}(\mu^{(t)})}$--square integrable martingale.

Given $q\in [0,1]$, let $k_{q}\in [K+1]$ be the integer defined as in \eqref{eq:kq}.
\begin{lemma}
\label{derivative_theo}
Consider two Parisi parameters $\mu^{(1)}$ and $\mu^{(2)}$ in $\textup{Pr}([0,1])^{\circ}$. 
Let $(\phi(\mu^{(t)}),\bm{r}(\mu^{(t)}))$ be the solution corresponding to the Parisi parameter $\mu^{(t)}$. Then, for all $q\in [0,1]$ and $t\in [0,1]$  the quantity $\phi(\mu^{(t)},q,\bm{\omega})$ is $\Wn$--almost surely derivable on $t$ and
\begin{equation}
\label{first_derivative}
\frac{\partial \phi(\mu^{(t)},q,\bm{\omega})}{\partial t}=\frac{1}{2}\widetilde{\mathbb{E}}_{\mu^{(t)}\bm{r}(\mu^{(t)})}\left[\int^{1}_{q}dp\,\delta \mu([0,p])\,\|\bm{r}(\mu^{(t)},p,\bm{\omega})\|^2_2\Bigg| \mathcal{F}^{\otimes n}_q\right].
\end{equation}
\end{lemma}
\begin{proof}
We proceed by induction.  At $q=1$, the random variable $\phi(\mu^{(t)},1,\Cdot)$ is independent of $t$, hence
\begin{equation}
\label{recursion_derivative_start}
\frac{\partial}{\partial t}\phi(\mu^{(t)},1,\bm{\omega})=\frac{\partial}{\partial t}\Psi(\bm{\omega})=0.
\end{equation}
Now fix $q<1$ and assume that \eqref{first_derivative} holds for all $p\in [q_{k_q},1]$. We must show that it also holds at $q$. Differentiating the right member of the recursion~\eqref{eq:rec3} with respect to $t$ we get
\begin{align}
\frac{\partial}{\partial t}{\phi(\mu^{(t)})}(q,\bm{\omega})&=-\frac{\delta \mu([0,q])}{(\mu^{(t)}([0,q]))^2} \log\,\mathbb{E}\left[\exp \left(\,\mu^{(t)}([0,q])\phi(\mu^{(t)},\,q_{k_q+1}\,,\bm{\omega}\,)\,\,\right)\big|\mathcal{F}^{\otimes n}_{q}\right]\label{eq:termII}\\
&+\frac{\delta \mu([0,q])}{\mu^{(t)}([0,q])}\frac{\mathbb{E}\left[\exp \left(\,\phi(\mu^{(t)},\,q_{k_q+1}\,,\bm{\omega}\,)\,\,\right)\phi(\mu^{(t)},\,q_{k_q+1}\,,\bm{\omega}\,)\big|\mathcal{F}^{\otimes n}_{q}\right]}{\mathbb{E}\left[\exp \left(\,\mu^{(t)}([0,q])\phi(\mu^{(t)},\,q_{k_q+1}\,,\bm{\omega}\,)\,\,\right) \big|\mathcal{F}^{\otimes n}_{q}\right]}\label{eq:termI}\\
&+\frac{\mathbb{E}\left[\exp \left(\,\mu^{(t)}([0,q])\phi(\mu^{(t)},\,q_{k_q+1}\,,\bm{\omega}\,)\,\,\right)\tfrac{\partial}{\partial t}\phi(\mu^{(t)},\,q_{k_q+1}\,,\bm{\omega}\,)\big|\mathcal{F}^{\otimes n}_{q}\right]}{\mathbb{E}\left[\exp \left(\,\mu^{(t)}([0,q])\phi(\mu^{(t)},\,q_{k_q+1}\,,\bm{\omega}\,)\,\,\right) \big|\mathcal{F}^{\otimes n}_{q}\right]}\label{eq:termIII}.
\end{align}
Equation~\eqref{eq:rec3} gives
\begin{equation}
\label{eq:rel0}
\eqref{eq:termII}=-\frac{\delta \mu([0,q])}{\mu^{(t)}([0,q])}\phi\(\mu^{(t)},\,q\,,\bm{\omega}\,\right)
=-\frac{\delta \mu([0,q])}{\mu^{(t)}([0,q])}\widetilde{\mathbb{E}}_{\mu^{(t)}\bm{r}(\mu^{(t)})}\big[\phi(\mu^{(t)},q,\bm{\omega})\big| \mathcal{F}^{\otimes n}_q\big],
\end{equation}
where, in the last equality, we used the fact that $\phi(\mu^{(t)},q,\cdot)$ is $\Fq{q}$--measurable.

Combining the above equalities~\eqref{eq:rec3} and~\eqref{eq:rec3exp}, for any $\Fq{1}$--measurable random variable $A$ we get
\been{
\label{eq:rel1111}
\aled{
&\frac{\Econd{\exp \left(\,\mu^{(t)}([0,q])\phi(\mu^{(t)},\,q_{k_q+1}\,,\bm{\omega}\,)\right)A(\bm{\omega})\,\,}{q}}{\Econd{\exp \left(\,\mu^{(t)}([0,q])\phi(\mu^{(t)},\,q_{k_q+1}\,,\bm{\omega}\,)\,\,\right)}{q}}\\
&=\Econd{\exp \left(\,\mu^{(t)}([0,q])(\phi(\mu^{(t)},\,q_{k_q+1}\,,\bm{\omega}\,)-\phi(\mu^{(t)},\,q\,,\bm{\omega}\,)\right)A(\bm{\omega})\,\,}{q}\\
&=\Econd{\exp \left(\,\gamma(\mu^{(t)}\bm{r}_q(\mu^{(t)}),q_{k_q+1}\,,\bm{\omega})\right)A(\bm{\omega})\,\,}{q}=\Eecond{\mu^{(t)}\bm{r}(\mu^{(t)})}{A(\bm{\omega})}{q}.
}
}
So
\been{
\label{eq:rel1}
\eqref{eq:termI}=\frac{\delta \mu([0,q])}{\mu^{(t)}([0,q])}\Eecond{\mu^{(t)}\bm{r}(\mu^{(t)})}{\phi(\mu^{(t)},q_{k_q+1},\bm{\omega})}{q}
}
and
\been{
\label{eq:rel2}
\aled{
\eqref{eq:termIII}&=\Eecond{\mu^{(t)}\bm{r}(\mu^{(t)})}{\frac{\partial \phi(\mu^{(t)},q_{k_q+1},\bm{\omega})}{\partial t}}{q}\\
&=\frac{1}{2}\widetilde{\mathbb{E}}_{\mu^{(t)}\bm{r}(\mu^{(t)})}\left[\int^{1}_{q_{k_q+1}}dp\,\delta \mu([0,p])\,\|\bm{r}(\mu^{(t)},p,\bm{\omega})\|^2_2\Bigg| \mathcal{F}^{\otimes n}_q\right],
}
}
where the last equality follows from the induction hypothesis.

Since $\mu^{(t)}([0,p])=x_{k+1}^{(t)}$ and $\delta\mu^{(t)}([0,p])=\delta x_{k+1}$ for all $p\in [q,q_{k_q+1}]$, then combining \eqref{eq:rel0} and \eqref{eq:rel1}
\begin{equation}
\label{eq:rel3}
\aled{
&\eqref{eq:termI}+\eqref{eq:termII}\\
&=\frac{\delta x_{k+1}}{x_{k+1}^{(t)}}\widetilde{\mathbb{E}}_{\mu^{(t)}\bm{r}(\mu^{(t)})}\Bigg[\int^{q_{k_q+1}}_{q_k} d\bm{\omega}(q') \cdot\bm{r}(q',\bm{\omega})-\frac{x_{k+1}^{(t)}}{2}\int^{q_{k_q+1}}_{q_k} dq' \,\|\bm{r}(q',\bm{\omega})\|^2_2\Bigg| \Fq{q}\Bigg]\\
&=\frac{1}{2}\widetilde{\mathbb{E}}_{\mu^{(t)}\bm{r}(\mu^{(t)})}\left[\int^{q_{k_q+1}}_{q} dp\delta\mu([0,p]) \,\|\bm{r}(p,\bm{\omega})\|^2_2\Bigg| \Fq{q}\right],
}
\end{equation}
Thus
\been{
\frac{\partial \phi(\mu^{(t)},q,\bm{\omega})}{\partial t}=\eqref{eq:termII}+\eqref{eq:termI}+\eqref{eq:termIII}=\frac{1}{2}\widetilde{\mathbb{E}}_{\mu^{(t)}\bm{r}(\mu^{(t)})}\left[\int^{1}_{q} dp\delta\mu([0,p]) \,\|\bm{r}(p,\bm{\omega})\|^2_2\Bigg| \Fq{q}\right],
}
concluding the proof.
\end{proof}
In the following, we consider two fixed Parisi parameters $\mu^{(1)}$ and $\mu^{(2)}$ in $\textup{Pr}([0,1])^{\circ}$ and define
\begin{equation}
\label{eq:phirt}
 \delta\phi=\phi(\mu^{(2)})-\phi(\mu^{(1)}),\quad \delta\bm{r}=\bm{r}(\mu^{(2)})-\bm{r}(\mu^{(1)}).
\end{equation}
We denote by $\delta\mu$ and $\mu^{(t)}$ the (signed)measures defined in \eqref{eq:deltamut} and \eqref{eq:mut} respectively. 

An immediate consequence of the above proposition is the following.
\begin{corollary}
\label{cor:corder}
For any $p\geq 1$ there exists a constant $a_p(\|\Psi\|_{\infty})$ depending only $p$ and $\|\Psi\|_{\infty}$ such that:
\begin{equation}
\label{eq:Lip0}
\E\left[\underset{q\in [0,1]}{\sup}\Big|\delta\phi(q,\bm{\omega})\Big|^p\right]^{\frac{1}{p}}\leq a_p(\|\Psi\|_{\infty}) \|\delta\mu\|_{\infty}
\end{equation}
Moreover:
\begin{equation}
\label{eq:incr}
\textup{\emph{if}  }\mu^{(1)}([0,q])\leq \mu^{(2)}([0,q])\,\,\forall q\in[0,1],\textup{  \emph{then}  }
\phi(\mu^{(1)},q,\bm{\omega})\leq \phi(\mu^{(2)}q,\bm{\omega})\,\,\forall q\in[0,1].
\end{equation}
\end{corollary}
\begin{proof}
A straightforward computation gives
\been{
\label{eq:deltaphi00}
\aled{
&\phi(\mu^{(2)},q,\bm{\omega})-\phi(\mu^{(1)},q,\bm{\omega})=\int^1_0dt\frac{\partial \phi(\mu^{(t)},q,\bm{\omega})}{\partial t}\\
&=\frac{1}{2}\int^1_0dt\widetilde{\mathbb{E}}_{\mu^{(t)}\bm{r}(\mu^{(t)})}\left[\int^{1}_{q}dq'\,\delta \mu([0,q'])\,\|\bm{r}(\mu^{(t)},q',\bm{\omega})\|^2_2\Bigg| \mathcal{F}^{\otimes n}_q\right].
}
}
If $\mu^{(1)}([0,q])\leq \mu^{(2)}([0,q]),\,\forall q\in[0,1]$, then
\been{
\delta \mu([0,q])\geq 0,\,\forall q\in[0,1]
}
So $\phi(\mu^{(2)})-\phi(\mu^{(1)})\geq 0$, proving \eqref{eq:incr}. The equivalence \eqref{eq:deltaphi00} gives
\been{
\aled{
&(\phi(\mu^{(2)},q,\bm{\omega})-\phi(\mu^{(1)},q,\bm{\omega}))^p\\
&\leq\frac{ \|\delta \mu\|^p_{\infty}}{2^p}\int^1_0dt\widetilde{\mathbb{E}}_{\mu^{(t)}\bm{r}(\mu^{(t)})}\left[\(\int^{1}_{0}dp\,\,\|\bm{r}(\mu^{(t)},p,\bm{\omega})\|^2_2\)^p\Bigg| \mathcal{F}^{\otimes n}_q\right]
}
}
By the inequality , $\mathcal{E}(\mu^{(t)}\bm{r}(\mu^{(t)}))\leq e^{2\|\Psi\|_{\infty}}$. Using this inequality, taking the $\sup$ over $q\in [0,1]$, the expectation over $\Wn$ and using the upper bounds \eqref{eq:uniqueness22} and \eqref{bound_rr2} , we complete the proof.
\end{proof}

In the following lemma, we will prove that also the map $\textup{Pr}([0,1])^{\circ}\ni \mu\mapsto \bm{r}(\mu^{(2)})$ is Lipschitz. 

As before, we consider two fixed Parisi parameters $\mu^{(1)}$ and $\mu^{(2)}$ in $\textup{Pr}([0,1])^{\circ}$ and define $\delta\phi$, $\delta\bm{r}$, $\delta\mu$ and $\mu^{(t)}$ as in~\eqref{eq:phirt},~\eqref{eq:deltamut}, and~\eqref{eq:mut}.
\begin{lemma}
\label{uniform_convergence_theo}
For any $p\geq 1$, there exist a constant $b_p(\|\Psi\|_{\infty})$ depending only $p$ and $\|\Psi\|_{\infty}$ such that:
\begin{equation}
\E\left[\left(\int^1_0dt\|\,\delta\bm{r}(q,\bm{\omega})\,\|^2_2\right)^p\right]^{\frac{1}{p}}\leq b_p(\|\Psi\|_{\infty})\|\delta\mu\|^2_{\infty}.
\end{equation}
\end{lemma}
\begin{proof}
Let
\been{
\quad \bm{r}_{+}=\bm{r}(\mu^{(2)})+\bm{r}(\mu^{(1)}).
}
By the Lemma~\ref{rem:uniqueness222}, the processes $\bm{r}(\mu^{(1)})$ and $\bm{r}(\mu^{(2)})$ are in $D^{2,p}_n$ and the stochastic integrals $I(\bm{r}(\mu^{(1)}))$ and $I(\bm{r}(\mu^{(2)}))$ are $\Wn-L_p$--integrable martingales. Thus $\delta\bm{r} \in D^{2,p}_n$, $\bm{r}_{+}\in  D^{2,p}_n$ and $I(\delta\bm{r})$ is a $\Wn-L_p$--integrable martingale. In the same way, by Proposition~\ref{prop:uniqueness2}, $\phi(\mu^{(1)})\in S^b_n$, $\phi(\mu^{(2)})\in S^b_n$, and $\delta\phi\in S^b_n$. The \BSDE  gives
\been{
\aled{
&\delta\phi(q,\bm{\omega})-\delta\phi(0,\bm{\omega})\\
&=I(\delta\bm{r},q,\bm{\omega})-\frac{1}{2}\int^q_0dt\(\mu^{(2)}([0,t])\|\bm{r}(\mu^{(2)},t,\bm{\omega})\|^2_2-\mu^{(1)}([0,t])\|\bm{r}(\mu^{(1)},t,\bm{\omega})\|^2_2\)\\
&=I(\delta\bm{r},q,\bm{\omega})-\frac{1}{2}\int^q_0dt\delta\mu([0,t])\|\bm{r}(\mu^{(2)},t,\bm{\omega})\|^2_2-\frac{1}{2}\int^q_0dt\mu^{(1)}([0,t])\bm{r}_+(t,\bm{\omega})\cdot \delta\bm{r}(t,\bm{\omega}).
}
}
Using the Ito Formula, we get
\been{
\aled{
&\delta\phi^2(q,\bm{\omega})-\delta\phi^2(0,\bm{\omega})\\
&=2\int^q_0\delta\phi(t,\bm{\omega})\bm{r}_{\bm{a}}(t,\bm{\omega})\cdot d\bm{\omega}(t)-\int^q_0dt\delta\phi(t,\bm{\omega})\delta\mu([0,t])\|\bm{r}(\mu^{(2)},t,\bm{\omega})\|^2_2\\
&-\int^q_0dt\delta\phi(t,\bm{\omega})\mu^{(1)}([0,t])\bm{r}_+(t,\bm{\omega})\cdot \delta\bm{r}(t,\bm{\omega})+\int^q_0dt\|\delta\bm{r}(t,\bm{\omega})\|^2_2\\
}
}
All the quantities in the above expression are in $\Wn-L_p$--integrable. Taking $q=1$, we get
\been{
\delta\phi(1,\bm{\omega})=0.
}
Moreover, we put $\int^1_0 dq\|\bm{r}_{\bm{a}}(q,\bm{\omega})\|^2_2$ on the left-hand side of the equation and the other terms in the right-hand side and take the absolute value raised to the power $p$ in both sides. Using the inequality $|A+B+C+D|^p\leq4^{p-1}(|A|^p+|B|^p+|C|^p+|D|^{p})$ and taking the expectation value, we get
\been{
\label{inequality_r}
\E\[\(\int^1_0dt\|\delta\bm{r}(t,\bm{\omega})\|^2_2\)^p\]\leq 2^{2p-1}\(\text{I}+\text{II}+\text{III}+\text{IV}\),
}
where
\been{
\text{I}=\mathbb{E}\left[\left|\delta\phi(0,\bm{\omega})\right|^{2p}\right],
}
\been{
\text{II}=2^{p}\mathbb{E}\left[\left|\int^1_0\delta\phi(q,\bm{\omega})\delta\bm{r}(q,\bm{\omega}) \cdot d\bm{\omega}(q)\right|^p\right],
}
\begin{equation}
\begin{multlined}
\text{III}=\mathbb{E}\left[\left|\int^1_0 dq\,\delta \mu([0,q])\,\delta\phi(q,\bm{\omega})\|\bm{r}_{2}(q,\bm{\omega})\|^2_2\right|^{p}\right],
\end{multlined}
\end{equation}
\begin{equation}
\begin{multlined}
\text{IV}=\mathbb{E}\left[\left|\int^1_0 dq\,\mu^{(2)}(q)\delta\phi(q,\bm{\omega})\bm{r}_{+}(q,\bm{\omega})\cdot\delta \bm{r}(q,\bm{\omega})\right|^{p}\right].
\end{multlined}
\end{equation}
By Corollary~\ref{cor:corder}
\begin{equation}
\label{eq:Iined}
\text{I}\leq a^{2p}_{2p}\|\delta \mu\|^{2p}_{\infty}.
\end{equation}
By the Burkholder Davis and Gundy inequality there exists a universal constant $C_p>0$ such that
\been{
\aled{
&\E\[\(\int^q_0\delta\phi(t,\bm{\omega})\delta\bm{r}(t,\bm{\omega})\cdot d\bm{\omega}(t)\)^{p}\]\leq \E\[\sup_{q\in [0,1]}\(\int^q_0\delta\phi(t,\bm{\omega})\delta\bm{r}(t,\bm{\omega})\cdot d\bm{\omega}(t)\)^{p}\]\\
&\leq C_p\E\[\(\int^1_0dt(\delta\phi(t,\bm{\omega}))^2\|\delta\bm{r}(t,\bm{\omega})\|^2_2\)^{\frac{p}{2}}\]\leq C_p\E\[\sup_{q\in [0,1]}|\delta\phi(q,\bm{\omega})|^{p}\(\int^1_0dt\|\delta\bm{r}(t,\bm{\omega})\|^2_2\)^{\frac{p}{2}}\]\\
&\leq C_p\sqrt{\E\[\sup_{q\in [0,1]}|\delta\phi(q,\bm{\omega})|^{2p}\]}\sqrt{\E\[\(\int^1_0dt\|\delta\bm{r}(t,\bm{\omega})\|^2_2\)^{p}\]}\,.
}
}
Thus, using the inequality~\eqref{eq:Iined}
\been{
\text{II}\leq 2^p\E\[\(\int^q_0\delta\phi(t,\bm{\omega})\delta\bm{r}(t,\bm{\omega})\cdot d\bm{\omega}(t)\)^{p}\]\leq 2^pa^{p}_{2p}\|\delta \mu\|^{p}_{\infty}\sqrt{\E\[\(\int^1_0dt\|\delta\bm{r}(t,\bm{\omega})\|^2_2\)^{p}\]}.
}
Using the H\"older inequality, Lemma~\ref{rem:uniqueness222} and the inequality~\eqref{eq:Iined}, we get
\begin{equation}
\aled{
&\text{III}\leq \mathbb{E}\left[\underset{q\in[0,1]}{\sup}|\delta \phi(q,\bm{\omega})|^{p}\left(\int^1_0 dq\,|\delta \mu([0,q])|\|\bm{r}_{2}(q,\bm{\omega})\|^2_2\right)^{p}\right]\\
&\leq \sqrt{\mathbb{E}\left[\underset{q\in[0,1]}{\sup}|\delta \phi(q,\bm{\omega})|^{2p}\right]}\sqrt{\E\left[\left(\int^1_0 dq\,\|\bm{r}_{2}(q,\bm{\omega})\|^2_2\right)^{2p}\right]}\|\delta \mu\|^p_{\infty}\leq a^p_{2p}\,K^p_{2p}\|\delta \mu\|_{\infty}^{2p}.
}
\end{equation}
Finally
\begin{equation}
\aled{
&\text{IV}\leq \mathbb{E}\left[\underset{q\in[0,1]}{\sup}|\delta \phi(q,\bm{\omega})|^{p}\left|\int^1_0 dq\,\mu^{(2)}(q)\,\bm{r}_{+}(q,\bm{\omega})\cdot\delta \bm{r}(q,\bm{\omega})\right|^{p}\right]\\
&\leq\mathbb{E}\left[\underset{q\in[0,1]}{\sup}|\delta \phi(q,\bm{\omega})|^{p}\left|\int^1_0 dq\,\,\|\bm{r}_{+}(q,\bm{\omega})\|^2_2\right|^{\frac{p}{2}}\left|\int^1_0 dq\,\|\delta \bm{r}(q,\bm{\omega})\|^2_2\right|^{\frac{p}{2}}\right]\\
&\leq \sqrt{\mathbb{E}\left[\underset{q\in[0,1]}{\sup}|\delta \phi(q,\bm{\omega})|^{2p}\left|\int^1_0 dq\,\,\|\bm{r}_{+}(q,\bm{\omega})\|^2_2\right|^{p}\right]}\sqrt{\mathbb{E}\left[\left(\int^1_0 dq\,\|\delta \bm{r}(q,\bm{\omega})\|^2_2\right)^{p}\right]},
}
\end{equation}
and using Lemma~\ref{rem:uniqueness222} and the Lemma~\ref{cor:corder}
\been{
\aled{
&\mathbb{E}\left[\underset{q\in[0,1]}{\sup}|\delta \phi(q,\bm{\omega})|^{2p}\left|\int^1_0 dq\,\,\|\bm{r}_{+}(q,\bm{\omega})\|^2_2\right|^{p}\right]\\
&\leq \sqrt{\mathbb{E}\left[\underset{q\in[0,1]}{\sup}|\delta \phi(q,\bm{\omega})|^{4p}\right]}\sqrt{\mathbb{E}\left[\left|\int^1_0 dq\,\,\|\bm{r}_{+}(q,\bm{\omega})\|^2_2\right|^{2p}\right]}\leq a^{2p}_{4p}K_{2p}\|\delta\mu\|^{2p}_{\infty}.
}
}
So
\been{
\text{IV}\leq a^p_{4p}\sqrt{K_{2p}}\|\delta\mu\|^{p}_{\infty}\sqrt{\mathbb{E}\left[\left(\int^1_0 dq\,\|\delta \bm{r}(q,\bm{\omega})\|^2_2\right)^{p}\right]}.
}
Let us take
\begin{equation}
X=\sqrt{\mathbb{E}\left[\left(\int^1_0dq \|\delta\bm{r}(q,\bm{\omega})\|^2_2\right)^{p}\right]}
\end{equation}
and combine the above inequalities in~\eqref{inequality_r}. We get:
\begin{equation}
\label{inequality_2}
X^2\leq \alpha_p \|\delta \mu\|_{\infty}^p X+\beta_p\|\delta \mu\|_{\infty}^{2p}
\end{equation}
where $\alpha_p$ and $\beta_p$ are two positive constants that depends only on $p$. That implies that
\begin{equation}
X\leq \,\frac{1}{2}\left(\alpha_p+\sqrt{\alpha^2_p+4\beta_p\,}\right)\,\|\delta \mu\|_{\infty}^{p},
\end{equation}
completing the proof.
\end{proof}
We now present the proof main result of this section.

\begin{proof}[Proof of Proposition~\ref{convergence_result}]
Let us consider the sequence of pairs of non-negative random variables
\begin{equation}
U_k=\underset{q\in[0,1]}{\sup}\left|\phi(\mu^{(k+1)},q,\bm{\omega})-\phi(\mu^{(k)},q,\bm{\omega})\right|
\end{equation}
and
\begin{equation}
V_k=\int^1_0 dq\|\bm{r}(\mu^{(k+1)},q,\bm{\omega})-\bm{r}(\mu^{(k)},q,\bm{\omega})\|^2_2.
\end{equation}
Corollary~\eqref{cor:corder} and Lemma~\eqref{uniform_convergence_theo} yield:
\begin{equation}
\label{inequality_start}
\E\big[ U^p_k\big]\leq a_p2^{-p k},\quad
\E\big[V^p_k\big]\leq b_p2^{-2p k}.
\end{equation}
Thus for any $n\leq m$ it holds
\been{
\aled{
\sup_{m>n}\E\[\underset{q\in[0,1]}{\sup}\left|\phi(\mu^{(m)},q,\bm{\omega})-\phi(\mu^{(n)},q,\bm{\omega})\right|^p\]^{\frac{1}{p}}\leq \sum^m_{k=n} \E\big[ U^p_k\big]^{\frac{1}{p}}\leq a_p^{\frac{1}{p}}2^{-n}
}
}
and
\been{
\sup_{m>n}\E\[\(\int^1_0 dq\|\bm{r}(\mu^{(m)},q,\bm{\omega})-\bm{r}(\mu^{(n)},q,\bm{\omega})\|^2_2\)^p\]^{\frac{1}{p}}\\
\leq \sum^m_{k=n} \E\big[ V^p_k\big]^{\frac{1}{p}}\leq b_p^{\frac{1}{p}}\frac{4}{3}2^{-2n}
} 
from which it is straightforward to obtain that the sequence $\(\left(\,\phi(\mu^{(k)}),\bm{r}(\mu^{(k)})\,\right)\)_{k\in \N}$ converges in $S^p_n \times D^{2,p}_n$ norm to a pair $(\phi,\bm{r})$. Moreover, by Markov inequality and~\eqref{inequality_start}, one gets
\begin{equation}
\Wn\left(\,U_k>2^{-\frac{k}{2}} \right)\leq \frac{\mathbb{E}\big[U^p_k\big]}{\epsilon^p}\leq a_p\left(\frac{1}{2^{\frac{k}{2}}}\right)^p,
\end{equation}
and in the same way:
\begin{equation}
\mathbb{W}\(\,V_k>2^{-k}\right)\leq b_p\left(\frac{1}{2^{k}}\right)^p,
\end{equation}
Hence, the Borel-Cantelli lemma yields
\been{
\Wn\(U_k> 2^{-\frac{k}{2}}\quad  \textup{for infitely many } k\)=0,
}
and
\been{
\Wn\(V_k> 2^{-k}\quad  \textup{for infitely many } k\)=0.
}
the above two inequalities, implies
\been{
\sum^{\infty}_{k=1}U_k<\infty,\quad \Wnas,
}
and
\been{
\sum^{\infty}_{k=1}V_k<\infty,\quad \Wnas,
}
Thus,
\been{
\lim_{n\to \infty}\sup_{m>n}\underset{q\in[0,1]}{\sup}\left|\phi_{\mu^{(m)}}(q,\bm{\omega})-\phi_{\mu^{(n)}}(q,\bm{\omega})\right|\leq \lim_{n\to \infty}\sum^{\infty}_{k=n}U_k=0,\quad \Wnas,
}
and
\been{
\lim_{n\to \infty}\sup_{m>n}\int^1_0 dq\|\bm{r}_{\mu^{(m)}}(q,\bm{\omega})-\bm{r}_{\mu^{(n)}}(q,\bm{\omega})\|^2_2\leq \lim_{n\to \infty}\sum^{\infty}_{k=n}V_k=0,\quad \Wnas.
}
Thus $\(\left(\,\phi(\mu^{(k)}),\bm{r}(\mu^{(k)})\,\right)\)_{k\in \N}$ converges almost surely. The almost sure convergence and the convergence in $S^p_n \times D^{2,p}_n$ norm to $(\phi,\bm{r})\in S^p_n \times D^{2,p}_n$, implies that the sequence  $\(\left(\,\phi(\mu^{(k)}),\bm{r}(\mu^{(k)})\,\right)\)_{k\in \N}$  converges to $(\phi,\bm{r})$ $\Wn$--almost surely.

Moreover, the definition of $(U_k)_{k\in \N}$ implies that $(\phi(\mu^{(k)}))_{k\in \N}$ converges to $\phi$ almost surely uniformly in the interval $[0,1]$. So the process $\phi$ has almost surely continuous sample paths in $[0,1]$ since, by Proposition~\ref{prop:existk}, the processes $\phi(\mu^{(k)})$ has almost surely continuous sample paths in $[0,1]$.

It remains to show that the pair $(\phi,\bm{r})$ is a solution to the equation~\eqref{selfEq1Aux} corresponding to the $\mu\in \textup{Pr}([0,1])$. Since the process $\bm{r}$ is in $ D^{2,p}_n$, for any $p\geq1$, then the Ito integral $I(\bm{r})$ and the integral $(q,\bm{\omega})\mapsto \int^q_0 dq' \mu([0,q']) \|\bm{r}(q',\bm{\omega})\|^2_2$  and are in $S^{p}_n$. Let
\begin{equation}
\text{I}_k(q,\bm{\omega})=\int^1_q\bm{r}(\mu^{(k)},q',\bm{\omega})\cdot d\bm{\omega}(q'),\quad \text{II}_k(q,\bm{\omega})=\int^1_q dq' \mu^{(k)}([0,q']) \|\bm{r}(\mu^{(k)},q',\bm{\omega})\|^2_2
\end{equation}
and
\begin{equation}
\text{I}(q,\bm{\omega})=\int^1_q \bm{r}(q',\bm{\omega})\cdot d\bm{\omega}(q'),\quad \text{II}(q,\bm{\omega})=\int^1_q dq' \mu([0,q']) \|\bm{r}(q',\bm{\omega})\|^2_2.
\end{equation} 
Note that $\text{I}_k$, $\text{II}_k$, $\text{I}$ and $\text{II}$ are not adapted process. We must prove the almost sure convergence of the sequences $(\text{I}_k)_{k\in \N}$ and $(\text{II}_k)_{k\in \N}$ to $\text{I}$ and $\text{II}$ respectively. From the BSDE it holds
\been{
\label{eq:almostk}
\phi(\mu^{(k)},q,\bm{\omega})+\text{I}_k(q,\bm{\omega})-\frac{1}{2}\text{II}_k(q,\bm{\omega})=\Psi(\bm{\omega})
}

Let us define the following non-negative random variables
\begin{equation}
G_k=\underset{q\in[0,1]}{\sup}\left|\,\text{I}_k(q,\bm{\omega})-\text{I}(q,\bm{\omega})\,\right|=\underset{q\in[0,1]}{\sup}\left|\int^1_q  \left(\bm{r}(\mu^{(k)},q',\bm{\omega})-\bm{r}(q',\bm{\omega})\right)\cdot d\bm{\omega}(q')\right|
\end{equation}
and
\begin{equation}
\aled{
&F_k=\underset{q\in[0,1]}{\sup}\left|\,\text{II}_k(q,\bm{\omega})-\text{II}(q,\bm{\omega})\,\right|\\
&=\underset{q\in[0,1]}{\sup}\left|\int^1_q dq'\left(\mu^{(k)}([0,q'])\|\bm{r}(\mu^{(k)},q',\bm{\omega})\|^2_2-\mu([0,q']) \|\bm{r}(q',\bm{\omega})\|^2_2\right)\right|.
}
\end{equation}
By BDG inequality, there is a positive constant $C_p$, depending only on $p$, such that
\begin{equation}
\E\left[G_k^p\right]\leq C_p \E\left[V_k^{p/2}\right]\leq C_p b_p 2^{-kp}.
\end{equation}
 Moreover, the inequality~\eqref{bound_rr2} yields
\begin{equation}
\aled{
&\E\left[F_k^p\right]\\
&\leq \|\mu^{(k)}-\mu\|^{p}_{\infty}\E\left[\left(\int^1_0 dq' \|\bm{r}(\mu^{(k)},q',\bm{\omega})\|^2_2\right)^p\right]\\
&+\E\left[\left(\int^1_0 dq' \mu([0,q']) \|\bm{r}(\mu^{(k)},q',\bm{\omega})-\bm{r}(q',\bm{\omega})\|^2_2\right)^p\right]\leq  2 ^{-pk}K_p+\E\left[V_k^{p/2}\right]\leq 2 ^{-pk} c_p
}
\end{equation}
where $c_p$ is a positive constant depending only on $p$. As the above two inequalities imply that the sequences $(G_k)_{k\in \N}$ and $(F_k)_{k\in \N}$ converge in $L^p$ and $\Wn$--almost surely to $0$, so the random variables $(\text{I}_k)_{k\in \N}$ and $(\text{II}_k)_{k\in \N}$ converge $\Wn$--almost surely uniformly in $q\in [0,1]$ to $\text{I}$ and $\text{II}$ respectively. Since $(\phi(\mu^{(k)}))_{k\in \N}$ converges $\Wn$--almost surely to to $\phi$, the equation~\eqref{eq:almostk} gives
\been{
\aled{
&\phi(q,\bm{\omega})+\text{I}(q,\bm{\omega})-\frac{1}{2}\text{II}(q,\bm{\omega})\\
&=\lim_{k\to \infty}\(\phi(\mu^{(k)},q,\bm{\omega})+\text{I}_k(q,\bm{\omega})-\frac{1}{2}\text{II}_k(q,\bm{\omega})\)=\Psi(\bm{\omega}),
}
}
completing the proof.
\end{proof}
\section{Uniqueness of the solution and proof of Theorem~\ref{th:BSDE}}
\label{sec:43}
In this section we establish the uniqueness result, which provides the last missing piece in the proof of Theorem~\ref{th:BSDE}. 

 \begin{proposition}[Uniqueness]
\label{prop:uniqueness}
The solution to the BSDE~\eqref{selfEq1Aux} is unique on $S_n\times D_{n}$.
\end{proposition}

\begin{proof}
By Proposition~\ref{prop:uniqueness2}, any solution to the \BSDE is in $S^b_n\times \widehat{D}^{\mu}_{n}$. Thus we only need to prove the uniqueness in $S^b_n\times \widehat{D}^{\mu}_{n}$. Let $(\phi^*,\bm{r}^*)\in S^b_n\times \widehat{D}^{\mu}_{n}$ and  $(\overline{\phi},\overline{\bm{r}})\in S^b_n\times \widehat{D}^{\mu}_{n}$ be two solutions. Let us define
\been{
\delta\phi=\phi^*-\overline{\phi},\quad \delta\bm{r}=\bm{r}^*-\overline{\bm{r}}.
}
By the \BSDE, we have
\been{
\label{eq:dif1}
\aled{
&\delta\phi(q,\bm{\omega})=-I(\delta\bm{r}_q,1,\bm{\omega})+\frac{1}{2}\int^1_qdt\mu([0,t])\|\bm{r}^*(t,\bm{\omega})\|^2_2-\frac{1}{2}\int^1_qdt\mu([0,t])\|\overline{\bm{r}}(t,\bm{\omega})\|^2_2\\
&=-I_{\mu\overline{\bm{r}}}(\delta\bm{r}_q,1,\bm{\omega})+\frac{1}{2}\int^1_qdt\mu([0,t])\|\delta\bm{r}(t,\bm{\omega})\|^2_2
}
}
or, in a similar way
\been{
\label{eq:dif2}
\delta\phi(q,\bm{\omega})=-I_{\mu \bm{r}^*}(\delta\bm{r}_q,1,\bm{\omega})-\frac{1}{2}\int^1_qdt\mu([0,t])\|\delta\bm{r}(t,\bm{\omega})\|^2_2.
}
As usual $\delta\bm{r}_q(q',\bm{\omega}):=\bm{1}_{q'>q}\delta\bm{r}_q(q',\bm{\omega})$. By Proposition~\ref{prop:uniqueness2}, the processes $\mathcal{E}(\mu\bm{r}^*)$ and $\mathcal{E}(\mu\overline{\bm{r}})$ are a non-negative $\Wn$--martingale of average $1$ adapted to $\Fqfilt{q}{[0,1]}$. Thus, we can define the probability measures $\W_{\mu\bm{r}^*}$ and $\W_{\mu\overline{\bm{r}}}$. Moreover, by Corollary~\ref{rem:uniqueness222} the process $I_{\mu\bm{r}^*}(\delta \bm{r})$ is a $\W_{\mu\bm{r}^*}$--martingale and $I_{\mu\overline{\bm{r}}}(\delta \bm{r})$ is a $\W_{\mu\overline{\bm{r}}}$--martingale. Since $\phi^*$ and $\overline{\phi}$ are $\Fqfilt{q}{[0,1]}$--adapted, from~\eqref{eq:dif1} we get
\been{
\label{eq:dif10}
\aled{
&\delta\phi(q,\bm{\omega})=\Eecond{\mu\overline{\bm{r}}}{\delta\phi(q,\bm{\omega})}{q}\\
&=\Eecond{\mu\overline{\bm{r}}}{-I_{\mu \overline{\bm{r}} }(\delta\bm{r}_q,1,\bm{\omega})}{q} +\frac{1}{2}\Eecond{\mu\bm{r}^*}{\int^1_qdt\mu([0,t])\|\delta\bm{r}(t,\bm{\omega})\|^2_2}{q}\\
&=\frac{1}{2}\Eecond{\mu\overline{\bm{r}}}{\int^1_qdt\mu([0,t])\|\delta\bm{r}(t,\bm{\omega})\|^2_2}{q}\geq 0,
}
}
and from~\eqref{eq:dif2} we get
\been{
\label{eq:dif20}
\aled{
&\delta\phi(q,\bm{\omega})=\Eecond{\mu\bm{r}^*}{\delta\phi(q,\bm{\omega})}{q}\\
&=\Eecond{\mu\bm{r}^*}{-I_{\mu\bm{r}^* }(\delta\bm{r}_q,1,\bm{\omega})}{q} -\frac{1}{2}\Eecond{\mu\bm{r}^*}{\int^1_qdt\mu([0,t])\|\delta\bm{r}(t,\bm{\omega})\|^2_2}{q}\\
&=-\frac{1}{2}\Eecond{\mu\bm{r}^*}{\int^1_qdt\mu([0,t])\|\delta\bm{r}(t,\bm{\omega})\|^2_2}{q}\leq 0.
}
}
So $\overline{\phi}(q,\bm{\omega})-\phi^*(q,\bm{\omega})\geq 0$ and $\phi^*(q,\bm{\omega})-\overline{\phi}(q,\bm{\omega})\geq 0$ for any $q\in[0,1]$. This implies that
\been{
\overline{\phi}(q,\bm{\omega})=\phi^*(q,\bm{\omega}),\quad \forall q\in [0,1],\,\W^{\otimes n}-\textup{a.s.}.
}
Moreover, the above identity and the identities~\eqref{eq:dif10} and~\eqref{eq:dif20} give
\been{
\aled{
&\Ee{\mu\bm{r}^*}\left[ \int^1_0 dt \mu([0,t])\,\|\bm{r}^*(t, \,\bm{\omega})- \overline{\bm{r}}(t, \,\bm{\omega})\|^2_2\right]\\
&=\Ee{\mu\overline{\bm{r}}}\left[ \int^1_0 dt \mu([0,t])\,\|\bm{r}^*(t, \,\bm{\omega})- \overline{\bm{r}}(t, \,\bm{\omega})\|^2_2\right]=0.
}
}
Let $q_0\in [0,1]$ be the quantity defined in~\eqref{eq:qx0}. The above identity implies
\been{
\int^1_{q_0} dt \,\|\bm{r}^*(t, \,\bm{\omega})- \overline{\bm{r}}(t, \,\bm{\omega})\|^2_2=0,\quad \W^{\otimes n}-\textup{a.s.}
}
By the above results $(\phi^*,\bm{r}^*)=(\overline{\phi},\overline{\bm{r}})$ for all $q\in [q_0,1]$ (in the sense of $S^n_n\times D_n$). Finally, by Lemma~\ref{lem:BSDEproof00}, the existention of the solution in $q\in [0,q_0]$ is unique and it is given by~\eqref{eq:solq0}, proving that  $(\phi^*,\bm{r}^*)=(\overline{\phi},\overline{\bm{r}})$ for all $q\in[0,1]$ $\Wnas$.
\end{proof}
The proof of Theorem~\ref{th:BSDE} is now trivial.
\begin{proof}[Proof of Theorem~\ref{th:BSDE}]
Proposition~\ref{prop:existk} and the convergence result~\ref{convergence_result} implies the existence of the solution to the \BSDE for any allowable Parisi parameter. Proposition~\ref{prop:uniqueness2} yields the uniqueness.
\end{proof}
\section{Proof of Theorem~\ref{th:func_cont}}
\label{sec:4}
Given the existence and uniqueness result for the \BSDE, we can finally prove Theorem~\ref{th:func_cont}.

We recall the definition of the functionals $\Gamma:B_n\times \textup{Pr}([0,1])\times D_n\times [0,1)\times S_n$ and $\Phi:B_n\times \textup{Pr}([0,1])\times D_n\times [0,1)\times S_n$ defined in~\eqref{eq:auxiliary_func0} and~\eqref{eq:auxiliary_funccond}
 \been{
\label{eq:auxiliary_func0bis}
 \aled{
&\Gamma\big(\,\Psi,\,\mu,\,\bm{r},\,q\,,\bm{\omega}\big)=\\
&\mathbb{E}\left[\mathcal{E}(\mu \bm{r},1,\bm{\omega})\Psi ( \bm{\omega}\,)\Big| \,\mathcal{F}^{\otimes n}_q\, \right]-\frac{1}{2}\mathbb{E}\left[\int^1_q dt\,\mu([0,t])\mathcal{E}(\mu \bm{r},t,\bm{\omega})\,\|\,\bm{r}(\,t, \,\bm{\omega}\,) \,\|^2_2\,\,\bigg| \,\mathcal{F}^{\otimes n}_q\,\right]
}
}
and
\been{
\label{eq:auxiliary_funccondbis}
\Phi(\Psi,\mu,q,\bm{\omega})=\sup_{\bm{r}\in D_n}\Gamma\big(\,\Psi,\,\mu,\,\bm{r},\,q,\bm{\omega}).
}
Proposition~\ref{prop:uniqueness2} ensures the following integrability property.
\begin{lemma}
\label{lem:martcon2}
If $(\phi^*,\bm{r}^*)\in S_n^b\times \widehat{D}^\mu_n$ is a solution to the BSDE~\eqref{selfEq1Aux}, then for any $\bm{v}\in \widehat{D}^\mu_n$ the process $I_{\mu \bm{v}}(\bm{r}^*)$, defined as in~\eqref{eq:Iv}, is a $\W_{\mu\bm{v}}$--uniformly integrable martingale with mean $0$. Moreover
\been{
\label{eq:martcon2bound}
\Ee{\mu\bm{v}}\[\int^1_{0}dt\mu([0,t])\|\bm{r}^*(t,\bm{\omega})\|^2_2\]<\infty.
}
\end{lemma}
\begin{proof}

Let $K(\bm{v},\bm{r}^*)$ be the process
\been{
K(\bm{v},\bm{r}^*,q,\bm{\omega})=\phi^*(q,\bm{\omega})-\phi^*(0,\bm{\omega})-\frac{1}{2}\int^1_0dt\mu([0,t])\|\bm{v}(t,\bm{\omega})\|^2_2.
}
The inequality~\eqref{eq:uniqueness2} yields
\been{
\aled{
& \sup_{q\in [0,1]}|K(\bm{v},\bm{r}^*,q,\bm{\omega})|\leq 2\|\Psi\|_{\infty}+\frac{1}{2}\int^1_0dt\mu([0,1])\|\bm{v}(t,\bm{\omega})\|^2_2,\quad \W^{\otimes n}-\textup{a.s.}
}
}
So, since $\bm{v}\in \widehat{D}^{\mu}_n$, it follows that:
\been{
\label{eq:boundenss}
\aled{
&\Ee{\mu\bm{v}}\[\sup_{q\in [0,1]}|K(\bm{v},\bm{r}^*,q,\bm{\omega})|\]\leq 2\|\Psi\|_{\infty}+\frac{1}{2}\Ee{\mu\bm{v}}\[\int^q_0dt\mu([0,q])\|\bm{v}(q,\bm{\omega})\|^2_2\]<\infty.
}
}
Let us define the increasing sequence of stopping times $(\tau_k)_{k\in \N}$ such that
\been{
\tau_k=\inf\left\{q\in [q_0,1]; \int^q_{0}dt \,\|\bm{r}^*(t, \,\bm{\omega})\|^2_2\geq k\right \}
}
with $\inf\emptyset=1$. Since $\bm{r}^*\in D_n$, $\tau_k\uparrow 1$ $\Wnas$. 
By Lemma~\ref{lem:mart_cond}, the stochastic integral $I^{\tau_k}_{\mu \bm{v}}(\bm{r}^*)$ is a $\W_{\mu\bm{v}}$--martingale with mean $0$.
The \BSDE gives
\been{
\label{eq:KI}
I_{\mu\bm{v}}(\bm{r}^*,q,\bm{\omega})=K(\bm{v},\bm{r}^*,q,\bm{\omega})+\frac{1}{2}\int^{q}_{0} dt \mu([0,t])\,\left\|\bm{v}(t, \,\bm{\omega})-\bm{r}^*(t, \,\bm{\omega})\right\|^2_2.
}
As a consequence, for any stopping time $T\in  [0,1]$ the above equality~\eqref{eq:KI} implies
\been{
\label{eq:upperK0}
\frac{1}{2}\Ee{\mu \bm{v}}\[\int^{T\wedge \tau_{k}}_{0} dt \mu([0,t])\,\left\|\bm{v}(t, \,\bm{\omega})-\bm{r}^*(t, \,\bm{\omega})\right\|^2_2\]=-\Ee{\mu\bm{v}}\[K^{\tau_k}(\bm{v},\bm{r}^*,T,\bm{\omega})\].
}
The right-hand member is upper bounded by the right-hand member of \eqref{eq:boundenss}, that is a finite number independent of $k$ and $T$. Thus the Monotone Convergence Theorem gives
\been{
\label{eq:boundnness2}
\aled{
&\Ee{\mu \bm{v}}\[\int^{T}_{0}dt\mu([0,t])\|\bm{v}(t, \,\bm{\omega})-\bm{r}^*(t, \,\bm{\omega})\|^2_2\]\\
&=\sup_{k\in \N}\Ee{\mu \bm{v}}\[\int^{T\wedge \tau_{k}}_{0} dt \mu([0,t])\,\left\|\bm{v}(t, \,\bm{\omega})-\bm{r}^*(t, \,\bm{\omega})\right\|^2_2\]<\infty.
}
}
By ~\eqref{eq:KI}, we have
\been{
\label{eq:upper_b}
|I_{\mu\bm{v}}(\bm{r}^*,q,\bm{\omega})|\leq  \sup_{q\in [0,1]}|K(\bm{v},\bm{r}^*,q,\bm{\omega})|+\frac{1}{2}\int^{1}_{0} dt \mu([0,t])\,\left\|\bm{v}(t, \,\bm{\omega})-\bm{r}^*(t, \,\bm{\omega})\right\|^2_2.
}
Thus, by~\eqref{eq:boundenss} and~\eqref{eq:boundnness2}, the processes $I_{\mu\bm{v}}(\bm{r}^*)$ and $\(I^{T\wedge\tau_k}_{\mu\bm{v}}(\bm{r}^*)\)_{k\in \N}$ are all upper bounded by the same $\Fq{1}$--masurable $\W_{\mu\bm{v}}$--integrable random variable. So $I_{\mu\bm{v}}(\bm{r}^*)$ is $\W_{\mu\bm{v}}$ uniformly integrable. Thus it is a $\W_{\mu\bm{v}}$ uniformly integrable martingale and it is $\Fqfilt{q}{[0,1]}$--adapted. Furthermore, the triangular inequality yields
\been{
\label{eq:qv_boundess}
\aled{
&\Ee{\mu\bm{v}}\[\int^1_{0}dt\mu([0,t])\|\bm{r}^*(t,\bm{\omega})\|^2_2\]\\
&\leq \Ee{\mu\bm{v}}\[\int^1_{0}dt\mu([0,t])\|\bm{v}(t,\bm{\omega})\|^2_2\]+\Ee{\mu \bm{v}}\[\int^{1}_{0}dt\mu([0,t])\|\bm{r}^*(t,\bm{\omega})-\bm{v}(t,\bm{\omega})\|^2_2\]<\infty,
}
}
proving~\eqref{eq:martcon2bound}. 
\end{proof}
The above lemma implies that for any $\bm{v}\in \widehat{D}^\mu_n$
\been{
\label{eq:after_lem}
\Eecond{\mu\bm{v}}{I(\bm{r}^*_q,1,\bm{\omega})}{q}=\Eecond{\mu\bm{v}}{\int^1_{0}dt \mu([0,t])\, \bm{r}^*(t,\bm{\omega})\cdot\bm{v}(t,\bm{\omega}) }{q}<\infty.
}
where $\bm{r}^*_q$ is defined in \eqref{eq:rq2}.

We now establish the relation between the solution of the BSDE and the functional $\Gamma$.
\begin{lemma}
\label{lemma_Global_minimum}
If the pair of processes $(\phi^*,\bm{r}^*)\in S_n^b\times \widehat{D}^{\mu}_n$ is a solution to the BSDE~\eqref{selfEq1Aux}, then, for any $\bm{v}\in \widehat{D}^\mu_n$, it holds
\been{
\label{eq:G_diff}
\Gamma(\Psi,\mu,\bm{v},q,\bm{\omega})-\phi^*(q,\bm{\omega})=-\frac{1}{2}\Ee{\mu\bm{v}}\left[ \int^1_q dt \mu([0,t])\,\|\bm{r}^*(t, \,\bm{\omega})- \bm{v}(t, \,\bm{\omega})\|^2_2\Bigg|\mathcal{F}^{\otimes n}_q\right]
}
and
\begin{equation}
\label{value_process_solution}
\phi^*(q,\bm{\omega})=\Gamma(\Psi,\mu,\bm{r}^*,q,\bm{\omega})\,.
\end{equation}
\end{lemma}
\begin{proof}
The \BSDE and the equation \eqref{eq:after_lem} give
\begin{equation}
\label{eq:KI0}
\begin{aligned}
&\Gamma(\Psi,\,\mu,\,\bm{v},\,q,\bm{\omega})\\&=\widetilde{\mathbb{E}}_{\mu\bm{v}}\left[\Psi ( \,\bm{\omega}\,)\,|\mathcal{F}^{\otimes n}_q\right]-\frac{1}{2} \widetilde{\mathbb{E}}_{\mu\bm{v}}\left[\int^1_q dt \mu([0,t])\,\left \|\,\bm{v}(t, \,\bm{\omega}) \,\right\|^2_2\,\Bigg|\mathcal{F}^{\otimes n}_q\right]\\
&=\phi^*(q,\bm{\omega})+\widetilde{\mathbb{E}}_{\mu\bm{v}}\left[\left.I(\bm{r}^*_q,1,\bm{\omega})-\frac{1}{2}\int^1_q dt \mu([0,t])\,\(\left \|\,\bm{v}(t, \,\bm{\omega}) \,\right\|^2_2+\left \|\,\bm{r}^*(t, \,\bm{\omega}) \,\right\|^2_2\)\,\right|\mathcal{F}^{\otimes n}_q\right]\\
&=\phi^*(q,\bm{\omega})-\frac{1}{2}\widetilde{\mathbb{E}}_{\mu\bm{v}}\left[\int^{q}_{0} dt \mu([0,t])\,\left\|\bm{v}(t, \,\bm{\omega})-\bm{r}^*(t, \,\bm{\omega})\right\|^2_2\].
\end{aligned}
\end{equation}
proving \eqref{eq:G_diff}. The proof of \eqref{value_process_solution} is given by replacing $\bm{v}$ with $\bm{r}^*$ in~\eqref{eq:G_diff}.
\end{proof}
The proof of Theorem~\ref{th:func_cont} is now trivial.
\begin{proof}[Proof of Theorem~\ref{th:func_cont}]
If $|\Psi(\bm{v})|<C$ $\Wn$--a.s., then, by Lemma~\ref{lem:convDDE}, for any $\bm{v}\in D_n$ it holds
\been{
\left|\E\[\mathcal{E}(\mu\bm{v},1,\bm{\omega})\Psi(\bm{\omega})\]\right|\leq C.
}
Thus, if $\bm{v}\in D_n\setminus \widehat{D}^{\mu}_n$
\been{
\label{eq:min1}
\Gamma(\Psi,\mu,\bm{v},q,\bm{\omega})=-\infty,\quad \forall \bm{v}\in D_n\setminus \widehat{D}^{\mu}_n.
}
Let $(\phi^*,\bm{r}^*)\in S^b_n\times\widehat{D}^{\mu}_n$ be the solution to the \BSDE. If $\bm{v}\in \widehat{D}^{\mu}_n$, then Lemma~\ref{lemma_Global_minimum} gives
\been{
\label{eq:min2}
\aled{
&\Gamma(\Psi,\mu,\bm{v},q,\bm{\omega})=\phi^*(q,\bm{\omega})-\frac{1}{2}\Eecond{\bm{v}}{\int^1_qdt\mu([0,t])\|\bm{r}^*(t,\bm{\omega})-\bm{v}(t,\bm{\omega})\|^2_2}{q}\\
&=\Gamma(\Psi,\mu,\bm{r}^*,q,\bm{\omega})-\frac{1}{2}\Eecond{\bm{v}}{\int^1_qdt\mu([0,t])\|\bm{r}^*(t,\bm{\omega})-\bm{v}(t,\bm{\omega})\|^2_2}{q}\\
&\leq \Gamma(\Psi,\mu,\bm{r}^*,q,\bm{\omega}),\quad \forall \bm{v}\in \widehat{D}^{\mu}_n
}
}
Thus
\been{
\phi^*(q,\bm{\omega})=\Gamma(\Psi,\mu,\bm{r}^*,q,\bm{\omega})=\sup_{\bm{v}\in S_n}\Gamma(\Psi,\mu,\bm{v},q,\bm{\omega})=\Phi(\Psi,\mu,q,\bm{\omega}).
}
\end{proof}
\section{Proof of Theorem~\ref{th:func_quant}}
\label{sec:7}
In this section, we derive the properties of the solution to the \BSDE, proving Theorem~\ref{th:func_quant}.

Given $\Psi\in B_n$ and $\mu\in\textup{Pr}([0,1]$, let $(\phi(\Psi,\mu),\bm{r}(\Psi,\mu))\in S^b_n\times \widehat{D}^{\mu}_n$ the solution to the \BSDE with the final condition $\Psi$ and the Parisi parameter $\mu$.
By the existence and uniqueness results in Theorem~\ref{th:BSDE}, $(\Psi,\mu)\mapsto \phi(\Psi,\mu)$ is a well-defined functional from $B_n\times \textup{Pr}([0,1])$ to $S^b_n$. Moreover, by Theorem~\ref{th:func_cont}, we can represent the solution to the \BSDE as the $\mu$--RSB expectation $\Phi(\Psi,\mu)$, given in the variational formula $\Phi(\Psi,\mu)$ in~\eqref{eq:auxiliary_funccondbis}. 

We start by proving the statement $1)$ of Theorem~\ref{th:func_quant}
\begin{proposition}
Given $\Psi\in B_n$ and $\mu\in\textup{Pr}([0,1])$, the process $\mathcal{E}(\mu\bm{r}(\Psi,\mu))$ is a non-negative $\Wn$--martingale, adapted to $\Fqfilt{q}{[0,1]}$ and with mean $1$.
\end{proposition}
\begin{proof}
The proof is given by combining lemmas~\ref{lem:convDDE},~\ref{lem:martcond}, and Proposition~\ref{prop:uniqueness2}.
\end{proof}
The remaining statements of the theorem are proved in the following subsection.
\subsection{Proof of the statement 2) of Theorem~\ref{th:func_quant}}
In this subsection, we prove the statement $1)$ of Theorem~\ref{th:func_quant}.

The convexity is a direct consequence of Theorem~\ref{th:func_cont} and~\ref{th:BSDE}.
\begin{proposition}[Convexity on $\Psi$]
For any $\Psi^{(0)}\in B_n$, $\Psi^{(1)}\in B_n$ and $\mu\in \textup{Pr}([0,1])$, it holds
\begin{equation}
\label{first_derivative_complete_psi}
\aled{
&t\phi(\Psi^{(1)},\mu,q,\bm{\omega})+(1-t)\phi(\Psi^{(0)},\mu,q,\bm{\omega})\geq \phi(t\Psi^{(1)}+(1-t)\Psi^{(0)},\mu,q,\bm{\omega}).
}
\end{equation}
\end{proposition}
\begin{proof}
Let
\been{
\Psi^{(t)}=t\Psi^{(1)}+(1-t)\Psi^{(0)}.
}
By Teorems~\ref{th:func_cont} and~\ref{th:BSDE}, we have
\been{
\label{eq:eqphiPhi}
\phi(\Psi^{(t)},\mu,q,\bm{\omega})=\Phi(\Psi^{(t)},\mu,q,\bm{\omega})=\sup_{\bm{v}\in D_n}\Gamma(\Psi^{(t)},\mu,\bm{v},q,\bm{\omega})
}
The functional $\Gamma$ is linear on $\Psi$. So
\been{
\Gamma(\Psi^{(t)},\mu,\bm{v},q,\bm{\omega})=t\Gamma(\Psi^{(1)},\mu,\bm{v},q,\bm{\omega})+(1-t)\Gamma(\Psi^{(0)},\mu,\bm{v},q,\bm{\omega})
}
Thus the equivalence~\eqref{eq:eqphiPhi} we get
\been{
\aled{
&\phi(\Psi^{(t)},\mu,q,\bm{\omega})=\sup_{\bm{v}\in D_n}\(t\Gamma(\Psi^{(1)},\mu,\bm{v},q,\bm{\omega})+(1-t)\Gamma(\Psi^{(0)},\mu,\bm{v},q,\bm{\omega})\)\\
&\leq t\sup_{\bm{v}\in D_n}\Gamma(\Psi^{(1)},\mu,\bm{v},q,\bm{\omega})+(1-t)\sup_{\bm{v}\in D_n}\Gamma(\Psi^{(0)},\mu,\bm{v},q,\bm{\omega})\\
&\leq t\phi(\Psi^{(1)},\mu,q,\bm{\omega})+(1-t)\phi(\Psi^{(0)},\mu,q,\bm{\omega}).
}
}
\end{proof}
We now provide the proof of the derivative formula. Given $\Delta \Psi\in B_n$ and $t\in [0,1]$, let
\been{
\Psi^{(t)}=\Psi+t\Delta\Psi,
}
\been{
\Delta_t\phi(q,\bm{\omega})=\phi(\Psi^{(t)},\mu,q,\bm{\omega}) -\phi(\Psi,\mu,q,\bm{\omega}),
}
and
\been{
\Delta_t \bm{r}(q,\bm{\omega})=\bm{r}(\Psi^{(t)},\mu,q,\bm{\omega}) -\bm{r}(\Psi,\mu,q,\bm{\omega}).
}
Given $q\in [0,1]$, let $ \Delta_t \bm{r}_q(q',\bm{\omega})=\bm{1}_{q'>q}\Delta_t \bm{r}(q',\bm{\omega})$. Let us denote by $(\phi,\bm{r})$, $(\phi^{(t)},\bm{r}^{(t)})$ respectively the solutions of the \BSDE corresponding to the final conditions $\Psi$ and $\Psi^{(t)}$. By Propositions~\ref{prop:uniqueness2}, the DDEs $\mathcal{E}(\mu \bm{r})$ and $\mathcal{E}(\mu \bm{r}^{(t)})$ are $\Fqfilt{q}{[0,1]}$--adapted $\Wn$--martingales.

We first give the following two preliminary results.
\begin{lemma}
\label{lem:martdeltar}
The process $I_{\mu\bm{r}^{(t)}}(\Delta_t\bm{r}_q)$ is a $\W_{\mu\bm{r}^{(t)}}$--square integrable martingale with mean $0$ and  $I_{\mu\bm{r}}(\Delta_t\bm{r}_q)$ is a $\W_{\mu\bm{r}}$--square integrable martingale with mean $0$. Moreover
\end{lemma}
\begin{proof}
By Proposition~\ref{prop:uniqueness2}, $\bm{r}^{(t)}\in \widehat{D}^{\mu}_n$ and Corollary~\ref{rem:uniqueness222} we have
\been{
\aled{
&\Ee{\mu\bm{r}^{(t)}}\[\int^1_qds \|\Delta_t\bm{r}(s,\bm{\omega})\|^2_2\]\leq e^{2\|\Psi^{(t)}\|_{\infty}}\E\[\int^1_qds \|\Delta\bm{r}(s,\bm{\omega})\|^2_2\]\\
&\leq e^{2\|\Psi^{(t)}\|_{\infty}}\(\E\[\int^1_0ds \|\bm{r}(s,\bm{\omega})\|^2_2\]+\E\[\int^1_0ds \|\bm{r}^{(t)}(s,\bm{\omega})\|^2_2\]\)\\
&\leq e^{2\|\Psi^{(t)}\|_{\infty}}(K(1,\|\Psi^{(t)}\|_{\infty})+K(1,\|\Psi\|_{\infty}))<\infty.
}
}
and, in the same way,
\been{
\aled{
&\Ee{\mu\bm{r}}\[\int^1_0ds \|\Delta_t\bm{r}(s,\bm{\omega})\|^2_2\]\leq e^{2\|\Psi\|_{\infty}}(K(1,\|\Psi^{(t)}\|_{\infty})+K(1,\|\Psi\|_{\infty}))<\infty.
}
}
Hence, Lemma~\ref{lem:mart_cond} completes the proof.
\end{proof}
We can now prove the derivative formula.
\begin{proposition}[Derivative on $\Psi$]
For any $\Psi\in B_n$, $\Delta\Psi\in B_n$ and $\mu\in \textup{Pr}([0,1])$, it holds
\begin{equation}
\label{first_derivative_complete_psi2}
\aled{
&\left.\frac{\partial \phi(\Psi+t\Delta \Psi,\mu,q,\bm{\omega})}{\partial t}\right|_{t=0}=\Eecond{\mu\bm{r}(\Psi,\mu)}{\Delta \Psi(\bm{\omega})}{q},\quad \Wnas.
}
\end{equation}
\end{proposition}
\begin{proof}
Using the \BSDE, after a bit of manipulation, we get
\been{
\label{Dif1}
\aled{
&\Delta_t \phi(q,\bm{\omega})=t\Delta\Psi(\bm{\omega})-I_{\mu\bm{r}^{(t)}}(\Delta_t\bm{r}_q,1,\bm{\omega})-\frac{1}{2}\int^1_qds\mu([0,s])\|\Delta_t \bm{r}(s,\bm{\omega})\|^2_2.
}
}
Manipulating the \BSDE in a different way, we also get
\been{
\label{Dif2}
\aled{
&\Delta_t \phi(q,\bm{\omega})=t\Delta\Psi(\bm{\omega})-I_{\mu\bm{r}}(\Delta_t\bm{r}_q,1,\bm{\omega})+\frac{1}{2}\int^1_qds\mu([0,s])\|\Delta_t\bm{r}(s,\bm{\omega})\|^2_2.
}
}
The random variable $\Delta_t \phi(q,\bm{\omega})$ $\Fq{q}$--measurable. Since, by Lemma~\ref{lem:martdeltar}, $I_{\mu\bm{r}^{(t)}}(\Delta_t\bm{r}_q)$ is a $\W_{\mu\bm{r}^{(t)}}$--square integrable martingale, \eqref{Dif1} gives
\been{
\label{Dif11}
\aled{
&\Delta_t \phi(q,\bm{\omega})=\Eecond{\mu\bm{r}^{(t)}}{\Delta_t \phi(q,\bm{\omega})}{q}\\
&=t\Eecond{\mu\bm{r}^{(t)}}{\Delta\Psi(\bm{\omega})}{q}-\frac{1}{2}\Eecond{\mu\bm{r}^{(t)}}{\int^1_qds\mu([0,s])\|\Delta_t\bm{r}(s,\bm{\omega})\|^2_2}{q}
}
}
Since, by Lemma~\ref{lem:martdeltar}, $I_{\mu\bm{r}}(\Delta_t\bm{r}_q)$ is a $\W_{\mu\bm{r}}$--square integrable martingale, \eqref{Dif2} gives
\been{
\label{Dif21}
\aled{
&\Delta_t \phi(q,\bm{\omega})=\Eecond{\mu\bm{r}}{\Delta_t \phi(q,\bm{\omega})}{q}\\
&=t\Eecond{\mu\bm{r}}{\Delta\Psi(\bm{\omega})}{q}+\frac{1}{2}\Eecond{\mu\bm{r}}{\int^1_qds\mu([0,s])\|\Delta_t\bm{r}(s,\bm{\omega})\|^2_2}{q}.
}
}
Thus, taking the difference from~\eqref{Dif21} and~\eqref{Dif11}, we get
\been{
\label{eq:subtract}
\aled{
&2t\(\Eecond{\mu\bm{r}^{(t)}}{\Delta\Psi(\bm{\omega})}{q}-\Eecond{\mu\bm{r}}{\Delta\Psi(\bm{\omega})}{q}\)\\
&=\Eecond{\mu\bm{r}}{\int^1_qds\mu([0,s])\|\Delta_t\bm{r}(s,\bm{\omega})\|^2_2}{q}+\Eecond{\mu\bm{r}^{(t)}}{\int^1_qds\mu([0,s])\|\Delta_t\bm{r}(s,\bm{\omega})\|^2_2}{q}.
}
}
By Proposition~\ref{prop:uniqueness2}, it holds
\been{
\mathcal{E}(\mu\bm{r}^{(t)},1,\bm{\omega}|q)\Delta\Psi(\bm{\omega})\leq e^{4\|\Psi^{(t)}\|_{\infty}}\leq  e^{4\|\Psi\|_{\infty}+4\|\Delta\Psi\|_{\infty}}<\infty.
}
Thus the Dominated Convergence Theorem gives
\been{
\lim_{t\to 0}\(\Eecond{\mu\bm{r}^{(t)}}{\Delta\Psi(\bm{\omega})}{q}-\Eecond{\mu\bm{r}}{\Delta\Psi(\bm{\omega})}{q}\)=0,\quad \Wnas.
}
Since the right-hand member~\eqref{eq:subtract} is a sum of two non-negative quantities, the above limit and the equality~\eqref{eq:subtract} gives
\been{
\lim_{t\to 0}\frac{1}{2 t}\Eecond{\mu\bm{r}}{\int^1_qds\mu([0,s])\|\Delta_t\bm{r}(s,\bm{\omega})\|^2_2}{q}=0,\quad \Wnas.
}
Using the above limit and~\eqref{Dif21}, we get
\been{
\lim_{t\to 0}\frac{\Delta_t \phi(q,\bm{\omega})}{ t}=\Eecond{\mu\bm{r}}{\Delta\Psi(\bm{\omega})}{q}.
}
\end{proof}
\subsection{Proof of the statement 3) of Theorem~\ref{th:func_quant}}
This subsection is dedicated to proving the derivative formula stated in the point $2)$ of Theorem~\ref{th:func_quant}. The results is an extension of the Lemma~\ref{derivative_theo}, proven when the Parisi parameters are in $\textup{Pr}([0,1])^{\circ}$, to arbitrary Parisi parameters in $\textup{Pr}([0,1])$. 
\begin{proposition}[Derivative on $\mu$]
\label{derivative_theo_coplete}
For any $\Psi\in B_n$, it holds
\begin{equation}
\label{first_derivative_complete}
\aled{
&\left.\frac{\partial \phi(\Psi,\mu^{(0)}+t(\mu^{(1)}-\mu^{(0)}),q,\bm{\omega})}{\partial t}\right|_{t=0}\\
&=\frac{1}{2}\widetilde{\mathbb{E}}_{\mu^{(0)}\bm{r}(\Psi,\mu^{(0)})}\left[\int^{1}_{q}ds\,(\mu^{(1)}-\mu^{(0)})([0,s])\,\|\bm{r}(\Psi,\mu,s,\bm{\omega})\|^2_2\Bigg| \mathcal{F}^{\otimes n}_q\right],\quad \Wnas.
}
\end{equation}
Moreover there exists a constant $K(\|\Psi\|_{\infty})$ (depending only on $\|\Psi\|_{\infty}$)  such that
\been{
\aled{
&\Eecond{\mu^{(0)}\bm{r}(\Psi,\mu^{(0)})}{\int^1_qdt(\mu^{(1)}-\mu^{(0)})([0,t])\|r(\Psi,\mu,t,\bm{\omega})\|^2_2}{q}\\
&\leq K(\|\Psi\|_{\infty})\sup_{t\in[0,1]}\|\mu^{(1)}-\mu^{(0)}\|_{\infty},\quad \Wnas.
}
}
\end{proposition}
Given two Parisi parameters $\mu^{(0)}\in \textup{Pr}([0,1])$ and $\mu^{(1)}\in \textup{Pr}([0,1])$, we define
\been{
(\phi^{(0)},\bm{r}^{(0)}):=(\phi(\Psi,\mu^{(0)}),\bm{r}(\Psi,\mu^{(0)})),\quad (\phi^{(1)},\bm{r}^{(1)}):=(\phi(\Psi,\mu^{(1)}),\bm{r}(\Psi,\mu^{(1)})),
}
and
\been{
\delta \phi:=\phi^{(1)}-\phi^{(0)},\quad \delta \bm{r}:=\bm{r}^{(1)} -\bm{r}^{(0)},\quad 
\delta\mu:=\mu^{(1)}-\mu^{(0)}
}
We first provide the extension of Lemma \ref{uniform_convergence_theo} in $\textup{Pr}([0,1])$.
\begin{proposition}
\label{deltaphir}
Let $\mu^{(1)}$ and $\mu^{(0)}$ be two elements of $\textup{Pr}([0,1])$. Then, for any $p\geq 1$, there exist a constant $b_p(\|\Psi\|_{\infty})$ depending only $p$ and $\|\Psi\|_{\infty}$ such that:
\begin{equation}
\label{eq:Lipr0complete}
\E\left[\left(\int^1_0dt\|\,\delta \bm{r}(q,\bm{\omega})\|^2_2\right)^p\right]^{\frac{1}{p}}\leq b_p(\|\Psi\|_{\infty})\|\delta \mu\|^2_{\infty}.
\end{equation}
\end{proposition}
\begin{proof}
If $\mu^{(0)}=\mu^{(1)}$ the proposition is proven by the uniqueness result~\ref{prop:uniqueness}. So we consider two distinct measures.
Let $(\nu^{(0)}_k)_{k\in \N}\subseteq \textup{Pr}([0,1])^{\circ}$ and $(\nu^{(1)}_k)_{k\in \N}\subseteq \textup{Pr}([0,1])^{\circ}$ be two sequences of discrete probability measures such that
\been{
\|\nu^{(0)}_k-\mu^{(0)}\|_{\infty}\leq 2^{-k-1}\|\delta\mu\|\leq 2^{-k},\quad \|\nu^{(1)}_k-\mu^{(1)}\|_{\infty}\leq 2^{-k-1}\|\delta\mu\|\leq 2^{-k}.
}
Let us also define
\been{
\delta\nu_k:=\nu^{(1)}_k-\nu^{(0)}_k.
}
Note that 
\been{
\label{eq:nu_ineq}
\|\delta\nu_k\|_{\infty}\leq \|\delta\mu\|_{\infty}+\|\nu^{(0)}_k-\mu^{(0)}\|_{\infty}+\|\nu^{(1)}_k-\mu^{(1)}\|_{\infty}\leq (1+2^{-k})\|\delta\mu\|_{\infty}.
}
Let us denote by $(\phi^{(0)},\bm{r}^{(0)})$, $(\phi^{(1)},\bm{r}^{(1)})$, $(\phi^{(0)}_{k},\bm{r}^{(0)}_{k})$, $(\phi_{1,k},\bm{r}^{(1)}_{k})$ respectively the solutions of the \BSDE corresponding to the Parisi parameters $\mu^{(0)}$, $\mu^{(1)}$, $\nu^{(0)}_k$ and $\nu^{(1)}_k$ and
\been{
\delta \bm{r}_k:=\bm{r}^{(1)}_{k}-\bm{r}^{(0)}_{k},\quad \Delta \bm{r}^{(0)}_k:=\bm{r}^{(0)}_{k}-\bm{r}^{(0)},\quad \Delta \bm{r}^{(1)}_k:=\bm{r}^{(1)}_{k}-\bm{r}^{(1)}.
}
The triangular inequality gives
\been{
\label{eq:triangle1}
\aled{
&\E\[\left\|\delta\bm{r}(q,\bm{\omega})\right\|^p\]^{\frac{1}{p}}\leq\E\[\left\|\Delta\bm{r}^{(0)}_{k}(q,\bm{\omega})\right\|^p\]^{\frac{1}{p}}+\E\[\left\|\Delta\bm{r}^{(1)}_{k}(q,\bm{\omega})\right\|^p\]^{\frac{1}{p}}+\E\[\left\|\delta\bm{r}_{k}(q,\bm{\omega})\right\|^p\]^{\frac{1}{p}}
}
}
By Proposition~\ref{convergence_result}, $\bm{r}^{(0)}_{k}$ and $\bm{r}^{(1)}_{k}$ converge in $D^{2,p}_n$ norm to $\bm{r}^{(0)}$ and $\bm{r}^{(1)}$ respectively. Thus there exists $k_{*}$ such that
\been{
\label{eq:triangle11}
\E\[\left\|\Delta\bm{r}^{(i)}_{k}(q,\bm{\omega})\right\|^p\]^{\frac{1}{p}}\leq \|\delta\mu\|^2_{\infty},\quad \forall k\geq k_{*},\,i\in \{0,1\}.
}
By Lemma~\ref{uniform_convergence_theo}, for any two measures in $\textup{Pr}([0,1])^{\circ}$, $\nu^{(0)}_k$ and $\nu^{(1)}_k$, it holds
\been{
\label{eq:triangle12}
\E\[\left\|\delta\bm{r}_{k}(q,\bm{\omega})\right\|^p\]^{\frac{1}{p}}\leq b_p(\|\Psi\|_{\infty})\|\delta\nu^{(k)}\|^2_{\infty}\leq b_p(\|\Psi\|_{\infty})(1+2^{-k})\|\delta\mu\|^2_{\infty},
}
where, in the least inequality, we used the inequality~\eqref{eq:nu_ineq}. The coefficient $b_p(\|\Psi\|_{\infty})$ does not depend on $k$. Hence, combining the inequalities~\eqref{eq:triangle1},~\eqref{eq:triangle11}, and~\eqref{eq:triangle12}, we prove the inequality~\eqref{eq:Lipr0complete}.
\end{proof}
The above result is the key to extending the derivative formula in Lemma~\ref{derivative_theo} to general $\mu\in \textup{Pr}([0,1])$. We now provide a bound of the finite difference.
\begin{lemma}
\label{deltaphieq}

The following inequality holds
\been{
\aled{
&\left|\delta\phi(q,\bm{\omega})-\frac{1}{2}\Eecond{\mu^{(1)}\bm{r}^{(1)}}{\int^1_qds\delta\mu([0,s])\|\bm{r}^{(0)}(\,s,\bm{\omega})\|^2_2}{q}\right|\\
&\leq\frac{1}{2}e^{2\|\Psi\|_{\infty}}\Econd{\int^1_qds\|\delta\bm{r}(s,\bm{\omega})\|^2_2}{q}.
}
}
\end{lemma}
\begin{proof}
 Using the \BSDE we have
\been{
\label{eq:diBSDE}
\aled{
&\delta \phi(q,\bm{\omega})\\
&=-I(\delta\bm{r}_q,1,\bm{\omega})+\frac{1}{2}\int^1_qds\mu^{(1)}([0,s])\|\bm{r}^{(1)}(s,\bm{\omega})\|^2_2-\frac{1}{2}\int^1_qds\mu^{(0)}([0,s])\|\bm{r}^{(0)}(\,t,\bm{\omega})\|^2_2\\
&=-I_{\mu^{(1)}\bm{r}^{(1)}}(\delta\bm{r}_q,1,\bm{\omega})+\frac{1}{2}\int^1_qds\delta\mu([0,s])\|\bm{r}^{(0)}(s,\bm{\omega})\|^2_2-\frac{1}{2}\int^1_qds\mu^{(1)}([0,s])\|\delta \bm{r}(s,\bm{\omega})\|^2_2.
}
}
By Proposition~\ref{prop:uniqueness2} $\bm{r}^{(1)}\in D^{\mu}_n$. Thus, by Lemma~\ref{lem:martcon2} $I_{\mu^{(1)}\bm{r}^{(1)}}(\bm{r}_q^{(1)})$ and $I_{\mu^{(1)}\bm{r}^{(1)}}(\bm{r}_q^{(0)})$ are $\W_{\mu^{(1)}\bm{r}^{(1)}}$--martingales with mean $0$. Hence $I_{\mu^{(1)}\bm{r}^{(1)}}(\delta\bm{r}_q)$ is  $\W_{\mu^{(1)}\bm{r}^{(1)}}$--martingale with mean $0$.
Thus, since $\delta \phi$ is adapted, \eqref{eq:diBSDE} gives
\been{
\aled{
&2\delta \phi(q,\bm{\omega})=2\Eecond{\mu^{(1)}\bm{r}^{(1)}}{\delta \phi(q,\bm{\omega})}{q}\\
&=\Eecond{\mu^{(1)}\bm{r}^{(1)}}{\int^1_qds\delta\mu([0,s])\|\bm{r}^{(0)}(s,\bm{\omega})\|^2_2}{q}\\
&+\Eecond{\mu^{(1)}\bm{r}^{(1)}}{\int^1_qds\mu^{(1)}([0,s])\|\delta \bm{r}(s,\bm{\omega})\|^2_2}{q}.
}
}
By Proposition~\ref{prop:uniqueness2} $\mathcal{E}(\mu^{(1)}\bm{r}^{(1)},1,\bm{\omega})\leq e^{2\|\Psi\|_{\infty}}$ and $\mu^{(1)}([0,t])\leq 1$, completing the proof.
\end{proof}
Given $t\in [0,1]$, let
\begin{equation}
\mu^{(t)}:=\mu^{(0)}+t \delta\mu\in \textup{Pr}([0,1]),\quad \bm{r}^{(t)}:=\bm{r}(\Psi,\mu^{(t)}),\quad \delta_{0,t}\bm{r}:=\bm{r}^{(t)}-\bm{r}^{(0)}.
\end{equation}
\begin{lemma}
\label{lem:convas}
For any $q\in [0,1]$
\been{
\lim_{t\to 0}\frac{1}{t}\Econd{\int^1_qds\|\delta_{0,t}\bm{r}(s,\bm{\omega})\|^2_2}{q}=0,\quad \Wnas.
}
\end{lemma}
\begin{proof}
Given $t\in [0,1]$, $t_1\in [0,t]$, and $t_2\in[t_1,t]$, let $\delta_{t_1,t_2}\bm{r}:=\bm{r}^{(t_1)}-\bm{r}^{(t_2)}$ and
\been{
U_{t_1,t_2}(q,\bm{\omega})=\frac{1}{t_2}\Econd{\int^1_qds\|\delta_{t_1,t_2}\bm{r}(s,\bm{\omega})\|^2_2}{q}
}
By Lemma~\ref{deltaphir}, for any $p\geq 1$ it holds
\been{
\label{eq:Uineq1}
\aled{
&\E\[|U_{t_1,t_2}(q,\bm{\omega})|^p\]\leq \frac{1}{t^p_2}\E\[\left|\int^1_qds\|\delta\bm{r}_{t_1,t_2}(s,\bm{\omega})\|^2_2\right|^p\]\leq b^p_p(\|\Psi\|_{\infty})\frac{\|\mu^{(t_1)}-\mu^{(t_2)}\|^{2p}_{\infty}}{t^p_2},
}
}
where, since $0\leq t_2-t_1\leq t_2$ and $2p-1\geq p$,
\been{
\label{eq:Uineq2}
\|\mu^{(t_1)}-\mu^{(t_2)}\|^{2p}_{\infty}\leq (t_1-t_2)^{2p}\|\delta\mu\|^{2p}_{\infty}\leq t^{p}_2(t_1-t_2)\|\delta\mu\|^{2p}_{\infty}.
}
Using the inequality
\been{
\aled{
\left|U_{0,t_1}(q,\bm{\omega})-U_{0,t_2}(q,\bm{\omega})\right|&\leq \frac{t_2-t_1}{t_2}U_{0,t_1}(q,\bm{\omega})+U_{t_1,t_2}(q,\bm{\omega}).
}
}
and the inequality $2(|a|+|b|)^p\leq 2^p |a|^p+2^p |b|^p$, for any decreasing sequence $(t_k)_{k\in \N}$ converging to $0$ we have
\been{
\aled{
&\E\[|U_{0,t_{k+1}}(q,\bm{\omega})-U_{0,t_k}(q,\bm{\omega})|^p\]\\
&\leq 2^{p-1}\(\frac{(t_{k}-t_{k+1})^p}{ t^p_k}\E\[|U_{0,t_{k+1}}(q,\bm{\omega})|^p\]+\E\[|U_{t_{k+1},t_k}(q,\bm{\omega})|^p\]\)\\
&\leq 2^pb^p_p(\|\Psi\|_{\infty})\|\delta\mu\|^{2p}_{\infty}(t_{k}-t_{k+1}).
}
}
Thus
\been{
\aled{
\sum_{k\in \N}\E\[|U_{0,t_{k+1}}(q,\bm{\omega})-U_{0,t_k}(q,\bm{\omega})|^p\]\leq 2^pb^p_p(\|\Psi\|_{\infty})\|\delta\mu\|^{2p}_{\infty}t_1<\infty,\quad \forall p\geq 0.
}
}
In particular, if we take $p=1$ and $p=2$, the above inequality implies that both the series of the expectation values and of the variances converge. By Kolmogorov's two-series Theorem, $\sum_{k\in \N}|U_{0,t_{k+1}}(q,\bm{\omega})-U_{0,t_k}(q,\bm{\omega})|$ converges $\Wnas$. Thus $\Wnas$ it holds 
\been{ 
\lim_{M\to\infty}\sup_{m>M}|U_{0,t_{M}}(q,\bm{\omega})-U_{0,t_m}(q,\bm{\omega})|\leq \lim_{M\to\infty}\sum^{\infty}_{k=M}|U_{0,t_{k+1}}(q,\bm{\omega})-U_{0,t_k}(q,\bm{\omega})|=0.
}
So $(U_{0,t_{k}})_{k\in\N}$ is a Cauchy sequence $\Wnas$. Moreover, the inequalities \eqref{eq:Uineq1} and  \eqref{eq:Uineq2} give
\been{
\lim_{t\to 0}\E\[|U_{0,t}(q,\bm{\omega})|\]=0.
}
It follows that $U_{0,t}$ converges to $0$ in probability as $t\to 0$. Since, for any decreasing sequence $t_{k}\downarrow 0$, $(U_{0,t_{k}})_{k\in\N}$ is Cauchy, then $U_{0,t}$ converges to $0$ $\Wnas$.

\end{proof}
Combing Lemma~\ref{deltaphieq} and~\ref{lem:convas}, we can finally extend Lemma~\ref{derivative_theo} to general $\mu\in \textup{Pr}([0,1])$.

\begin{proof}[Proof of Proposition~\ref{first_derivative_complete}]
Let $\phi^{(t)}=\phi(\Psi,\mu^{(t)})$ and $\bm{r}^{(t)}=\bm{r}(\Psi,\mu^{(t)})$ and
\been{
A_{t,q}(\bm{\omega})=\mathcal{E}(\mu^{(t)}\bm{r}^{(t)},q,\bm{\omega})\int^1_qds\delta\mu([0,s])\|\bm{r}^{(0)}(s,\bm{\omega})\|^2_2.
}
By Lemma~\ref{deltaphieq} and Lemma~\ref{lem:convas}
\been{
\lim_{t\to 0}\(\frac{1}{t}\delta\phi^{(t)}(q,\bm{\omega})-\frac{1}{2}\Econd{A_{t,q}(\bm{\omega})}{q}\)=0,\quad \Wnas.
}
Using Proposition~\ref{prop:uniqueness2}, we have
\been{
\label{eq:xiiii}
A_{t,q}(\bm{\omega})\leq e^{2\|\Psi\|_{\infty}}\int^1_0ds\|\bm{r}^{(t)}(s,\bm{\omega})\|^2_2.
}
By Corollary~\ref{rem:uniqueness222}, the random variable $\int^1_qds\|\bm{r}^{(t)}(s,\bm{\omega})\|^2_2$ is $\Wn$--integrable. Thus the Dominated Convergence Theorem gives
\been{
\lim_{t\to 0 }\Econd{A_{t,q}(\bm{\omega})}{q}=\Econd{A_{0,q}(\bm{\omega})}{q},\quad \Wnas.
}
Thus, the limit $\lim_{t\to 0}\frac{1}{t}\delta\phi^{(t)}(q,\bm{\omega})$ exists almost surely and it is equal to $\Econd{\Xi_{0,q}(\bm{\omega})}{q}$. Finally, Corollary~\ref{rem:uniqueness222} and the inequality \eqref{eq:xiiii} completes the proof.
\end{proof}
\section{Proof of Theorem~\ref{th:The_Theorem}}
\label{sec:10}
This section is devoted to the proof of Theorem~\ref{th:The_Theorem}.

We briefly recall the $K$--RSB formalism introduced in Section~\ref{sec:intro}, and clarify some of the associated topological aspects.

As before, let $\mathcal{M}^{(K+1)} := [-1,1]$ endowed with the standard topology on $\mathbb{R}$ and the associated Borel $\sigma$--algebra. For each $l\in [K]_0$, define $\mathcal{M}^{(l)}$ as the space of probability measures on $\mathcal{M}^{(l+1)}$, equipped with the weak topology of measures and the associated Borel $\sigma$--algebra. We then define the product space
\begin{equation}
\label{eq:product-space-M}
\mathcal{M}^{[K+1]_0} := \mathcal{M}^{(0)} \times \cdots \times \mathcal{M}^{(K+1)}.
\end{equation}
Under these choices for the $\sigma$--algebras, each $\mathcal{M}^{(l)}$ is a standard Borel space, and so is the product space $\mathcal{M}^{[K+1]_0}$ (or details on the choice of the $\sigma$--algebra, see Remark~\ref{rem:topo}).

Given $\zeta \in \mathcal{M}^{(0)}$, define the random sequence
\begin{equation}
\label{eq:random-sequence2}
W = (M^{(0)}, \cdots, M^{(K+1)}),
\end{equation}
as in~\eqref{eq:random-sequence}, starting from $M^{(0)} = \zeta$, and denote by $\mathbb{M}_{\zeta}$ its distribution. Let $\bm{W} = (W_1, \cdots, W_n)$
denote a collection of independent realizations of the sequence~\eqref{eq:random-sequence2}.

For each $l \in [K+1]_0$, we denote by $\mathcal{A}_l$ the $\sigma$--algebra generated by the tuple $(M^{(0)}, \cdots, M^{(l)})$, and by $\mathcal{A}^{\otimes n}_l$ the $n$--fold product $\sigma$--algebra, as defined in~\eqref{eq:delAl}.

For a given increasing sequence $x=(x_l)_{l\in [K+1]_0}$ defined as in~\eqref{eq:x_seq0} and a real-valued, bounded, and $\mathcal{A}^{\otimes n}_{K+1}$--measurable random variable $\Psi$, let $\widehat{\Phi}(\Psi, x)$ the quantity defined in~\eqref{eq:phihat}.

We also recall the definition of the Wiener filtered probability space
\been{
(C([0,1],\R^n),\mathcal{F}^{\otimes n},\mathbb{W}^{\otimes n},\(\mathcal{F}^{\otimes n}_q\)_{q\in[0,1]}),
}
defined in Sections~\ref{sec:intro} and~\ref{sec:3} and used throughout the manuscript.

The proof of Theorem~\ref{th:The_Theorem} is based on the following intermediate result.
\begin{proposition}
\label{prop:equivphiphi}
Let $K\in \N$ and let $\mathcal{M}^{[K+1]_0}$ be the space defined in~\eqref{eq:product-space-M}. Let $x=(x_l)_{l\in [K+1]_0}$ be the increasing sequence defined as in~\eqref{eq:x_seq0}. 

Let
\been{
\label{eq:q_seq2}
0=q_0\leq q_1\leq\cdots\leq q_K\leq q_{K+1}=1,
}
and define a discrete Parisi parameter $\mu \in \textup{Pr}([0,1])^{\circ}$ as in~\eqref{discrete}:
\begin{equation}
\label{discrete2}
\mu([0,q]):=\sum_{l\in [K+1]_0} x_l\bm{1}_{q\in [q_{l-1},q_{l})}+\bm{1}_{q=1}.
\end{equation}
 Given $\zeta\in M^{(0)}$ and $W\sim \mathbb{M}_{\zeta}$ be the $\mathcal{M}^{[K+1]_0}$--valued random sequence defined in~\eqref{eq:random-sequence}. Then
there exists a $\mathcal{F}_1$--measurable Wiener functional $\bar{m}:C([0,1],\R^n)\to \mathcal{M}^{[K+1]_0}$ such that
\begin{enumerate}
    \item 
    $
    W\overset{d}{=}\bar{m}(\omega).
    $
    \item Given $n\in \N$ and a bounded, real-valued, and $\mathcal{A}^{\otimes n}_{K+1}$--measurable functional $\Psi:\mathcal{M}^{[K+1]_0}\to \R$, define \\
    \been{
    \widetilde{\Psi}(\omega_1,\cdots,\omega_n)=\Psi(\bar{m}(\omega_1),\cdots,\bar{m}(\omega_n)).
    }
    Then
    \been{
\label{eq:equiv_to_prove}
\widehat{\Phi}(\Psi, x)=\Phi(\widetilde{\Psi},\mu,0),
}
where $\widehat{\Phi}$ is the operator defined in~\eqref{eq:phihat} and $\Phi$ is the $\mu$--RSB expectation defined in~\eqref{eq:auxiliary_func}.
\end{enumerate}
\end{proposition}
In order to prove the first statement, we will first establish the equivalence in distribution of $W$ with a suitable measurable functional on $(0,1)^{K+1}$, and then we obtain the result on the Wiener space. The proof of the second statement relies on the specific structure of functional $\bar{m}$; once the proposition is established, the proof of Theorem~\ref{th:The_Theorem} follows directly.

The following classical result from probability theory plays a crucial role in proving the first statement of the proposition. For a proof, we refer the reader to \cite[Lemma 3.1]{Austin} or \cite[Theorem 5.10]{Kallenberg}.
\begin{lemma}[Noise-Outsourcing Lemma]
\label{lem:NOS}
Let $T$ and $S$ be standard Borel spaces, and let $(Y,X)$ be a random variable taking value on $T \times S$. Then (possibly after enlarging the probability space) there exist a random variable $U \sim \mathrm{Uniform}(0,1)$, independent of $X$, and a Borel-measurable function $f: T \times (0,1) \to S$ such that
\been{
( Y,X) = \big(Y, f(X, U)\big) \quad  \text{a.s.}.
}
\end{lemma}
The generalization of the above lemma to sequences of arbitrary length is straightforward. 
\begin{lemma}[Recursive Noise-Outsourcing for sequences]
\label{lem:NOS02}
Let $K\in \N$ and $S^{(0)}, \cdots, S^{(K+1)}$ be standard Borel spaces. For any random sequence $(X^{(l)})_{l\in [K+1]_0}$ taking value on the product space $S^{(0)}  \times \cdots \times S^{(K+1)}$, there exists (possibly after enlarging the probability space) there exists a sequence of independent random variables $(U^{(l)})_{l\in [K+1]_0}$ uniformly distributed on $(0,1)$ and a sequence of Borel-measurable functions
\been{
\(g_l: (0,1)^{l+1} \to S^{(l)}\)_{l\in [K+1]_0}
}
such that
\begin{equation}
\label{eq:recursive_representation}
\aled{
(X^{(0)}, \cdots, X^{(K+1)})= \left(g_0(U^{(0)}),\, \cdots,\, g_{K+1}(U^{(0)}, \cdots, U^{(K+1)})\right), \quad \text{a.s.}.
}
\end{equation}
\end{lemma}
\begin{proof}
We will first extend the Noise Outsourcing Lemma~\ref{lem:NOS} for sequences of arbitrary length. We will prove that, for any $K\in \N$, there exists a sequence of Borel-measurable functions
\been{
\(f_l: S^{(0)} \times [0,1)^{l} \to S^{(l)}\)_{l\in[K+1]},
}
such that
\been{
\label{eq:recursive_representation0}
\aled{
&(X^{(0)}, \cdots, X^{(K+1)})\\
&= \left(X^{(0)},\, f_1(X^{(0)}, U^{(1)}), \cdots,\, f_{K+1}(X^{(0)}, U^{(1)}, \cdots, U^{(K+1)})\right) \quad \text{a.s.}.
}
}
We proceed by induction on $K$. The base case $K=0$ is the classical Noise-Outsourcing Lemma~\ref{lem:NOS}.

Assume the claim holds for sequences of length $K+1$. Given $(X^{(0)},\cdots,X^{(K+1)})$, define 
\begin{equation}
\label{eq:Y_def}
Y := (X^{(0)},X^{(1)}) \in S^{(0)} \times S^{(1)}.
\end{equation}
The sequence $(Y,X^{(2)},\cdots X^{(K+1)})$ contains $K+1$ elements. The induction hypothesis (I.H) there exists a sequence of $K$ Borel functions $\widetilde{f}_2,\cdots ,\widetilde{f}_{K+1}$ such that
\begin{equation}
\label{eq:IH_representation}
(Y, X^{(2)}, \cdots, X^{(K+1)}) = \bigl(Y, \widetilde{f}_2(Y,U^{(2)}), \cdots, \widetilde{f}_{K+1}(Y,U^{(2)},\cdots,U^{(K+1)})\bigr) \quad \text{a.s.}
\end{equation}
Now, we can apply the Noise-Outsourcing Lemma~\ref{lem:NOS} on the pair $(X^{(0)},X^{(1)})$:
\begin{equation}
\label{eq:NOSit}
(X^{(0)},X^{(1)}) \overset{\eqref{lem:NOS}}{=}\widetilde{f}_1(X^{(0)},U^{(1)}), \quad \text{a.s.}.
\end{equation}
for some Borel function $f_1$. The variable $U^{(1)}$ is independent of $U^{(2)},\cdots,U^{(K+1)}$ since the pair $(X^{(0)},X^{(1)})$ is independent of $U^{(2)},\cdots,U^{(K+1)}$. For $l \ge 2$, set
\begin{equation}
\label{eq:f_l_def}
f_l(X^{(0)}, U^{(1)}, \cdots, U^{(l)}) := \widetilde{f}_l\bigl(f_1(X^{(0)}, U^{(1)}), U^{(2)}, \cdots, U^{(l)}\bigr).
\end{equation}
Then we obtain the representation
\been{
\label{eq:sequence_representation}
\aled{
&(X^{(0)},\cdots, X^{(K+1)})\\
&=\left(X^{(0)},\, f_1(X^{(0)}, U^{(1)}), \cdots,\, f_{K+1}(X^{(0)}, U^{(1)}, \cdots, U^{(K+1)})\right), \quad \text{a.s.},
}
}
Thus, if the~\eqref{eq:recursive_representation0} holds for a sequence of $K+1$ random variables, then it holds for a sequence of $K+2$ random variables, proving the induction hypothesis and thus the equality~\eqref{eq:recursive_representation0}.

Finally, the random variable $X^{(0)}$ can be represented as a measurable function $g_0$ of a  random variable $U^{(0)}$ uniformly distributed on $(0,1)$ and independent of $U^{(1)},\cdots, U^{(K+1)}$. Thus we get
\been{
\label{eq:Xg}
\aled{
&(X^{(0)},\cdots, X^{(K+1)})=\left(g_0(U^{(0)}),\, \cdots,\, g_{K+1}(U^{(0)}, \cdots, U^{(K+1)})\right), \quad \text{a.s.},
}
}
where for $l> 0$
\been{
g_l(U^{(0)}, \cdots, U^{(l)}):=f_l(g_0(U^{(0)}), U^{(1)}, \cdots, U^{(l)}),
}
completing the proof.
\end{proof}
Note that $\sigma(U^{(0)}, \cdots, U^{(l)})$ is not, in general, equal to $\sigma(X^{(0)}, \cdots,X^{(l)})$. However, under suitable conditions, the corresponding conditional expectations can be shown to coincide.
\begin{lemma}
\label{lem:NOU1}
Let $F$ be a bounded real-valued random variable measurable with respect to $\sigma(X^{(0)}, \cdots, X^{(K+1)})$. Then, for every $l \in [K+1]$, we have almost surely
\begin{equation}
\begin{aligned}
\label{eq:NOU1}
&\E\[F\middle| \sigma(U^{(0)}, \cdots, U^{(l)})\] = \E\[F\middle|\sigma(X^{(0)},\cdots, X^{(l)}) \],\quad \textup{a.s.}.
\end{aligned}
\end{equation}
\end{lemma}
\begin{proof}
We first introduce a bit of notation. Let $S:=S^{(0)}\times \cdots \times S^{(K+1)}$ and define
\been{
g:=(g_0,\cdots,g_{K+1}),\quad g^{[l_1,l_2]}:=(g_{l_1},\cdots,g_{l_2}),
}
and
\been{
U:=(U^{(0)},\cdots ,U^{(K+1)}),\quad \sigma(U^{[0,l]}):=\sigma(U^{(0)},\cdots,U^{(l)}).
}
We also denote by $\sigma({g}^{[0,l]})$ the $\sigma$--algebra generated by $g_0(U^{(0)}),\cdots,g_l(U^{(0)},\cdots, U^{(l)})$ and $\sigma(g):=\sigma({g}^{[1,K+1]})$. Note that $\sigma({g}^{[0,l]})\subseteq \sigma(U^{[0,l]})$ for any $l\in [K+1]_0$.

Since $F$ is bounded, real-valued, and measurable with respect to $\sigma(X^{(0)},\cdots, X^{(K+1)})$ and $S$ is a standard Borel space, then, by Lemma~\ref{lem:NOS02}, there exists a bounded, real-valued, and $\sigma(g)$ measurable function such that
\been{
\label{eq:Ff}
F=f\(U\), \quad \text{a.s.},
}
and
\been{
\label{eq:Ff2}
 \E\[F\middle|\sigma(X_1,\cdots, X_{l}) \]=\E\[f\(U\)\middle|\sigma({g}^{[0,l]}) \], \quad \text{a.s.}.
}
Hence, it suffice to prove the equivalence between $\E\[f\(U\)\middle|\sigma({g}^{[0,l]}) \]$ and $\E\[f\(U\)\middle|\sigma({U}^{[0,l]}) \]$.

We will first provide the proof for the indicator functions of any set $C\in \sigma(g)$
\been{
\label{eq:1c_toprove}
\E\[\bm{1}_{C}\middle|  \sigma(U^{[0,l]})\]=\E\[\bm{1}_{C}\middle| \sigma({g}^{[0,l]})\],\quad \forall l\in [K+1]_0.
}
The collection of measurable sets  $C\in  \sigma({g}^{[0,l]})$ for which the above equality holds forms a Dynkin system. By Dynkin $\pi-\lambda$ Theorem, t therefore suffices to verify the equality on a class of sets that is closed under finite intersections and generates $\sigma(\bar{g}^{[0,l]})$. To this end, we consider the following family of measurable subsets of $(0,1)^{K+2}$
\been{
\mathcal{C}:=\left\{g^{-1}(A^{(0)}\times \cdots \times A^{(K+1)})\middle| A^{(l)}\in (S^{(l)})',\quad \forall l\in [K+1]_0\right\},
}
where $(S^{(j)})'$ is the Borel $\sigma-$algebra of $S^{(j)}$. For any measurable cylindrical set $A^{(0)}\times \cdots\times A^{(K+1)}$ we have 
\been{
\aled{
&\E\[\bm{1}_{g^{-1}(A^{(0)}\times \cdots\times A^{(K+1)})}\middle| \sigma(U^{[0,l]})\]\\
&=\bm{1}_{({g}^{[0,l]})^{-1}(A^{(0)}\times \cdots\times A^{(l)})}\int_{(0,1)^{K+1-l}}du_{l+1}\cdots du_{K+1} \bm{1}_{({g}^{[l+1,K+1]})^{-1}(A^{(l+1)}\times \cdots \times A^{(K+1)})}.
}
}
The above function is clearly $\sigma(\bar{g}^{[0,l]})$--measurable. Thus
\been{
\E\[\bm{1}_{{g}^{-1}(A^{(0)}\times \cdots \times A^{(K+1)})}\middle| \sigma(U^{[0,l]})\]=\E\[\bm{1}_{{g}^{-1}(A^{(0)}\times\cdots \times A^{(K+1)})}\middle| \sigma(\bar{g}^{[0,l]})\].
}
So the equality~\eqref{eq:1c_toprove} holds for any set in $\mathcal{C}$. Hence, by Dynkin Theorem, it holds for any set in $\sigma(\mathcal{C})=\sigma(g)$. Thus, by the linearity of the expectation value, the Monotone Convergence Theorem and the Monotone Class Theorem for functions, we conclude that, for any bounded, real-valued, and $\sigma(g)$--measurable function $f$, the following identity holds almost surely
\been{
\E\[f\(U\)\middle|  \sigma(U^{[0,l]})\] = \E\[f\(U\)\middle|\sigma(g^{[0,l]}) \].
}
Finally, the almost sure equivalences~\eqref{eq:recursive_representation} and~\eqref{eq:Ff} complete the proof. 
\end{proof}
We will now use the representation~\eqref{eq:recursive_representation} to construct a sequence of Wiener functionals that is equivalent in distribution to the sequence $(X^{(l)})_{l\in [K+1]_0}$.

Define a $\(\mathcal{F}^{\otimes n}_q\)_{q\in[0,1]}$--adapted process $(q,\bm{\omega})\mapsto u(q,\omega)$ by
\been{
\label{eq:defu}
u(q,\omega):=\erf\(\omega(0)\)\bm{1}_{q= 0}+\sum^{K}_{l=0}\erf\(\frac{\omega\(q\)-\omega\(q_l\)}{\sqrt{q-q_l}}\)\bm{1}_{q\in \(q_l,q_{l+1}\]},
}
where the sequence $(q_0,\cdots,q_{K+1})$ is defined in~\eqref{eq:q_seq2} and $x\mapsto\erf(x)$ is the function
\been{
\label{eq:erf}
    \erf(x)=\int^{x}_{-\infty}\frac{dy}{\sqrt{2\pi}}e^{-\frac{y^2}{2}}.
}
So, given a sequence of $K+2$ independent random variable uniformly distributed on $(0,1)$, we have 
\been{
\label{eq:uu}
(u(q_0,\omega),\cdots,u(q_{K+1},\omega))\overset{d}{=}( {U}^{(0)}, \dots,U^{(K+1)}).
}
\begin{lemma}
\label{lem:equivUuu}
Let $\phi$ be a real-valued, bounded, and $\sigma(U^{(0)},\cdots,U^{(K+1)})$--measurable random variable and $\widetilde{\phi}$ be a real-valued, bounded, and $\sigma(u(q_0,\omega),\cdots,u(q_{K+1},\omega))$--measurable random variable. If
\been{
\label{eq:uuphi}
\phi\overset{d}{=}\widetilde{\phi}
}
then
\begin{equation}
\label{eq:lem-equivUuu-statement}
\aled{
\E\left[\phi \,\middle|\, \sigma(X^{(0)},\cdots,X^{(l)})\right]\overset{d}{=}\E\left[\widetilde{\phi} \,\middle|\, \mathcal{F}_{q_l}\right].
}
\end{equation}
\end{lemma}
\begin{proof}
Lemma~\ref{lem:NOU1}, the equivalences in distribution~\eqref{eq:uu} and~\eqref{eq:uuphi} imply:
\been{
\label{eq:uu2}
\aled{
&\E\left[\phi \,\middle|\, \sigma(X^{(0)},\cdots,X^{(l)})\right]\underset{\ref{lem:NOU1}}{=}\E\left[\phi \,\middle|\, \sigma(U^{(0)},\cdots,U^{(l)})\right]\overset{d}{=}\E\left[\widetilde{\phi} \,\middle|\, \sigma(u(q_0,\omega),\cdots,u(q_l,\omega))\right].
}
}
Since $u$ is adapted, then $\sigma(u(q_0,\omega),\cdots,u(q_l,\omega))\subseteq \mathcal{F}_{q_l}$. Moreover, from the definition of $u$ in~\eqref{eq:defu}, for each $l\in [K+1]$, $u(q_l,\omega)$ depends only on the Brownian increment $\omega(q_l)-\omega(q_{l-1})$. Since the Brownian motion has independent increments, the collection $\(u(q_j,\omega)\)_{j>l}$ is independent of $\mathcal{F}_{q_l}$. Thus, for any measurable set $B\in (0,1)$ and $j>l$
\been{
\aled{
&\E\left[\bm{1}_{\{u(q_j,\omega) \in B\}}\]=\E\left[\bm{1}_{\{u(q_j,\omega) \in B\}} \,\middle|\, \sigma(u(q_0,\omega),\cdots,u(q_l,\omega))\right]=\E\left[ \bm{1}_{\{u(q_j,\omega) \in B\}} \,\middle|\, \mathcal{F}_{q_l}\right].
}
}
Hence, for any measurable rectangle $B_0\times \cdots\times B_{K+1} \subset (0,1)^{K+2}$, we have
\begin{equation}
\label{eq:lem-equivUuu-independence}
\aled{
&\E\left[\prod_{r=0}^{K+1} \bm{1}_{\{u(q_r,\omega) \in B_r\}} \,\middle|\, \sigma(u(q_0,\omega),\cdots,u(q_l,\omega))\right]\\
&= \prod_{r=0}^{l} \bm{1}_{\{u(q_r,\omega) \in B_r\}} \cdot \E\left[\prod_{r=l+1}^{K+1} \bm{1}_{\{u(q_r,\omega) \in B_r\}} \,\middle|\, \mathcal{F}_{q_l}\right]
=\E\left[\prod_{r=0}^{K+1} \bm{1}_{\{u(q_r,\omega) \in B_r\}} \,\middle|\, \mathcal{F}_{q_l}\right].
}
\end{equation}
Since the collection of the set of the form
\been{
\{\omega;\, \(u(q_0,\omega),\cdots u(q_{K+1},\omega)\)\in B_0\times \cdots\times B_{K+1}\} 
}
is closed under finite intersections and generate $\sigma(u(q_0,\omega),\cdots,u(q_l,\omega))$, then the Functional Monotone Class Theorem extends the equivalence~\eqref{eq:lem-equivUuu-independence} to all bounded measurable, real-valued, and $\sigma(u(q_0,\omega),\cdots,u(q_l,\omega))$--measurable random variables $\phi$. Finally, using~\eqref{eq:uu2}, we complete the proof.
\end{proof}
Combining the above three lemmas and the equivalence in distribution~\eqref{eq:uu}, we obtain the following representation for the sequence $W$. In the following lemma, we use the notation
\been{
\mathcal{M}^{[0,l]}:=\mathcal{M}^{(1)}\times \cdots \times \mathcal{M}^{(l)}.
}
\begin{lemma}
\label{lem:equivWh}
Given $\zeta\in \mathcal{M}^{(0)}$, let $W$ be a $\mathbb{M}_{\zeta}$--distributed random sequence defined as in~\eqref{eq:random-sequence}. Then
\been{
\label{eq:equivWh2}
W\overset{d}{=}\left(m^{(0)}\(0,\omega\), \cdots, m^{(K+1)}\(q_{K+1},\omega\)\right).
}
for some collection of random variables $m^{(l)}(q_l,\omega) \in \mathcal{M}^{(l)}$, where for each $ l \in [K+1]_0$, the map $\omega \mapsto m^{(l)}(q_l,\omega)$ is $\mathcal{F}_{q_l}$--measurable

Moreover, for any $l\in [K+1]_0$ and any pair of random variables $\phi$ and $\widetilde{\phi}$ such that:
\begin{enumerate}
\item  $\phi$ is bounded, real-valued, and $\mathcal{A}_{K+1}$--measurable,
\item $\widetilde{\phi}$ is bounded, real-valued, and $\sigma\left(m^{(0)}\(0,\omega\), \cdots, m^{(K+1)}\(q_K,\omega\)\right)$--measurable,
\item $\phi\overset{d}{=}\widetilde{\phi}$,
\end{enumerate}
it holds that
\been{
\label{eq:equivcond0000}
\aled{
&\E\left[\phi\middle|\mathcal{A}_l\right]\overset{d}{=}\E\left[\widetilde{\phi}\middle|\mathcal{F}_{q_l}\right],\quad \textup{a.s.}
}
}
\end{lemma}
\begin{remark}
    In this specific case $m^{(0)}\(0,\omega\)=\zeta$ $\mathbb{W}$--almost surely.
\end{remark}
\begin{proof}
By Lemma~\ref{lem:NOS02} and the equivalence in distribution~\eqref{eq:uu} there exists a sequence of measurable function $\(g_l:(0,1)^{l+1}\to \mathcal{M}^{(l)}\)_{l\in [K+1]_0}$ such that
\been{
\label{eq:equivW}
W\overset{\ref{lem:NOS02}}{=}\(g_0(U^{(0)}),\cdots,g_l(U^{(0)},\cdots,U^{(K+1)})\)\overset{d}{=}\(m^{(0)}\(q_{0},\omega\),\cdots,m^{(K+1)}\(q_{K+1},\omega\)\)
}
where
\been{
\label{eq:equivHG}
m^{(l)}\(q_{l},\omega\)=g_{l}\(u(0,\omega),u\(q_1,\omega\),\cdots,u\(q_{l},\omega\)\),\quad \forall l\in [K+1]_0.
}
The Wiener function $m^{(l)}\(q_{l},\omega\)$ is $\mathcal{F}_{q_l}$--measurable, proving the statement~\eqref{eq:equivWh2}. Using the definition~\eqref{eq:equivHG}, the statement~\eqref{eq:equivcond0000} is proved by Lemma~\ref{lem:equivUuu}.
\end{proof}

We can finally provide the proof of Proposition~\ref{prop:equivphiphi}
\begin{proof}[Proof of Proposition~\ref{prop:equivphiphi}]
We define $ \bar{m}:C([0,1],\R)\to \mathcal{M}^{[K+1]_0}$ as:
  \been{
 \bar{m}(\omega):=\left(m^{(0)}\(0,\omega\), \cdots, m^{(K+1)}\(q_K,\omega\)\right),
 }
where the above sequence is the one defined in~\eqref{eq:equivWh2}. So, by Lemma~\ref{lem:equivWh}, $\bar{m}(\omega)\overset{d}{=}W$, proving the first statement. Hence we take
\been{
\label{eq:hoice}
\widetilde{\Psi}(\bm{\omega}):=\Psi(\bar{m}(\omega_1),\cdots ,\bar{m}(\omega_n)).
}
Since $\Psi$ is bounded and $\mathcal{A}^{\otimes n}_{K+1}$--measurable and $\bar{m}$-- is $\mathcal{F}_1$--measurable, $\widetilde{\Psi}$ is bounded and $\mathcal{F}_1$--measurable.

Combining Theorem~\ref{th:func_cont} with the explicit solution to the BSDE~\eqref{selfEq0Aux} for a discrete Parisi parameter given in Proposition~\ref{prop:existk}, we get that 
\been{
\label{eq:rec10bis}
\Phi(\widetilde{\Psi},\mu,0)=\E\[\frac{1}{x_{1}}\log(\Xi_{1}(\bm{\omega}))\],
}
where $\Xi_{1}(\bm{\omega})$ is defined iteratively as follows
\been{
\Xi_{l}(\bm{\omega}):=
\begin{dcases}
    e^{\widetilde{\Psi}(\bm{\omega})}, &\textup{if }l=K+1;\\
    \Econd{\(\Xi_{l+1}(\bm{\omega})\)^{\frac{x_{l}}{x_{l+1}}}}{q_{l}},\, &\textup{if }l\in [K].
\end{dcases}
}
The operator $\widehat{\Phi}$ is defined in a similar way in~\eqref{eq:recursive-T0} and~\eqref{eq:phihat}, with 
\been{
\widehat{\Phi}(\Psi, x)=\E\[\frac{1}{x_1}\log(\widehat{\Xi}_{1}(\Psi, x, \bm{W}))\].
}
and
\begin{equation}
\widehat{\Xi}_{l}(\Psi, x, \bm{W}) :=
\begin{dcases}
   e^{\Psi(\bm{W})}, & \text{if } l = K+1; \\
   \E\left[ \left. \widehat{\Xi}_{l+1}(\Psi, x,\bm{W})^{\frac{x_{l}}{x_{l+1}}} \right| \mathcal{A}^{\otimes n}_{l} \right] , & \text{if } l\in[K].
\end{dcases}
\end{equation}
Since, by~\eqref{eq:equivWh2}, $(\bar{m}(\omega_1),\cdots \bar{m}(\omega_n))\overset{d}{=}\bm{W}$ and $\Psi$ is $\mathcal{A}_{K+1}$--measurable and bounded, then, by the definition~\eqref{eq:hoice},
\been{
\Xi_{K+1}(\bm{\omega})=e^{\widetilde{\Psi}(\bm{\omega})}\overset{d}{=}e^{\Psi(\bm{W})}=\widehat{\Xi}_{K+1}(\Psi, x,\bm{W}).
}
We now proceed by backward induction on $ l \in [K+1]_0$. Assume that for some $ l+1 $, we have
\been{
\Xi_{l+1}(\bm{\omega})\overset{d}{=}\widehat{\Xi}_{l+1}(\Psi, x,\bm{W}),
}
then, since $\Xi_{l+1}$ and $\widehat{\Xi}_{l+1}$ are bounded, Lemma~\ref{lem:equivWh} yields
\been{
\Xi_{l}(\bm{\omega})=\Econd{\(\Xi_{l+1}(\bm{\omega})\)^{\frac{x_{l}}{x_{l+1}}}}{q_{l}}\overset{d}{\underset{\eqref{eq:equivcond0000}}{=}}\E\[\(\widehat{\Xi}_{l+1}(\Psi, x,\bm{W})\)^{\frac{x_{l}}{x_{l+1}}}\middle|\mathcal{A}^{\otimes n}_l\]=\widehat{\Xi}_{l}(\Psi, x,\bm{W})
}
So, if the equivalence in distribution holds up to $l+1$ it holds up to $l$, proving the induction hypothesis and then the proposition.
\end{proof}
We can finally prove the Theorem.
\begin{proof}[Proof of Theorem~\ref{th:The_Theorem}]
Fix a level of replica symmetry breaking $K\in \N$. Let $x:=(x_l)_{l\in [K+1]_0}$ and $q:=(q_l)_{l\in [K+1]_0}$ be the increasing sequences defined in~\eqref{eq:x_seq0} and~\eqref{eq:q_seq2}, respectively. Consider the space $\mathcal{M}^{[K+1]_0}$ from~\eqref{eq:product-space-M}, and define a Parisi parameter $\mu\in \Pr([0,1])^{\circ}$ as in~\eqref{discrete2}.

We divide the proof into two steps.

\medskip
\noindent\textbf{Step 1.} 
For any random $\zeta \in \mathcal{M}^{(0)}$, there exists $m \in B_1([-1,1])$ such that
\been{
\widehat{\mathcal{P}}_K(\zeta, x)\geq \inf_{m \in B_1([-1,1])}\mathcal{P}(m,\mu)
}

\noindent\textit{Proof of Step 1.} 
Let
  \been{
  \label{eq:mbar10}
 \bar{m}(\omega):=\left(m^{(0)}\(0,\omega\), \cdots, m^{(K+1)}\(q_K,\omega\)\right)\in \mathcal{M}^{[K+1]_0},
 }
 be the sequence defined as in~\eqref{eq:equivWh2}. We also define 
 \been{
 m(\omega):=m^{(K+1)}\(q_K,\omega\)\in [-1,1].
 }
 Note that $m\in B_1([-1,1])$. From the definition~\eqref{eq:psievmr}, we see that $\psi^{(*)} \circ \zeta^{\times 2c}(\bm{W})$ depends only on the last elements of the random sequences $\bm{W}=(W_1,\cdots W_n)$. So, comparing the definitions~\eqref{eq:psievmr} and~\eqref{eq:psic2}, we get
\been{
\widetilde{\Psi}(\bm{\omega})=\psi^{(*)} \circ \zeta^{\times 2c}[\bar{m}(\omega_1),\cdots,\bar{m}(\omega_n) ]=\psi^{(*)} \(m(\omega),\cdots ,m(\omega)\)=\psi^{(*)} \circ m^{\times 2c}(\bm{\omega}).
}
Thus, by Proposition~\ref{prop:equivphiphi}
\been{
\widehat{\Phi}(\psi^{(*)}_{\bm{J}} \circ \zeta^{\times 2x}, x)=\Phi(\psi^{(*)}_{\bm{J}} \circ m^{\times 2c}, \mu,0)
}
implying that $\widehat{\mathcal{P}}_K(\zeta, x)=\mathcal{P}(m,\mu)$. Ths implies
\been{
\widehat{\mathcal{P}}_K(\zeta, x)\geq \inf_{m \in B_1([-1,1])}\mathcal{P}(m,\mu)
}

\medskip
\noindent\textbf{Step 2.} 
For any random $m \in B_1([-1,1])$, there exists $\zeta \in \mathcal{M}^{(0)}$ such that
\been{
\mathcal{P}(m,\mu)\geq \inf_{\zeta \in\mathcal{M}^{(0)}}\mathcal{P}_K(\zeta, x).
}

\noindent\textit{Proof of Step 2.} 
Given $m\in B_1([0,1])$ we can construct a random sequence in $\mathcal{M}^{[K+1]_0}$ by taking
\been{
m^{(K+1)}(q_{K+1},\omega)=m(\omega),
}
and, for any $l\in [K]_0$, $m^{(l)}(q_l,\omega)$ is the unique element of $\mathcal{M}^{(l)}$ such that, for any open set $A\subseteq \mathcal{M}^{(l+1)}$
\been{
m^{(l)}(q_l,\omega)(A)=\E\[\bm{1}_{\{m^{(l+1)}(q_{l+1},\omega)\in A\}}\middle|\mathcal{F}_{q_l}\].
}
Let us denote the whole sequence by $\bar{m}(\omega)$ as in~\eqref{eq:mbar10}. Comparing the definition of  $\bar{m}(\omega)$ with the definition~\eqref{eq:random-sequence}, we see that, conditionally on $m^{(0)}(0,\omega)$, the law of $\bar{m}$ is $\mathbb{M}_{m^{(0)}(0,\omega)}$. So we can define a random sequence $W$ as in~\eqref{eq:random-sequence} such that
\been{
W\overset{d}{=}\bar{m}(\omega),\quad \textup{conditionally on }m^{(0)}(0,\omega).
}
Thus, following the proof of the previous step, we deduce that $\mathcal{P}(m,\mu)=\mathcal{P}_K\(m^{(0)}(0,\omega), x\).$ Thus
\been{
\mathcal{P}(m,\mu)\geq \inf_{\zeta \in\mathcal{M}^{(0)}}\mathcal{P}_K(\zeta, x).
}

\medskip
\noindent Combing the two steps, we deduce that, for any increasing sequence $x\in (x_l)_{l\in [K+1]_0}$ and any Parisi parameter $\mu\in \Pr([0,1])^{\circ}$ defined as in~\eqref{discrete2}
\been{
 \inf_{\zeta \in \mathcal{P}_{K+1}}\widehat{\mathcal{P}}_K(\zeta,x)= \inf_{m \in B_1([-1,1])}\mathcal{P}(m,\mu).
}
By Theorem~\ref{th:func_quant}, the operator $(\Psi,\mu)\mapsto \Phi(\Psi,\mu,0)$ is continuous in $\mu$. Thus $(m,\mu)\mapsto \mathcal{P}(m,\mu)$ is continuous in $\mu$, yielding to
\been{
\inf_{\substack{\zeta \in \mathcal{P}_{K+1}\\ \mu\in \Pr([0,1])}}\mathcal{P}(m,\mu)=\inf_{\substack{\zeta \in \mathcal{P}_{K+1}\\ \mu\in \Pr([0,1])^{\circ}}}\mathcal{P}(m,\mu)=\inf_{K\in \N}\inf_{\substack{\zeta \in \mathcal{P}_{K+1}\\0=x_0<x_1\leq \cdots x\leq x_{K+1}=1}}\widehat{\mathcal{P}}_K(\zeta,x).
}
The Franz-Leone upper bound~\eqref{eq:RSB_UB} completes the proof.
\end{proof}
\section*{Acknowledgments}
The author thanks Prof. Giuseppe Genovese and Prof. David Belius for helpful comments on earlier drafts.
Thanks are also due to Simone Franchini for valuable discussions and insights. 


\providecommand{\bysame}{\leavevmode\hbox to3em{\hrulefill}\thinspace}
\providecommand{\MR}{\relax\ifhmode\unskip\space\fi MR }
\providecommand{\MRhref}[2]{%
  \href{http://www.ams.org/mathscinet-getitem?mr=#1}{#2}
}
\providecommand{\href}[2]{#2}

\end{document}